\documentclass[11pt,a4paper]{amsart}
\usepackage{amsmath,amssymb, amsbsy}
\usepackage{subfigure}
\usepackage{graphpap,latexsym,epsf}
\usepackage{color,psfrag}
\usepackage[dvips]{graphicx}
\usepackage{enumerate}
\newcommand{\R}{{\mathbb R}}
\newcommand{\N}{{\mathbb N}}

\renewcommand{\H}{{\mathcal H}}

\renewcommand{\phi}{\varphi}

\renewcommand{\ge }{\geqslant}
\renewcommand{\geq }{\geqslant}
\renewcommand{\le }{\leqslant}
\renewcommand{\leq }{\leqslant}
\def\neweq#1{\begin{equation}\label{#1}}
\def\endeq{\end{equation}}
\def\eq#1{(\ref{#1})}

\newtheorem{theorem}{Theorem}
\newtheorem{proposition}[theorem]{Proposition}
\newtheorem{lemma}[theorem]{Lemma}

\newtheorem{definition}{Definition}

\begin{document}

\title[]{Loss of energy concentration \\ in nonlinear evolution beam equations}

\author{Maurizio GARRIONE - Filippo GAZZOLA}
\date{\today}


\begin{abstract}
Motivated by the oscillations that were seen at the Tacoma Narrows Bridge, we introduce the notion of solutions with a \emph{prevailing mode} for the nonlinear evolution beam equation
$$
u_{tt} + u_{xxxx} + f(u)= g(x, t)
$$
in bounded space-time intervals. We give a new definition of \emph{instability} for these particular solutions, based on the loss of energy concentration on their prevailing mode. We distinguish between two different forms of energy transfer, one \emph{physiological} (unavoidable and depending on the nonlinearity) and one due to the insurgence of instability. 
We then prove a theoretical result allowing to reduce the study of this kind of infinite-dimensional stability to that of a finite-dimensional approximation. With this background, we study the occurrence of instability for three different kinds of nonlinearities $f$ and for some forcing terms $g$, highlighting some of their structural properties and performing some numerical simulations.
\end{abstract}

\maketitle

\section{Introduction}

The linear beam equation
\neweq{cdbeam}
u_{tt}+u_{xxxx}+\gamma^2 u=0, \quad \mbox{for }(x,t)\in(0,\pi)\times(0,\infty), \quad \gamma \in \R,
\endeq
with Cauchy-Navier initial-boundary conditions
\begin{equation}\label{IBC}
\begin{array}{ll}
u(0,t)=u(\pi,t)=u_{xx}(0,t)=u_{xx}(\pi,t)=0\quad & \mbox{for }t\in(0,\infty)\\
u(x,0)=u_0(x)\, ,\quad u_t(x,0)=u_1(x)\quad & \mbox{for }x\in(0,\pi),
\end{array}
\end{equation}
can be solved by separating variables. In particular, if $u_0(x)$ and $u_1(x)$ are both proportional to $\sin(jx)$ for some positive integer $j$, then the solution of \eqref{cdbeam}-\eqref{IBC} has the form
$$
u(x, t)= A\sin(\alpha t + \psi) \sin(jx),
$$
with $\alpha=\sqrt{j^4+\gamma^2}>0$ and $A, \psi \in \R$ depending on the initial data.
We highlight this well-known property for the linear autonomous
problem \eqref{cdbeam} as follows:
\begin{center}
{\bf if the initial data are concentrated on a single mode (such as $\sin(jx)$)}\\
{\bf then the solution remains concentrated on the same mode for all $t>0$.}
\end{center}
This property characterizes linear problems (and some classes of nonlocal problems, see \cite{bbfg} and earlier works on the wave equation \cite{CazWeiP, CazWei}) but does not hold for nonlinear beam equations such as
\neweq{nonlinear}
u_{tt}+u_{xxxx}+f(u)=0\quad \mbox{for }(x,t)\in(0,\pi)\times(0,\infty)
\endeq
where
\neweq{flip}
f\mbox{ 
is locally Lipschitz-continuous in }\R, \mbox{ non-decreasing, and such that}\ f(0)=0.
\endeq
Notice that $f(u)=\gamma^2 u$ (as in \eq{cdbeam}) satisfies \eq{flip} which is a kind of ``coercivity'' assumption. Even if $f$ is nonlinear,
one can still solve \eq{nonlinear} by separating variables and seek solutions in the form
\neweq{general}
u(x,t)=\sum_{n=1}^\infty\phi_n(t)\sin(nx),
\endeq
where the unknowns are the coefficients $\phi_n$ that solve the infinite-dimensional ODE system
\begin{equation}\label{infinite}
\ddot{\phi}_n(t)+n^4\phi_n(t)+\frac{2}{\pi}\int_0^\pi f\Big(\sum_{m=1}^\infty\phi_m(t)\sin(mx)\Big)\sin(nx)\, dx=0\qquad(n=1,...,\infty).
\end{equation}
Nonlinear wave-type systems \cite{buitelaar} and beam equations \cite{AbrHorVak} admit many resonances, since the dynamical system itself is infinite dimensional, as \eq{infinite}. As we shall see,
these resonances are difficult to detect because the initial energy of the system immediately spreads on infinitely many modes. Indeed, two fairly different kinds of energy transfer are observable inside the system. On one hand, there is a ``physiological'' energy exchange between modes, which starts instantaneously and occurs at any energy level: it is governed by the particular nonlinearity $f$ in \eqref{infinite}, so that couples of modes may possibly interact. On the other hand, resonances may occur and become
visible thanks to a (possibly delayed) sudden and violent transfer of energy from some modes to other ones, \emph{only occurring above some energy threshold}. In Section \ref{remarks}, we will specify what we mean by ``non-physiological'' transfer of energy between modes, see Definition \ref{unstable}. To make this resonance phenomenon detectable in a sufficiently clear way, we will focus on the situation where \emph{only a single mode is initially large}, thus owning most of the energy which will possibly spread onto the other modes at later time. Actually, this is what is observed in real oscillating structures: for instance, from the report on the Tacoma Narrows Bridge (TNB) collapse \cite[p.20]{ammann}, we learn that in the months prior to the collapse
\begin{equation}\label{tacoma}
\mbox{``one principal mode of oscillation prevailed and the modes of oscillation frequently changed''.}
\end{equation}
In our analysis, we will mainly focus on an autonomous system, viewing the initial data concentrated on a single mode as the effect of former external forces. Actually, we ideally start our study when the oscillation of the structure is maintained in amplitude by a somehow perfect balance between the external input of energy (e.g., from the wind) and the internal dissipation. 
\par
It is well-known that in general structures some vibrations are more destructive than others: for the TNB, the ``bad'' oscillations were the torsional ones. In fact, each real structure becomes vulnerable whenever one of its components vibrates in one of these bad ways. Imagine that a beam modeled by equation \eqref{cdbeam} is part of a complex structure with several interacting components: in this case, the bad modes are the ones having a frequency badly tuned with the frequencies of the other components. The target is then to prevent too much energy to move to these dangerous modes.
\par
In the following, we will deal both with the nonlinear
unforced beam equation \eq{nonlinear} and with its forced counterpart
\neweq{nonlinearforced}
u_{tt}+u_{xxxx}+f(u)=g(x, t)\quad \mbox{for }(x,t)\in(0,\pi)\times(0,\infty),
\endeq
considered together with the initial-boundary conditions \eqref{IBC}.
We define the total energy associated with the problem under investigation as
\neweq{energy}
E(t)=\int_0^\pi\Big(\frac{u_t^2}{2}+\frac{u_{xx}^2}{2}+F(u)\Big)\, dx, \quad \mbox{with } F(s)=\int_0^s f(\sigma) \, d\sigma,
\endeq
taking into account the contribution of the kinetic, the bending and the potential energy, respectively.
In the unforced case such a quantity is constant in time, that is $E(t)\equiv E(0)$,
but even if $E(0)$ is concentrated on a single mode, $E(t)$ can suddenly move to other modes due to the coupling effect produced by the nonlinearity. We study this phenomenon in our numerical experiments by taking initial data $u(x,0)$ and $u_t(x,0)$ almost completely concentrated on a given Fourier component $\varphi_j$ (so that $j$ will be the ``prevailing'' mode), and observe if some instability - \textbf{meant as a particular loss of energy concentration with subsequent energy transfer} - arises in a time interval $[0, T]$, see Sections \ref{stability} and \ref{remarks} for a precise explanation.
The forced situation is more delicate; in order to identify unambiguously the way the energy transfer takes place, we deal with forcing terms of the form
\neweq{g}
g(x,t)=\alpha\sin(jx)\sin(\gamma t), \mbox{ for some } \gamma \neq 0, \; \alpha > 0,
\endeq
namely only acting on the $j$-th mode, which for this reason will be considered the prevailing one. This choice is motivated by the comments reported in \cite[p. 119-120]{ammann}, where the incidence of the wind on the deck of the TNB was observed to shed (periodic-in-time) vortices concentrated on one particular mode of oscillation. In the expression \eqref{g}, $\gamma$ is the frequency of the vortex acting on the $j$-th mode, whereas $\alpha$ is the
magnitude of the excitation; the case where $\sin(\gamma t)$ is replaced by $\cos(\gamma t)$ is completely similar.
Of course, the forced equation \eqref{nonlinearforced} is more difficult since there is no preservation of energy: therefore, the energy transfer is not straightforward to characterize and, in principle, one can expect its occurrence to depend on a suitable relationship between $\gamma$ and $j$. Also the role of the magnitude $\alpha$ may not be completely clear.
\par
We observe that the isolated case aims at highlighting the {\em structural properties} of the nonlinear beam, whereas the forced case
is related to its {\em aerodynamic properties}. Hence, we start investigating the role played by the structural properties in the energy transfer mechanism and
if the aerodynamic properties may lead to the same behaviour.
\par
In some recent papers \cite{arga,bfg1,bergaz}, the energy transfer between modes in some nonlinear bridge models has been emphasized.
In the sequel, we will discuss the three following nonlinearities:
$$
f(u)=u^+, \quad f(u)=u^3, \quad f(u)=(u^+)^3,
$$
all satisfying \eq{flip} and aiming
at modeling the restoring force due to the hangers+cables-system on a simplified suspension bridge whose deck is seen as a beam.
The choice of these nonlinearities will be discussed later, see also \cite{bookgaz} for a survey of results and for a review of the history of suspension bridges.
\par
The paper is organized as follows. In Section \ref{entrans}, we state well-posedness for \eqref{IBC}-\eqref{nonlinearforced} and an approximation result which represents the theoretical justification to our numerical experiments. In Section \ref{principale}, we introduce and comment our new notion of stability for solutions of \eqref{nonlinearforced} with a prevailing mode. In Section \ref{finitodim}, we explain how to proceed with the finite-dimensional approximation in order to highlight the loss of energy concentration, both for the isolated and the forced equation. This procedure is implemented for the three
considered nonlinearities in the subsequent sections.

\section{Existence, uniqueness, and approximation of the solutions}\label{entrans}

We first give a well-posedness result for the nonlinear forced equation \eqref{nonlinearforced}, with a smooth forcing term $g$.
We introduce the Hilbertian Sobolev spaces
$H^2_*(0,\pi):=H^2(0, \pi) \cap H^1_0(0, \pi)$ and
$$H^4_*(0,\pi):=\Big\{v\in H^4(0,\pi);\, v, v'' \in H^2_*(0,\pi)\Big\}\, ,$$
and we denote by $\H(0,\pi)$
the dual space of $H^2_*(0,\pi)$.
The regularity of a function $v$ that is expanded in Fourier series as
$$
v(x)=\sum_{n=1}^\infty c_n\sin(nx)
$$
depends on the behaviour of the sequence $\{c_n\}_{_n}$. Introducing the spaces
$$
\ell^2_m:=\left\{\{c_n\}_{_n}\in\R^\infty;\, \sum_{n=1}^\infty k^{2m}c_n^2<\infty\right\}, \quad m=0, 2, 4,
$$
so that $\ell^2_0=\ell^2$ (the usual space of squared summable real sequences), it turns out that 
$$
v\in L^2(0,\pi)\ \Longleftrightarrow\ \{c_n\}_{_n} \in\ell^2_0\, , \quad v\in H^2_*(0,\pi)\ \Longleftrightarrow\ \{c_n\}_{_n}\in\ell^2_2\, , \quad v\in H^4_*(0,\pi)\ \Longleftrightarrow\ \{c_n\}_{_n}\in\ell^2_4\,.
$$
We can now state an existence, uniqueness, and regularity result for \eq{nonlinearforced}.

\begin{theorem}\label{exist}
Assume that \eqref{flip} holds and that $u_0\in H^2_*(0,\pi)$, $u_1\in L^2(0,\pi)$, $g \in C^1(\R_+ \times [0, \pi])$. Then there exists a unique
$$u\in C^0\big(\R_+;H^2_*(0,\pi)\big),\quad\mbox{with}\quad u_t\in C^0\big(\R_+;L^2(0,\pi)\big)\quad\mbox{and}\quad
u_{tt}\in C^0\big(\R_+;\H(0,\pi)\big),$$
such that:\par
$(i)$ $u$ satisfies the initial conditions $u(x,0)=u_0(x)$ and $u_t(x,0)=u_1(x)$;\par
$(ii)$ $u$ satisfies \eqref{nonlinearforced} in the following weak sense:
$$
\langle u_{tt}(t),\psi\rangle+\int_0^\pi u_{xx}(t)\psi_{xx}+\int_0^\pi f\big(u(t)\big)\psi = \int_0^\pi g(t) \psi,
$$
for all $\psi\in H^2_*(0,\pi)$ and a.e.\ $t>0$.\par
If, moreover, $u_0\in H^4_*(0,\pi)$ and $u_1\in H^2_*(0,\pi)$, then the solution $u$ of \eqref{nonlinearforced} satisfies
$$u\in C^0\big(\R_+;H^4_*(0,\pi)\big),\quad\mbox{with}\quad u_t\in C^0\big(\R_+;H^2_*(0,\pi)\big)\quad\mbox{and}\quad u_{tt}\in C^0\big(\R_+;L^2(0,\pi)\big),$$
and is therefore a strong solution of \eqref{nonlinearforced}.
\end{theorem}
\begin{proof}
The proof can be obtained with a Galerkin procedure similar to the one in \cite[Theorem 4.1, p. 210]{Tem}, with minor changes due to the fact that $f$ only satisfies \eqref{flip}.
\end{proof}

Theorem \ref{exist} guarantees existence and uniqueness of the strong solution $u(x, t)$ of \eqref{nonlinearforced} but, to analyze the energy transfer between modes, we have to deal with \emph{approximate solutions} as follows.
As in \cite{Krol}, we use two different approximations, both depending on an integer parameter $N$ which measures
the degrees of freedom, namely the number of nontrivial Fourier coefficients.
\par
The first approximation requires the knowledge of the solution $u(x, t)$ of \eq{nonlinearforced} and consists in
projecting it onto the finite dimensional subspace of $H^2_*(0,\pi)$ spanned by $\{\sin(x),...,\sin(Nx)\}$. If we denote by $P_N$ the corresponding
projector onto this space and by $Q_N = I - P_N$ its $H^2_*$-orthogonal complement, we have
\begin{equation}\label{proiettori}
P_Nu(x,t)=\sum_{n=1}^N\phi_n(t)\sin(nx), \quad Q_N u(x, t)=\sum_{n=N+1}^{+\infty} \phi_n(t) \sin(nx),
\end{equation}
where $u(x, t)$ is the solution of \eq{nonlinearforced} and the $\phi_n$'s are as in \eq{general}. In the unforced case $g(x, t) \equiv 0$, one has energy conservation and therefore, since $F(u) \geq 0$ by \eqref{flip},
$$
\frac{(N+1)^4}{2} \int_0^\pi (Q_N u)^2 \leq \frac{1}{2} \int_0^\pi (Q_N u_{xx})^2 \leq \frac{1}{2} \int_0^\pi u_{xx}^2
\le\int_0^\pi\Big(\frac{u_t^2}{2}+\frac{u_{xx}^2}{2}+F(u)\Big)=E(t)=E(0),
$$
which readily implies that the $N$-th remainder $R_N$ of the series $\sum\varphi_n^2(t)$ is subject to the bound
\begin{equation}\label{resto}
R_N(t):=\frac{2}{\pi}\int_0^\pi|Q_Nu(x,t)|^2\, dx=\sum_{n>N}\varphi_n^2(t)\leq\frac{4 E(0)}{\pi(N+1)^4}, \quad \forall t \in \R_+.
\end{equation}
This shows that $P_Nu\to u$ in $L^\infty(\R_+;L^2(0,\pi))$ as $N \to \infty$.
\par
The second approximation is given by the
Galerkin procedure: we seek functions of the form
\neweq{tronca}
u^N(x, t)= \sum_{n=1}^N\phi_n^N (t)\sin(nx)
\endeq
that solve the following finite-dimensional ODE system, obtained by inserting $u^N(x, t)$ into \eqref{nonlinearforced} and then by projecting the resulting equation onto the space spanned by first $N$ modes:
\begin{equation}\label{finite}
\ddot{\phi}_n^N(t)+n^4\phi_n^N(t)+\frac{2}{\pi} \int_0^\pi f\Big(\sum_{m=1}^N\phi_m^N(t)\sin(mx)\Big)\sin(nx)\, dx=\frac{2}{\pi}\int_0^\pi g(x, t) \sin(nx) \, dx,
\end{equation}
with $n=1,...,N$. The following statement holds. 
\begin{theorem}\label{qualitativo}
Let $T > 0$ be fixed and let
$$
u(x, t) = \sum_{n=1}^{+\infty} \varphi_n(t) \sin (nx)
$$
be the strong solution of \eqref{IBC}-\eqref{nonlinearforced}. Let $\{\varphi_n^N\}_{_n}$ be the solution of the finite-dimensional system \eqref{finite} such that $\varphi_n^N(0)=\varphi_n(0)$, $\dot{\varphi}_n^N(0)=\dot{\varphi}_n(0)$. Then, for every $\epsilon > 0$ there exists $N^\epsilon$ such that, for every $N \geq N^\epsilon$, it holds
$$
\Vert \varphi_n - \varphi_n^{N} \Vert_{L^\infty(0, T)} \leq \epsilon,
$$
for every $n \in \{1, \ldots, N\}$.
\end{theorem}
Actually, this qualitative statement is somehow contained in the theory of Galerkin approximations: the sequence $\{u^N\}_{_N}$ converges to the solution $u$ of \eqref{IBC}-\eqref{nonlinearforced} provided by Theorem \ref{exist}. 
However, notice that here we focus on \emph{each} component $\varphi_n$, evaluating its difference with the corresponding finite-dimensional approximation $\varphi_n^N$. This will be an important point in our definition of stability.
It is also clear that this general statement cannot be optimal, nor it is directly applicable to concrete situations where finer estimates of the error are needed.
In fact, in Section \ref{dimostrazione} we will prove Theorem \ref{qualitativo} by providing a quantitative estimate of $\Vert \varphi_n - \varphi_n^N \Vert_{L^\infty(0, T)}$ (Theorem \ref{approximation}), which, as a byproduct, will equip us with an effective way of computing $N^\epsilon$; this will be the crucial point in order to deduce some general properties of the PDE
\eqref{nonlinearforced} through the numerical study of \eqref{finite}.

\section{Fundamental definitions and comments}\label{principale}

\subsection{Stability of solutions with a prevailing mode}\label{stability}

In this section, we characterize the energy transfer and instability phenomena for both equations \eqref{nonlinear} and \eqref{nonlinearforced}. We take initial data of the form
\begin{equation}\label{datiin}
u_0(x) = \sum_{n=1}^{+\infty} a_n \sin (nx), \quad u_1(x)=\sum_{n=1}^{+\infty} b_n \sin(nx),
\end{equation}
with $\{a_n\}_{_n} \in \ell^2_4$ and $\{b_n\}_{_n} \in \ell^2_2$,
and possibly a forcing term $g$ as in \eqref{g}. By Theorem \ref{exist}, these data uniquely determine the solution; therefore, from now on we will use them in order to characterize some particular solutions.   
\par
We first give two definitions regarding equation \eqref{nonlinearforced}. We are interested only in solutions which have a prevailing mode, i.e., characterized by data having the tendency to concentrate most of the dynamics on a sole mode, according to the following definition.
\begin{definition}\label{prevalente}
Let $0 < \eta < 1$. We say that a strong solution of \eqref{nonlinearforced} has the $j$-th mode $\eta$-\emph{prevailing} if $j$ is the only integer for which:
\begin{itemize}
\item[-] if $g(x, t) \equiv 0$, then
\begin{equation}\label{smallbig}
\sum_{n\neq j}(a_n^2+b_n^2)\leq \eta^4 (a_j^2+b_j^2)\, , \quad (a_n, b_n \mbox{ as in } \eqref{datiin});
\end{equation}
\item[-] if $g(x, t) \not\equiv 0$, then
\begin{equation}\label{forzato}
a_n^2+b_n^2 \leq \eta^4 (a_j^2+b_j^2)\, \mbox{ for every } n \neq j \quad \textrm{ and } \quad
 g(x, t)= \alpha \sin(jx) \sin(\gamma t),
\end{equation}
for some $\gamma \neq 0$ and $\alpha > 0$ (the case $\alpha<0$ being similar). 
\end{itemize}
For this solution, all the other modes $k \neq j$ are called \emph{residual}.
\end{definition}
Condition \eqref{smallbig} states that both the potential and the kinetic energy are initially almost completely concentrated on the $j$-th mode, while \eqref{forzato} also states that the $j$-th mode is the only forced one (slightly relaxing assumption \eqref{smallbig}).

Obviously, not all the solutions of \eqref{nonlinearforced} have an $\eta$-prevailing mode but, motivated by the observation \eqref{tacoma}, our stability analysis will be focused on this situation. 
\par
We now formalize our notion of \emph{instability} before a certain time instant $T$.
\begin{definition}\label{unstable}
Let $T_W > 0$. We say that a solution of \eqref{nonlinearforced}, in the form \eqref{general} and having $\eta$-prevailing mode $j$, is \emph{unstable} before time $T > 2T_W$ if there exist a residual mode $k$ and a time instant $\tau$ with $2T_W < \tau < T$  such that
\neweq{grande}
(i)  \ \ \Vert \varphi_k \Vert_{L^\infty(0, \tau)} \geq \eta \Vert \varphi_j \Vert_{L^\infty(0, \tau)}
\qquad \mbox{ and } \qquad
(ii) \ \ \frac{\Vert \varphi_k \Vert_{L^\infty(0, \tau)}}{\Vert \varphi_k \Vert_{L^\infty(0, \tau/2)}} \geq \frac{\alpha+1}{\eta},
\endeq
where $\eta$ is as in Definition \ref{prevalente} and $\alpha > 0$ is as in \eqref{forzato} ($\alpha = 0$ if $g \equiv 0$).
We say that $u$ is \emph{stable} until time $T$ if it is not unstable.
\end{definition}

Definition \ref{unstable} is the core of the paper and, for this reason, we now comment it in detail.

\subsection{Physiological transfer, Wagner effect and instability}\label{remarks}

Definition \ref{unstable} detects an energy transfer by comparing the behaviour of a residual mode $k$ both with respect to the $\eta$-prevailing one $j$ and with respect to itself in earlier time. In fact, condition $(i)$ in \eqref{grande} roughly says that the potential energy of the $k$-th mode has subverted the initial inequality \eqref{smallbig} by an order of magnitude along the time interval $[0, \tau]$, while condition $(ii)$ basically expresses the fact that the $k$-th mode has a ``Floquet multiplier'' equal to $(\alpha+1)/\eta$. In these cases, the system is likely to display destructive oscillations in the future when the residual mode $k$ is badly tuned with the structure (see \eqref{tacoma} and the subsequent comment). Here, by destructive we mean exponentially-like growing: of course, for $g \equiv 0$ the energy is preserved in the system and such oscillations are prevented to blow up, but if they become too large they may lead to failures of the structure, as for the TNB. 
As anticipated in the Introduction, we will mainly focus our attention on the autonomous case, viewing the initial data as the consequence of earlier external forces concentrated on a particular mode. Indeed, while in absence of wind or external loads the deck of a bridge remains still, the presence of the wind generates a forcing lift making the deck oscillating on a prevailing mode; this increases the internal energy of the structure, generating wide
oscillations. Our analysis starts when the oscillation is maintained in amplitude by an equilibrium between the input of energy from the wind and structural dissipation.
\par
One of the reasons why we choose to use the $L^\infty$-norm of $\varphi_k$ in \eqref{grande} instead of the total energy of the $k$-th mode, given by computing \eqref{energy} on $u(x, t)= \varphi_k(t) \sin (kx)$, is due to the fact that the Fourier components which remain too small (in absolute value) usually play a secondary role in applications. Among them, we refer in particular to the small oscillations of high modes, which may have large energy even for small $L^\infty$-norms.
Condition \eqref{grande}-$(i)$ indeed requires not only the residual mode to grow with respect to itself, but also to reach a relevant amplitude compared to that of the $\eta$-prevailing mode.
\par
Let us also observe that it may happen that the (potential$+$kinetic) energy splits and distributes itself from the $\eta$-prevailing mode onto \emph{more} residual modes which, in principle, could grow substantially, but without fulfilling Definition \ref{unstable}. For instance, two residual modes may satisfy the inequalities in \eqref{grande} but with smaller factors: in this case, it would be natural to consider the $\eta$-prevailing mode as unstable, although this is not detected by Definition \ref{unstable}. To overcome this ambiguity, one could replace \eqref{grande} with
$$
 \ \ \Big\Vert \sum_{n \neq j} (\varphi_n^2 + \dot{\varphi}_n^2) \Big\Vert_{L^\infty(0, \tau)} \geq \eta^2 \Vert \varphi_j^2 + \dot{\varphi}_j^2 \Vert_{L^\infty(0, \tau)}
\; \; \mbox{ and } \;
\ \ \frac{\left \Vert \sum_{n \neq j} (\varphi_n^2 + \dot{\varphi}_n^2) \right\Vert_{L^\infty(0, \tau)}}{\left \Vert \sum_{n \neq j} (\varphi_n^2 + \dot{\varphi}_n^2) \right\Vert_{L^\infty(0, \tau/2)}} \geq \left(\frac{\alpha+1}{\eta}\right)^2;
$$
however, as we have seen in our experiments, if more residual modes ``almost'' satisfy \eqref{grande}, then by slightly increasing the initial amplitude of the $\eta$-prevailing mode at least one of them fully satisfies \eqref{grande}.
The only difference in considering also the kinetic part of the energy is the possibility of obtaining slightly different thresholds of instability, but this does not appear significant in the applications. Since checking a condition which involves the $L^\infty$-norm of the Fourier components is in general much simpler, we thus prefer to follow Definition \ref{unstable}. For further comments, we refer the reader to the discussions in Section \ref{casecub}; let us also mention that all these situations may be observed more neatly in a finite-dimensional Hamiltonian system, see \cite{bgz}.
\par
Also the bound $T > 2T_W$ deserves a comment. The interval $[0, T_W]$ embodies a transient phase corresponding to the so-called \emph{Wagner effect} \cite{Wag}, consisting in a time delay in the appearance of the response to a sudden change of the action of an external input in a forced system.
This is crucial to distinguish between \emph{physiological} losses of energy concentration and losses due to \emph{instability}: the former occur at any energy level and in the interval of time $[0, T_W]$, whereas the latter may occur suddenly but only after the Wagner time $T_W$. In other words: in absence of instability, $T_W$ is the needed time for the solution to settle to a kind of ``equilibrium'' in which the amplitude of each mode does not exhibit further significant variations. For these reasons, our kind of instability may be detected only if the behaviour of the solution is studied until a time $T$ at least equal to twice the Wagner time (namely, $T > 2T_W$). Overall, Definition \ref{unstable} \emph{excludes the Wagner effect and spots a real loss of energy concentration from an unstable $\eta$-prevailing mode towards a residual one}.
\par
Finally, from the point of view of the applications, the interest for instability phenomena is usually focused on a \emph{bounded} time interval. An engineer is interested in testing whether a structure is able to withstand oscillations for a certain period, for example one day. Thus, local-in-time stability responses are satisfactory from an applied point of view, even if the instability may theoretically appear after an arbitrarily large time.  

\section{Finite-dimensional approximation} \label{finitodim}

We show here how to use the previous definitions in order to detect instability for equation \eqref{nonlinearforced}, reducing its study to the analysis of a finite-dimensional system of ODEs. Also based on our numerical simulations, a reasonable choice for the parameters appearing in Definitions \ref{prevalente} and \ref{unstable} is given by
\begin{equation}\label{lascelta}
\eta=\left\{
\begin{array}{ll}
0.1 & \textrm{ in  \eqref{smallbig}} \\
0.999 & \textrm{ in  \eqref{forzato}}, 
\end{array}
\right.
\qquad 
T_W= 1; 
\end{equation}
henceforth, we will fix these values in all our numerical simulations. In fact, 
we also tested smaller values of $\eta$, such as $0.01$, and different values of $T_W$, but the picture remained qualitatively the same. Since $\eta$ is fixed as in \eqref{lascelta}, from now on we will simply speak about prevailing modes. 
\par
We observe that the PDE \eqref{nonlinearforced}, where we assume $g(x, t)$ of the form \eqref{g}, is equivalent to the infinite-dimensional system
\begin{equation}\label{infiniteforced}
\ddot{\phi}_n(t)+n^4\phi_n(t)+\frac{2}{\pi}\int_0^\pi f\Big(\sum_{m=1}^\infty\phi_m(t)\sin(mx)\Big)\sin(nx)\, dx= \delta_{j n}\, \alpha \sin (\gamma t)\qquad(n=1,...,\infty),
\end{equation}
while \eqref{finite} reads
\begin{equation}\label{finiteg}
\ddot{\phi}_n^N(t)+n^4\phi_n^N(t)+\frac{2}{\pi} \int_0^\pi f\Big(\sum_{m=1}^N\phi_m^N(t)\sin(mx)\Big)\sin(nx)\, dx= \delta_{j n} \, \alpha \sin(\gamma t) \qquad(n=1,...,N),
\end{equation}
where $\delta_{j n}$ is the usual Kronecker delta symbol.
Our next objective is to discuss system \eqref{infiniteforced} through the study of its finite-dimensional approximation \eqref{finiteg}. We observe that the extension of the notion of prevailing mode for system \eqref{finiteg} is straightforward; in particular, the Galerkin approximation of a solution with prevailing mode $j$ has the same prevailing mode when $N \geq j$. 
With the next definition, we formalize our notion of stability for the finite-dimensional system \eqref{finiteg}, by choosing suitable factors that take into account the choice made in \eqref{lascelta}.

\begin{definition}\label{finitestability}
We say that a solution $\{\varphi_n^N\}_{_n}$ of system \eqref{finiteg} with prevailing mode $j$ is:
\begin{itemize}
\item[-] \emph{unstable before time $T > 2$} if there exist a time instant $\tau$ with $2 < \tau < T$ and an integer $k \neq j$, $k \in \{1, \ldots, N\}$ such that
\neweq{finitegrande}
\Vert \varphi_k^N \Vert_{L^\infty(0, \tau)} \geq 0.11 \Vert \varphi_j^N \Vert_{L^\infty(0, \tau)}
\qquad \mbox{ and } \qquad
\frac{\Vert \varphi_k^N \Vert_{L^\infty(0, \tau)}}{\Vert \varphi_k^N \Vert_{L^\infty(0, \tau/2)}} \geq 11(\alpha+1);
\endeq
\item[-] \emph{stable until time $T > 2$} if, for every $k \neq j$, $k \in \{1, \ldots, N\}$ and every $\tau$ with $2 < \tau < T$,
\neweq{finitestable}
\mbox{either} \qquad \Vert \varphi_k^N \Vert_{L^\infty(0, \tau)} \leq 0.09 \Vert \varphi_j^N \Vert_{L^\infty(0, \tau)} \qquad \mbox{or} \qquad
\frac{\Vert \varphi_k^N \Vert_{L^\infty(0, \tau)}}{\Vert \varphi_k^N \Vert_{L^\infty(0, \tau/2)}} \leq 9(\alpha+1).
\endeq
\end{itemize}
\end{definition}
We underline that Definition \ref{finitestability} is not exhaustive since \eqref{finitegrande} and \eqref{finitestable} are not complementary: there are some cases where the definition does not allow to establish whether the solution is stable or unstable. \par
We now give a finite-dimensional sufficient criterion to check stability for system \eqref{infiniteforced}.
\begin{theorem}\label{stable}
There exists a number $\overline{N}$ (depending only on $E(0)$, $g$ and $T$) such that, for every $N \geq \overline{N}$, a  solution
$$
u(x, t) = \sum_{n=1}^{+\infty} \varphi_n(t) \sin (nx)
$$
of \eqref{nonlinearforced} with prevailing mode $j \in \{1, \ldots, N\}$ is unstable before (resp., stable until) time $T > 2$ provided that the corresponding solution $\{\varphi_n^N\}_{_n}$ of \eqref{finiteg}, with $\varphi_n^N(0)=\varphi_n(0)$, $\dot{\varphi}_n^N(0)=\dot{\varphi}_n(0)$ is unstable (resp., stable), according to Definition~\ref{finitestability}.
\end{theorem}
\begin{proof}
Assume that $\{\varphi_n^N\}_{_n}$ (with prevailing mode $j$) is unstable for some sufficiently large integer $N$, so that \eqref{finitegrande} holds for some $k$. On one hand, Theorem \ref{qualitativo} states that the components $\varphi_n^N$, $n=1, \ldots, N$, are arbitrarily near the corresponding components $\varphi_n$ in $L^\infty(0, \tau)$, up to choosing $N$ sufficiently large. On the other hand, enlarging $N$ if necessary, we can assume that the $N$-th remainder $\sum_{n > N} \varphi_n(t)^2$ is arbitrarily small (compare with \eqref{resto} in the autonomous case).
It easily follows that conditions $(i)$-$(ii)$ in \eqref{grande} are satisfied by $\varphi_k$ and the solution of \eqref{nonlinearforced} is unstable. The proof for stability is analogous.
\end{proof}

Theorem \ref{stable} does not come unexpected, in view of Theorem \ref{qualitativo} above; nonetheless, we recall again that the quantitative proofs provided in Section \ref{dimostrazione} give a constructive way to find $\overline{N}$.
\par
Let us explain how we use Theorem \ref{stable} in order to study the stability of \eqref{nonlinearforced}. We first analyze the unforced case \eqref{nonlinear}; we fix the prevailing mode $j$ and observe that here Definition \ref{prevalente} only holds with condition \eqref{smallbig}. We choose a finite time $T > 2T_W = 2$ and we plot the solution of system \eqref{finiteg} for various choices of the number $N$ of modes and several different initial conditions satisfying \eqref{smallbig}. We verify if, for some $N$, there exists $\tau$ with $2 < \tau < T$ such that \eqref{finitegrande} holds for some integer $k \leq N$. If this is the case, we infer instability also for system \eqref{infinite} with the same initial data on the first $N$ components (thanks to Theorem \ref{stable}), otherwise we try to verify if the conditions for stability are fulfilled. If none of these alternatives holds, we increase $N$ or we modify the time $T$ and we repeat the experiment.
\par
We then shift to the forced case, by considering small initial data on all the modes and a forcing term $g(x, t) \neq 0$ acting on a single mode as in \eqref{forzato}. Here, Definition \ref{prevalente} holds with condition \eqref{forzato}, and we aim at seeing if there exist amplitudes $\alpha$ and/or frequencies $\gamma$ which trigger instability. Again, the response of the experiments is validated by Theorem \ref{stable} in its quantitative version (Theorem \ref{approximation} below).
\par
All the numerical experiments have been performed by integrating the considered system of ODE with the program Wolfram Mathematica 10.3$^\copyright$, with a standard machine precision.

\section{The case with the positive part}\label{casepos}

A first possible choice for the nonlinearity $f$ in \eqref{nonlinearforced} is due to McKenna-Walter \cite{McKennaWalter} and reads
$$
f(u)=\mu\, u^+,
$$
where $u^+=\max\{0,u\}$ and $\mu>0$ denotes the Hooke constant of elasticity of steel (hangers). Only the positive part is taken into account due to possible slackening, see
\cite[V-12]{ammann}: the hangers behave as linear springs if extended (when $u>0$) and give no contribution if they lose tension (when $u\le0$). The corresponding equation has been studied in several papers, e.g., \cite{drabek1,drabek3,LazMcK}.

\subsection{Structural properties of the equation}

We consider here the following problem:
\neweq{cdbeampositive}
\left\{\begin{array}{ll}
u_{tt}+u_{xxxx}+\mu u^+=0\quad & \mbox{for }(x,t)\in(0,\pi)\times(0,\infty)\\
u(0,t)=u(\pi,t)=u_{xx}(0,t)=u_{xx}(\pi,t)=0\quad & \mbox{for }t\in(0,\infty)\\
u(x,0)=u_0(x)\, ,\quad u_t(x,0)=u_1(x)\quad & \mbox{for }x\in(0,\pi).
\end{array}\right.\endeq
Notice that, although \eq{cdbeampositive} is not linear, it has the following ``half-linear'' properties:

\begin{proposition}\label{homogeneity}
If $u$ is the solution of \eqref{cdbeampositive} with initial data $u_0$ and $u_1$, then for all $\lambda>0$ the function $\lambda u$ is the
solution of \eqref{cdbeampositive} with initial data $\lambda u_0$ and $\lambda u_1$. If $u$ and $v$ are two solutions of \eqref{cdbeampositive}
(with possibly different initial conditions) having the same sign for all $(x,t)\in(0,\pi)\times [0,\infty)$, then $u+v$ is a solution
of \eqref{cdbeampositive} with initial data $2u_0$ and $2u_1$.
\end{proposition}
This almost linear behaviour is confirmed by the
following fact which shows that \eq{cdbeampositive} is reluctant to the energy transfer between certain modes.

\begin{theorem}\label{evenodd}
Assume that
\neweq{datiiniziali}
u_0(x)=\sum_{n=0}^\infty a_n\sin\big((2n+1)x\big)\ ,\quad u_1(x)=\sum_{n=0}^\infty b_n\sin\big((2n+1)x\big)\, ,
\endeq
with $\{a_n\}_{_n} \in \ell^2_4$ and $\{b_n\}_{_n} \in \ell^2_2$.
Then, the solution of \eqref{cdbeampositive} is given by
\neweq{onlyodd}
u(x,t)=\sum_{n=0}^\infty \phi_{2n+1}(t)\sin\big((2n+1)x\big)\, ,
\endeq
for suitable functions $\varphi_{2n+1} \in C^2(\R_+)$ ($n=0, 1, \ldots$).
Moreover, if
$$u_0(x)=a\sin(x)\ \quad \textrm{ and } \quad u_1(x)=b\sin(x)$$
for some $a,b\in\R$, then there exists $\varphi_1\in C^2(\R_+)$ such that the solution of \eqref{cdbeampositive} has the form
\begin{equation}\label{onlyfirst}
u(x,t)=\phi_1(t)\sin(x)\qquad\forall(x,t)\in(0,\pi)\times\R_+\, .
\end{equation}
\end{theorem}
\begin{proof}
We first recall a {\em multiple-angle formula}, see e.g.\ formula (11) in \cite{weistein}: for any positive integer $l$, it holds
\neweq{formula}
\sin(lx)=\sin(x)\sum_{i=0}^{\mathcal{I}((l-1)/2)}(-1)^i\left(\begin{array}{c}
l-i-1\\
i
\end{array}\right)2^{l-2i-1}\cos^{l-2i-1}(x),
\endeq
where $\mathcal{I}(\cdot)$ denotes the integer part.\par
We consider the solution of \eq{cdbeampositive}, written in the form \eq{general}: by \eq{formula}, we have
$$
u(x,t)=\sin(x)\sum_{n=1}^\infty\phi_n(t)P_{n-1}\big(\cos(x)\big),
$$
where $P_{n-1}$ is a polynomial of degree $n-1$ with respect to $\cos(x)$ (for instance, $P_0(x) \equiv 1$, $P_1(x)=2\cos (x)$, $P_2(x)=4\cos^2(x)-1$). Moreover, $P_{n}$ contains only even (resp.\ odd) powers of $\cos(x)$
if $n$ is even (resp.\ odd), so that
\neweq{rule}
\begin{array}{c}
P_{n}\big(\cos(x)\big)\mbox{ is symmetric with respect to $x=\tfrac\pi2$ if $n$ is even,}\\
P_{n}\big(\cos(x)\big)\mbox{ is skew-symmetric with respect to $x=\tfrac\pi2$ if $n$ is odd.}
\end{array}
\endeq
Recalling that $x \in [0, \pi]$, it follows that
$$
u^+(x,t)=\sin(x)\Big(\sum_{n=1}^\infty\phi_n(t)P_{n-1}\big(\cos(x)\big)\Big)^+
$$
and, for every $n$, $\phi_n$ satisfies
\begin{equation}\label{esplicito}
\ddot{\phi}_n(t)+n^4\phi_n(t)+\frac{2\mu}{\pi}\int_0^\pi\left(\sum_{m=1}^\infty\phi_m(t)P_{m-1}\big(\cos(x)\big)\right)^+\sin(x)\sin(nx)\, dx=0\, .
\end{equation}
By the uniqueness statement in Theorem \ref{exist}, it suffices to prove that, if $u_0$ and $u_1$ are as in \eq{datiiniziali}, then problem
\eq{cdbeampositive} admits a solution in the form \eq{onlyodd}, that is, $\varphi_{2n} \equiv 0$ for every $n$. We thus assume that there exists $\{\varphi_{2n+1}\}_{_n}$ (Fourier components of $u(x, t)$ as in \eq{onlyodd}) solving \eqref{esplicito} and
we notice that, as a consequence, the thesis is equivalent to showing that, for every even integer $2n$, it holds
$$
\int_0^\pi \left(\sum_{m=1}^\infty\phi_m(t)P_{m-1}\big(\cos(x)\big)\right)^+\sin(x)\sin(2nx)\, dx=0\,,
$$
namely (since we are assuming $\varphi_{2n} \equiv 0$ for every $n$)
$$
\int_0^\pi \left(\sum_{m=0}^\infty\phi_{2m+1}(t)P_{2m}\big(\cos(x)\big)\right)^+\sin(x)\sin(2nx)\, dx=0\,.
$$
However, the function $\sin(x)\sin(2nx)$ is
skew-symmetric with respect to $x=\tfrac\pi2$, while the positive part under the integral sign is symmetric with respect to $x=\tfrac\pi2$, in view
of \eq{rule}. This means that the last expression is indeed equal to $0$, and the thesis is proved.\par
The second part of the theorem can be proved analogously. Indeed, let $u(x, t)$ be a solution of the form \eqref{onlyfirst};
in view of \eqref{esplicito}, the thesis is now equivalent to showing that, for every $n \neq 1$, it holds
$$
\frac{2\mu \varphi_1^+(t)
}{\pi}\int_0^\pi\sin(x)\sin(nx)\, dx=0\, ,
$$
which is true in view of the orthogonality of $\sin(x)$ and $\sin(nx)$ on $(0, \pi)$.
The solution is then found by solving explicitly the Cauchy problem
$$
\ddot{\varphi}_1 + \varphi_1 + \mu \varphi_1^+ = 0\ ,\qquad \varphi_1(0)=a\ ,\qquad\dot{\varphi}_1(0)=b\ ,
$$
which admits a unique global $C^2$-solution.\end{proof}

Theorem \ref{evenodd} states that there is no transfer of energy from odd modes to even modes: if the initial data do not contain even modes then also
the solution does not, for all $t>0$. Moreover, in the particular case where the initial data only contain the first mode, also the solution
does, for all $t>0$. In fact, in this latter case, the solution may be found explicitly. For instance, if $\mu = 3$ and
\neweq{kinetic+}
u_0(x)=0\, ,\quad u_1(x)=\sin(x),
\endeq
then the solution of \eq{cdbeampositive}-\eq{kinetic+} is given by
$$
u(x,t)=\phi_1(t)\sin(x)\qquad\mbox{with }\phi_1(t)=\left\{\begin{array}{ll}\frac12 \sin(2t)\ & \mbox{if }0\le t\le\tfrac\pi2\\
\sin(\tfrac\pi2-t)\ & \mbox{if }\tfrac\pi2 \le t\le \tfrac{3\pi}{2}\\
\tfrac{3\pi}{2}\, \mbox{-\,periodic}\end{array}\right.\qquad\quad
\begin{array}{l}
\includegraphics[height=16mm, width=30mm]{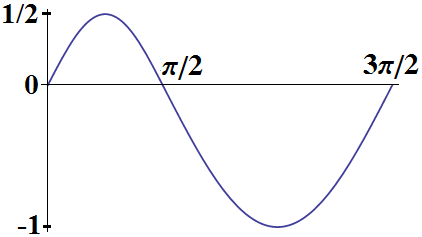}
\end{array}
$$
so that $u$ is $C^2$ in time; note that the negative part of $u$ is larger than the positive part.\par
On the other hand, the next result states that the first mode is ``attractive''.

\begin{theorem}\label{highermodes}
Assume that
\neweq{nontrivial}
u(x,t)=\sum_{n=1}^\infty\phi_n(t)\sin(nx)\not\equiv0
\endeq
is a solution of \eqref{cdbeampositive}. Then $\phi_1(t)\not\equiv0$.
\end{theorem}

\begin{proof}
Consider the solution of \eq{cdbeampositive} written in the form \eq{nontrivial}. If we multiply the equation in \eq{cdbeampositive} by $\sin(x)$ and we integrate over $(0,\pi)$, we obtain
$$
\ddot{\phi}_1(t)+\phi_1(t)+\frac{2\mu}{\pi}\int_0^\pi\left(\sum_{m=1}^\infty\phi_m(t)\sin(mx)\right)^+\sin(x)\, dx=0\, .
$$
By contradiction, if $\phi_1\equiv0$, then the previous equation yields
$$
\int_0^\pi\left(\sum_{m=2}^\infty\phi_m(t)\sin(mx)\right)^+\sin(x)\, dx=0\qquad\forall t\in\R_+\, .
$$
In turn, since $\sin(x)>0$ in $(0,\pi)$, this implies
$$\left(\sum_{m=2}^\infty\phi_m(t)\sin(mx)\right)^+=0\qquad\forall(x,t)\in(0,\pi)\times\R_+$$
and hence
$$
\sum_{m=2}^\infty\phi_m(t)\sin(mx)\le0\qquad\forall(x,t)\in(0,\pi)\times\R_+\, .
$$
Since by orthogonality it is
$$
\int_0^\pi\Big(\sum_{m=2}^\infty\phi_m(t)\sin(mx)\Big)\sin(x)\, dx=0\qquad\forall t\in\R_+\, ,
$$
this proves that
$$
\sum_{m=2}^\infty\phi_m(t)\sin(mx)=0\qquad\forall(x,t)\in(0,\pi)\times\R_+
$$
and contradicts $u\not\equiv0$.
\end{proof}

Theorem \ref{highermodes} states that, even if the initial data have zero component on the first mode, the solution of \eqref{cdbeampositive} always has a nonzero
component on the first mode. This one-way transfer of energy is clearly due to the particular nonlinearity $f(u)=\mu u^+$, since the
first mode is the only one having fixed sign. More generally, there is a one-way transfer of energy also between even and odd modes. As a simple example we prove the following.

\begin{theorem}\label{secondmode}
All the odd modes of the solution of the nonlinear problem \eqref{cdbeampositive} with initial conditions
\neweq{potential+}
u_0(x)=\sin(2x)\, ,\quad u_1(x)=0
\endeq
have a nontrivial Fourier coefficient $\varphi_n(t) \not\equiv 0$ for $t>0$.
\end{theorem}

\begin{proof}
For \eq{cdbeampositive}, system \eqref{infinite} yields the following equation for the $n$-th Fourier coefficient:
\neweq{thirdattempt}
\ddot{\phi}_n(t)+n^4\phi_n(t)+\frac{2\mu}{\pi}\int_0^\pi\left(\sum_{m=1}^\infty\phi_m(t)\sin(mx)\right)^+\sin(nx)\, dx=0\, .
\endeq
The positivity of the solution $u(x,t)$ obviously depends on $t$:
$$
\forall t\ge0\quad\exists I_t\subset[0,\pi]\qquad\mbox{s.t.}\qquad u(x,t)>0\Longleftrightarrow x\in I_t\, .
$$
Defining
$$
\alpha_{n,m}(t):=\int_{I_t}\sin(nx)\sin(mx)\, dx\qquad(t\ge0),
$$
equation \eq{thirdattempt} becomes
\neweq{fourthattempt}
\ddot{\phi}_n(t)+n^4\phi_n(t)+\frac{2\mu}{\pi}\sum_{m=1}^\infty\alpha_{n,m}(t)\phi_m(t)=0\, .
\endeq
Notice that $I_0=(0,\tfrac\pi2)$ and that
\neweq{alpha}
\alpha_{n,2}(0)=\left\{\begin{array}{ll}
\frac\pi4\quad & \mbox{if }n=2\\
0\quad & \mbox{if }n\mbox{ is even and }n\neq2\\
\tfrac{2}{4-n^2} \quad & \mbox{if }n \equiv 1\, \mbox{mod }4\\
\tfrac{2}{n^2-4} \quad & \mbox{if }n \equiv 3\, \mbox{mod }4\, .
\end{array}\right.\endeq
In view of \eq{potential+} we also have
\neweq{phizero}
\phi_2(0)=1\qquad\mbox{and}\qquad\phi_n(0)=0\quad\forall n\neq2\, ,\qquad\dot{\phi}_n(0)=0\quad\forall n\in\N\, ;
\endeq
therefore, \eq{fourthattempt} at $t=0$ becomes
$$
\ddot{\phi}_n(0)+n^4\phi_n(0)+\frac{2\mu}{\pi}\alpha_{n,2}(0)=0\qquad(n\in\N)\, .
$$
Combined with \eq{alpha} and \eq{phizero}, this readily yields
$$
\begin{array}{c}
\ddot{\phi}_2(0)=-\frac{32+\mu}{2}\, ,\qquad\ddot{\phi}_n(0)=0\mbox{ if }n\mbox{ is even and }n\neq2\, ,\\
\ddot{\phi}_n(0)=\tfrac{4\mu}{(n^2-4)\pi}\mbox{ if }n \equiv 1\, \mbox{mod }4\, ,\qquad
\ddot{\phi}_n(0)=\tfrac{4\mu}{(4-n^2)\pi}\mbox{ if }n \equiv 3\, \mbox{mod }4\, .
\end{array}
$$
Consequently, as soon as $t>0$ all the odd modes have nontrivial coefficients, some of them starting positive and some others starting negative.
\end{proof}

The forced situation is more delicate and deserves some comments. Indeed, since the problem is asymptotically linear, it is likely that some resonance phenomena appear for suitable frequencies of the forcing term which may interact with the characteristic frequency of the system, so as to indefinitely amplify the oscillations of the solution. For a single scalar asymmetric equation like
$$
\ddot{y}(t) + \alpha y^+(t) - \beta y^-(t) = e(t), 
$$
where $y^-=\max\{-y, 0\}$ and $\alpha, \beta$ are positive numbers, to recognize the occurrence of this phenomenon it is natural to analyze the nodal properties of the function
$$
\Phi_{\alpha, \beta, e}(\theta):= \int_0^{T} e(t)
\, \psi_{\alpha, \beta}(t+\theta) \, dt,
$$
where $\psi_{\alpha, \beta}$ is the nontrivial solution of the homogeneous asymmetric equation $\ddot{y} + \alpha y^+ - \beta y^- = 0$ fulfilling the initial conditions $y(0)=0$, $\dot{y}(0)=1$. For instance, if there exists a positive integer $M$ such that $(\alpha, \beta)$ lies on the $M$-th Fu\v{c}ik curve (namely, $\pi/\sqrt{\alpha} + \pi/\sqrt{\beta} = T/M$), there are forcing terms for which all the solutions are unbounded (for instance, $e(t)=\cos (2\pi M/T) t$, see \cite{Fab00}).
For further details about resonance for asymmetric oscillators, we mention for example the papers \cite{AloOrt98, Dan76, Fab00, FabFon98, Maw07, Ort02}.
\par
Taking into account the second part of the statement of Theorem \ref{evenodd}, such considerations can be extended to our system whenever the initial datum and the forcing term are completely concentrated on the first mode. As an example, we give the following statement, inspired by \cite{Ort02}.

\begin{theorem}\label{forz}
There exists a periodic function $e(t)$ of class $C^\infty$, having period equal to
\begin{equation}\label{periodoill}
\bar{T}=\frac{\pi}{\sqrt{\mu+1}} + \pi,
\end{equation}
such that, for any $0 < \vert \epsilon \vert \leq 1$, the solution
$$
u(x, t)= \sum_{n=1}^\infty \varphi_n(t) \sin (nx)
$$
of the problem
$$
\left\{
\begin{array}{ll}
u_{tt}+u_{xxxx}+\mu u^+=\epsilon e(t) \sin (x) \quad & \mbox{for }(x,t)\in(0,\pi)\times(0,\infty)\\
u(0,t)=u(\pi,t)=u_{xx}(0,t)=u_{xx}(\pi,t)=0\quad & \mbox{for }t\in(0,\infty)\\
u(x,0)=\sin(x)\, ,\quad u_t(x,0) = 0\quad & \mbox{for }x\in(0,\pi), \vspace{0.1cm}
\end{array}
\right.
$$
is such that $\varphi_1(t)$ is unbounded and $\varphi_n(t) \equiv 0$ for every integer $n > 1$.
\end{theorem}

\begin{proof}
The statement of Theorem \ref{forz} follows by applying directly \cite[Theorem 3]{Ort02}. Actually, the solution is all concentrated on the first mode by Theorem \ref{evenodd} and the thesis follows because the described behaviour is proper of the scalar asymmetric equation $\ddot\varphi_1(t) + (\mu+1) \varphi_1^+(t) - \varphi_1^-(t)= e(t)$.
\end{proof}

It is worth noticing that, in the situation of Theorem \ref{forz}, the oscillations of the first mode grow indefinitely but no transfer of energy takes place, all the residual modes remain equal to $0$ for any time instant and hence there is no instability (in the sense of our definitions).
\par
On the other hand, if we try to extend Theorem \ref{forz} to a mode other than the first, we have to take into account that the first mode always absorbs energy in view of Theorem \ref{highermodes}. Hence the picture appears different and, rather than unboundedness, a situation which is reminiscent of the beating phenomenon arises, as we briefly comment in the next subsection.

\subsection{Numerical results}
We first performed an experiment for problem \eqref{cdbeampositive}-\eqref{potential+} and obtained the plots in Figure \ref{plots12}.
\begin{figure}[!h]
\!\!\!\!\!\includegraphics[scale=0.45]{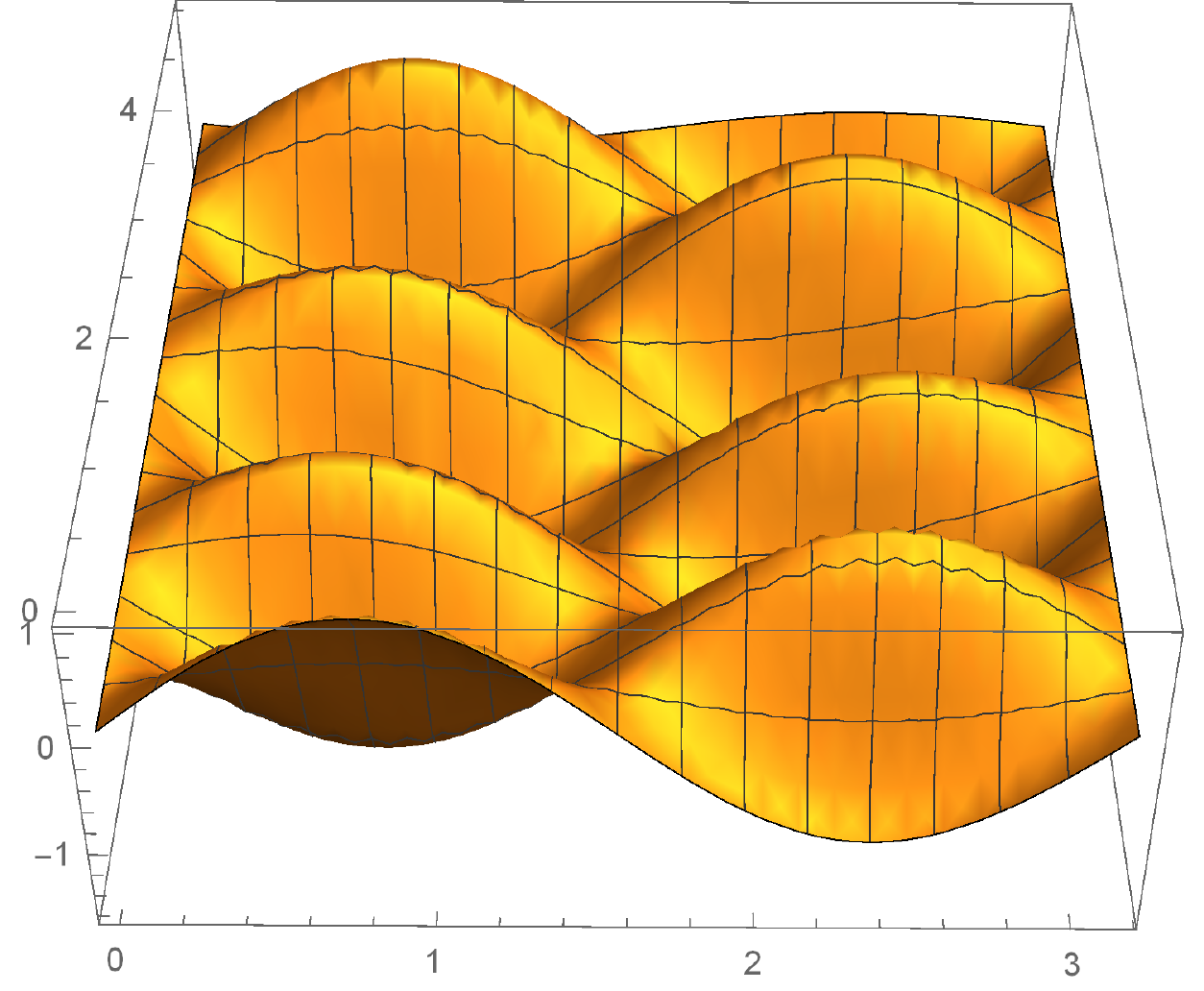}\qquad\quad\includegraphics[scale=0.35]{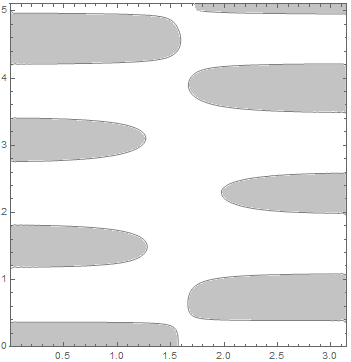}
\caption{For $\mu=3$ and $t\in(0,5)$, graph of the solution of \eq{cdbeampositive}-\eq{potential+} and of its positivity regions (gray).}
\label{plots12}
\end{figure}
\begin{figure}[!h]
\includegraphics[scale=0.35]{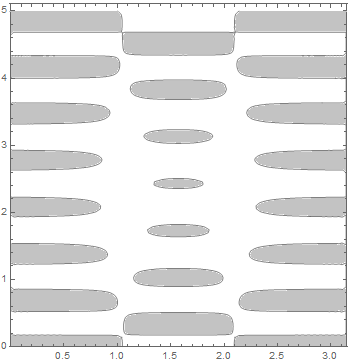}\qquad\qquad\qquad\includegraphics[scale=0.35]{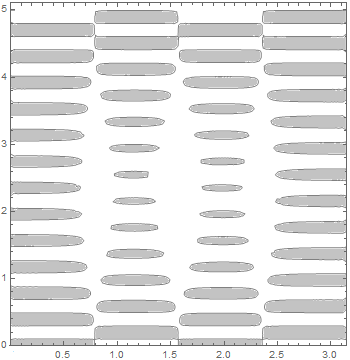}
\caption{For $\mu=3$ and $t\in(0,5)$, plot of the positivity regions of the solution of \eq{cdbeampositive} with initial data $u_0(x)=\sin (3x)$ and $u_0(x)=\sin (4x)$, respectively, and $u_1(x)=0$.}
\label{plots34}
\end{figure}
In the left picture the solution appears ``almost'' of the kind $u(x,t)\approx\phi_2(t)\sin(2x)$ but the positivity regions in the right picture
show that this is not the case, since there are values of $t>0$ where $u(x,t)<0$ for all $x\in(0,\pi)$: for those $t$ it is clear that the first mode $\sin(x)$
is dominant. Moreover, the negativity (white) region is larger, which means that mostly the force is not acting and the hangers are slacken, in line with the observation at the Tacoma Bridge, see \cite[V-12]{ammann}; a kind of periodic pattern can be observed as well. Actually, this feature of the first mode is independent of the prevailing mode: in Figure \ref{plots34} we show the picture when the prevailing modes are the third and the fourth one, with $u_0(x)=\sin (3x)$ and $u_0(x)=\sin (4x)$, respectively. Notice that Proposition \ref{homogeneity} allows us to deduce that the portrait is the same for any positive multiplicative constant in front of $u_0(x)$.
\par
In general, in our experiments we could not detect any form of instability for the autonomous problem, as if the half-linearity of the equation prevented the solution to display these effects: the growth of all the modes appears regular and of ``linear'' type.
\par
For the forced problem, in connection with the previous discussion about the Fu\v{c}ik spectrum, in Figure \ref{forzato1}
we show the solution of \eqref{finiteg} with $f(u)=3u^+$ and $N=5$, taking $g(x, t)=50\sin (x)\sin(\tfrac{4}{3} t)$, $\varphi_1^5(0)=0.01$, $\varphi_n^5(0)=0.00996$ for $n=2, 3, 4, 5$ and $\dot{\varphi}_n^5(0)=0$ for $n=1, \ldots, 5$. This choice is motivated by the behaviour highlighted by the numerical simulations for the scalar differential equation $\ddot{y} + y + 3 y^+ = 0$, where, setting $p(t)=\alpha \sin(\tfrac{4}{3}t)$, $\alpha \in \R$, the solution becomes unbounded. This seems to be the case also in Figure \ref{forzato1}, where we plot only the first three Fourier components since the other ones oscillate with semi-amplitude of about $0.01$ without evident changes in their behavior. A real structure may even collapse due to the excessive amplitude of the oscillations of the first mode, but still the solution of \eqref{finiteg} would be stable according to Definition \ref{finitestability}. Hence, we remark that in our setting instability is not equivalent to structural failures in the system, but is a clue of possible failures \emph{due to a switch in the kind of oscillations} of the solution.
\par
\begin{figure}[!ht]
\center
\includegraphics[scale=0.42]{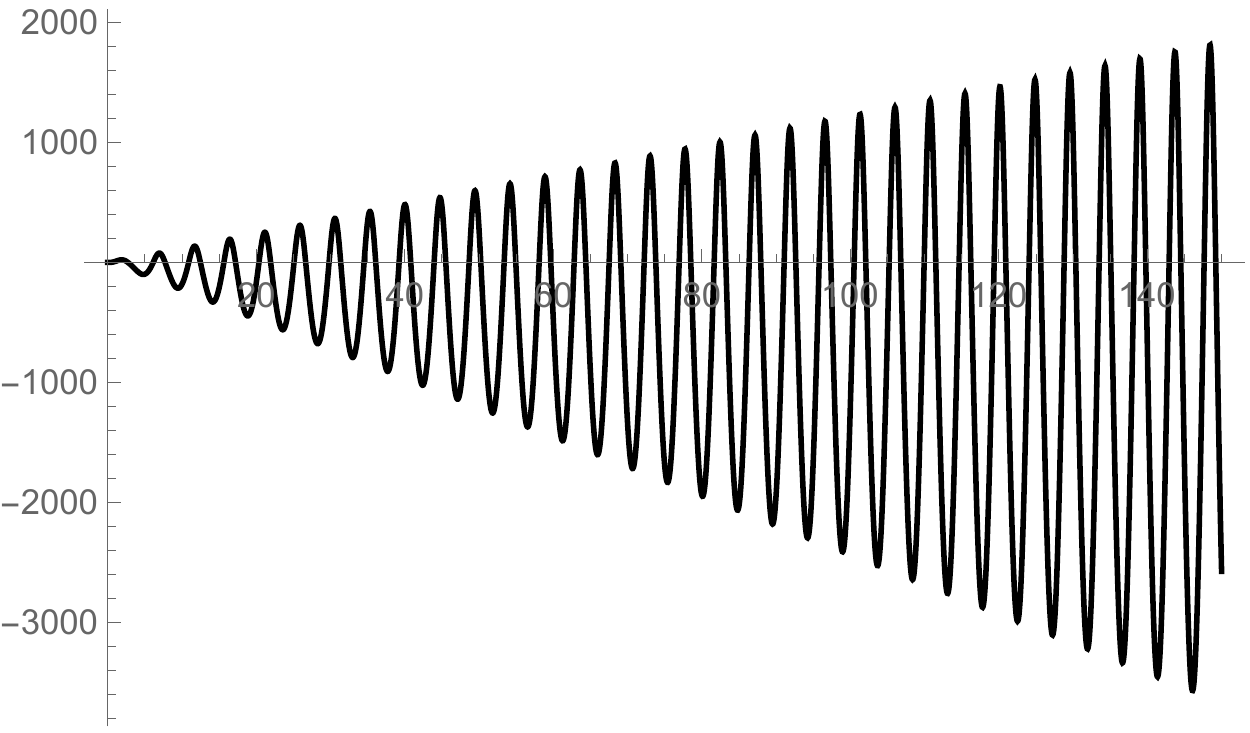}
\quad
\includegraphics[scale=0.42]{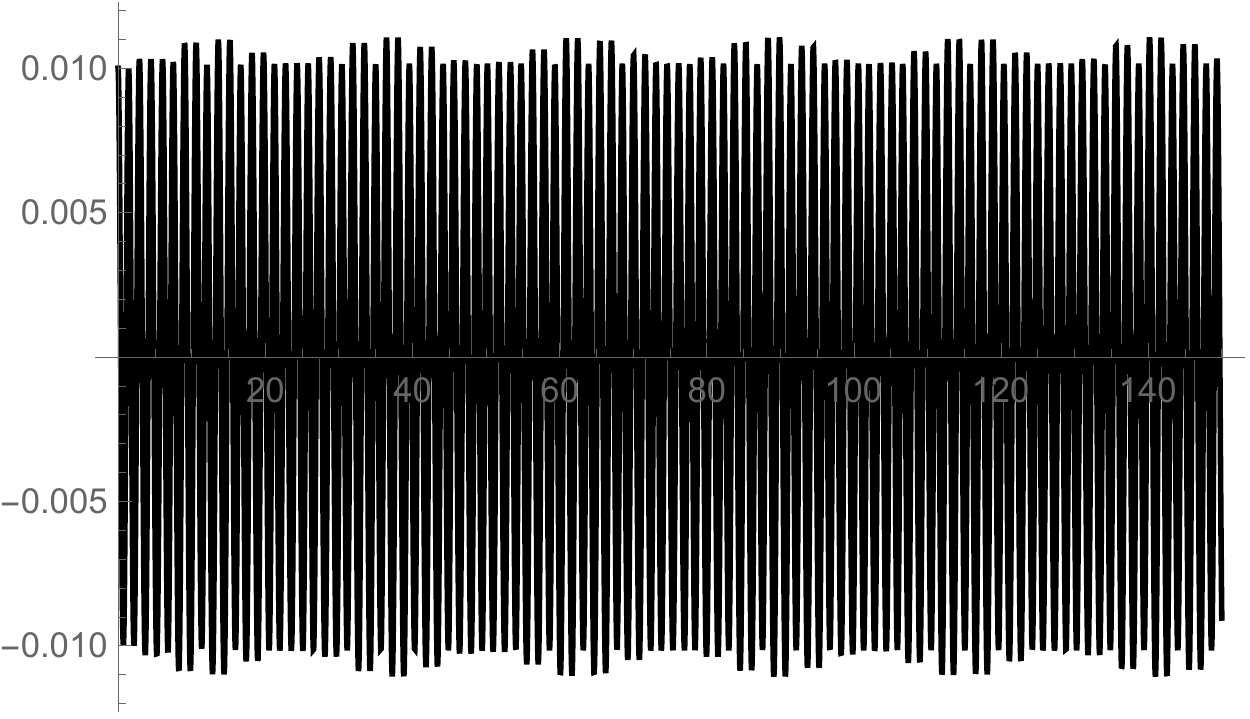}
\quad
\includegraphics[scale=0.42]{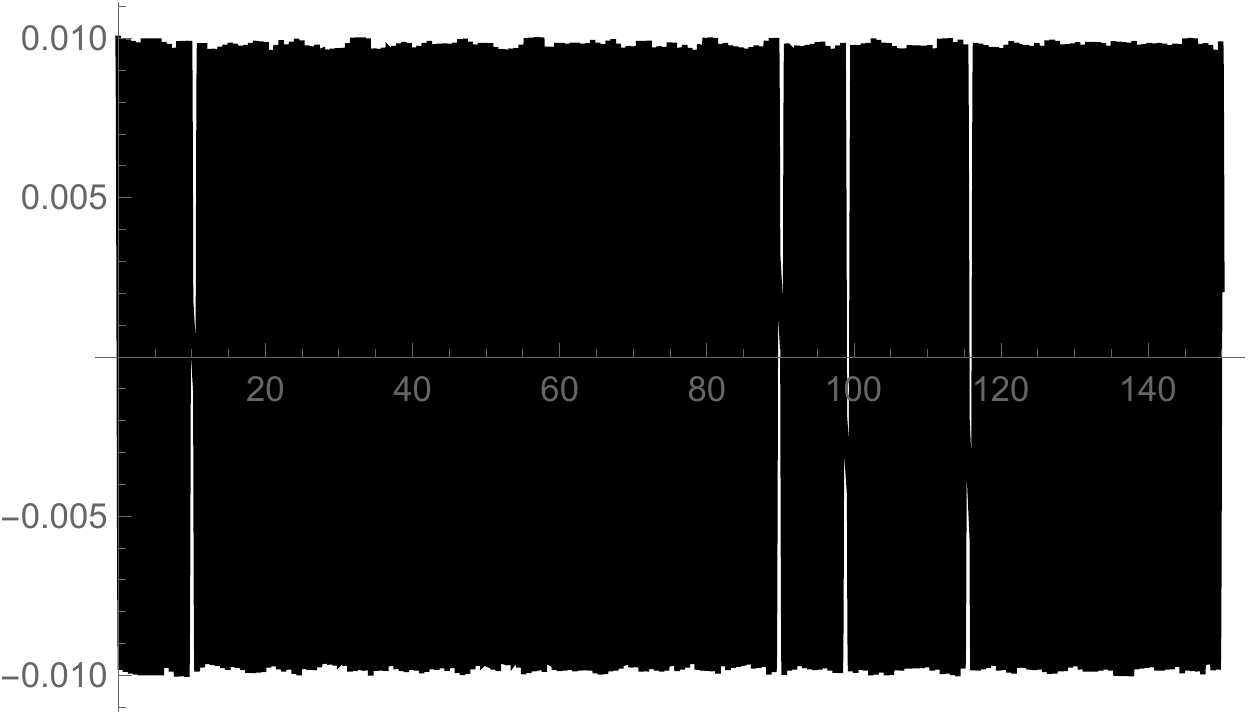}
\caption{The plots of $\varphi_1^5, \varphi_2^5, \varphi_3^5$ for problem \eqref{finiteg}, with $f(u)=3u^+$ and $g(x, t)=50\sin (x) \sin(\tfrac{4}{3} t)$, on the time interval $[0, T]=[0, 150]$.}
\label{forzato1}
\end{figure}
One then wonders if the same phenomenon may be detected with other prevailing modes $j \neq 1$, naturally leading to compare with the behaviour of the equation
\begin{equation}\label{asimmet}
\ddot\varphi_j(t) + (j^4+3) \varphi_j^+(t) - j^4 \varphi_j^-(t)= \alpha \sin (\gamma t).
\end{equation}
Here the picture is different from Theorem \ref{forz} since, in view of Theorem \ref{highermodes}, the first mode immediately absorbs energy from the second and increases its amplitude. The resulting scenario seems to give rise to bounded solutions presenting some patterns of ``growing-fainting'' oscillations, independently of the forced mode and of the amplitude of the forcing. We depict in Figure \ref{forzato2} the situation for the second mode, this time with $g(x, t)=\sin (2x) \sin[8\sqrt{19}t/(4 + \sqrt{19})]$, $\varphi_2^5(0)=0.01$, $\varphi_n^5(0)=0.00996$ for $n=1, 3, 4, 5$ and $\dot{\varphi}_n^5(0)=0$ for every $n=1, \ldots, 5$; the solution is stable according to Definition \ref{finitestability} and appears periodic-like.
\par
The choice of these ``strange'' forcing terms is suggested by Theorem \ref{forz}: indeed, the characteristic period associated with equation \eqref{asimmet} is $\pi/{\sqrt{j^4+3}}+\pi/j^2$ and we choose (both in Figures \ref{forzato1} and \ref{forzato2}, for $j=1$ and $j=2$, respectively) a forcing term with the same period, namely a possibly ``resonant'' one according to Theorem \ref{forz}.  
However, in view of Theorem \ref{highermodes} the effect of the considered external force on the dynamics of the system is completely different according to whether $j=1$ or $j=2$. 
\begin{figure}[!ht]
\center
\includegraphics[scale=0.42]{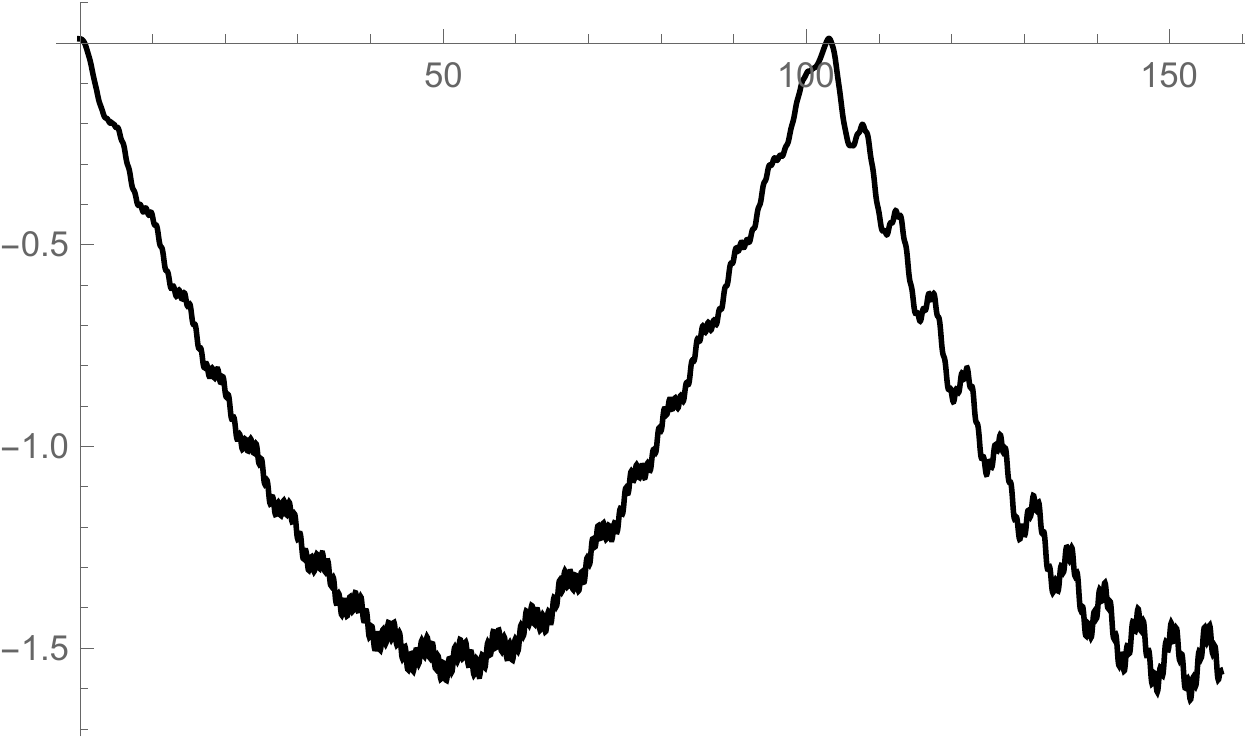}
\quad
\includegraphics[scale=0.42]{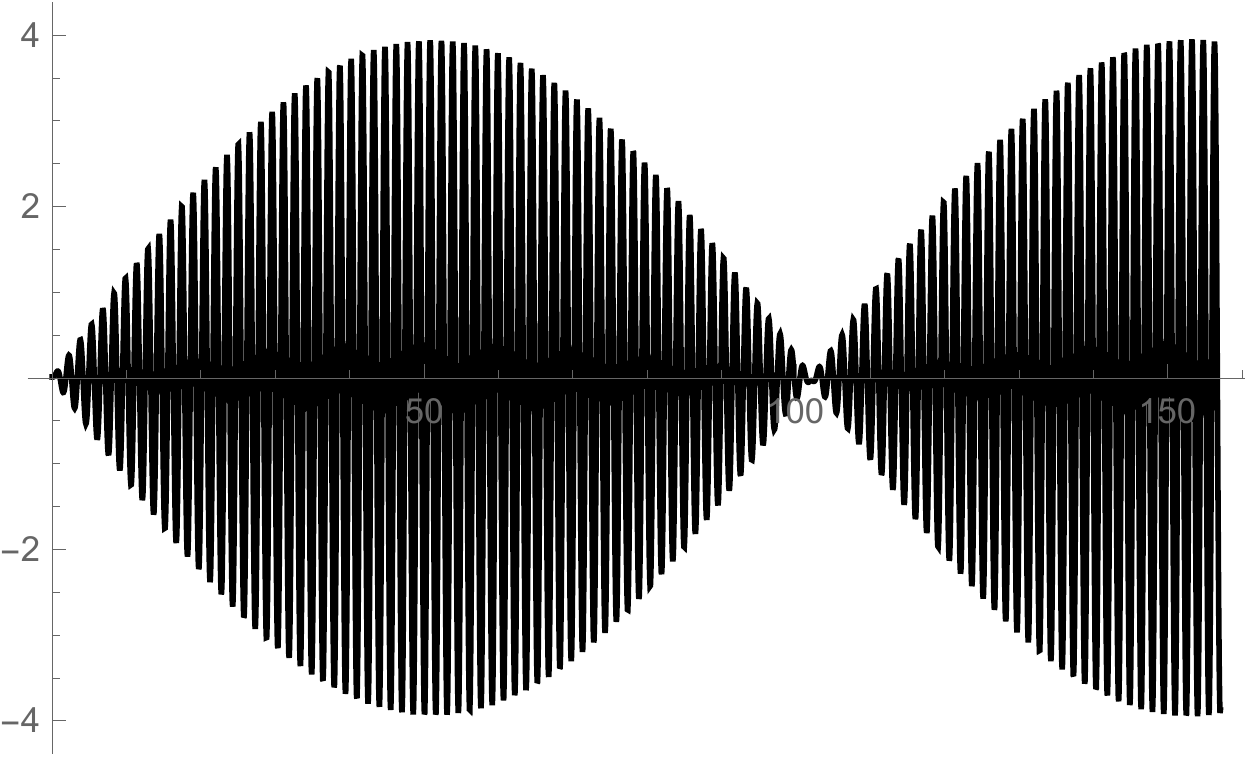}
\\
\vspace{0.3cm}
\includegraphics[scale=0.42]{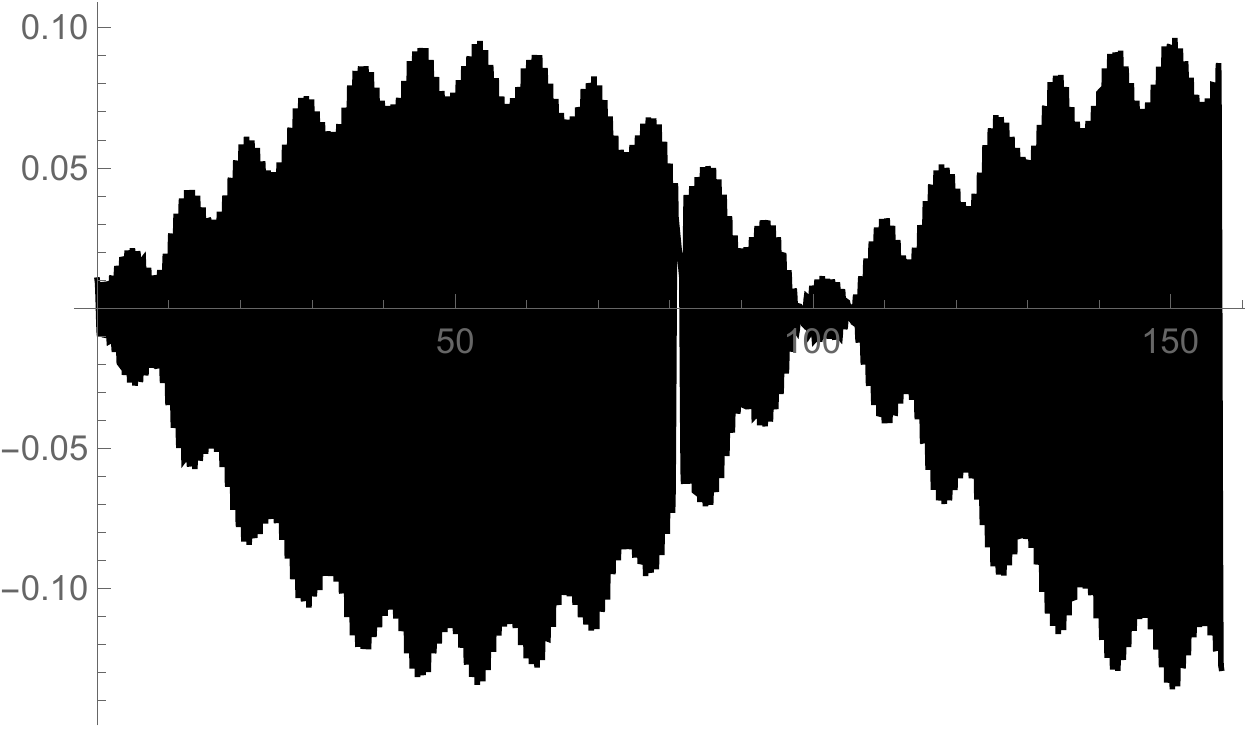}
\quad
\includegraphics[scale=0.42]{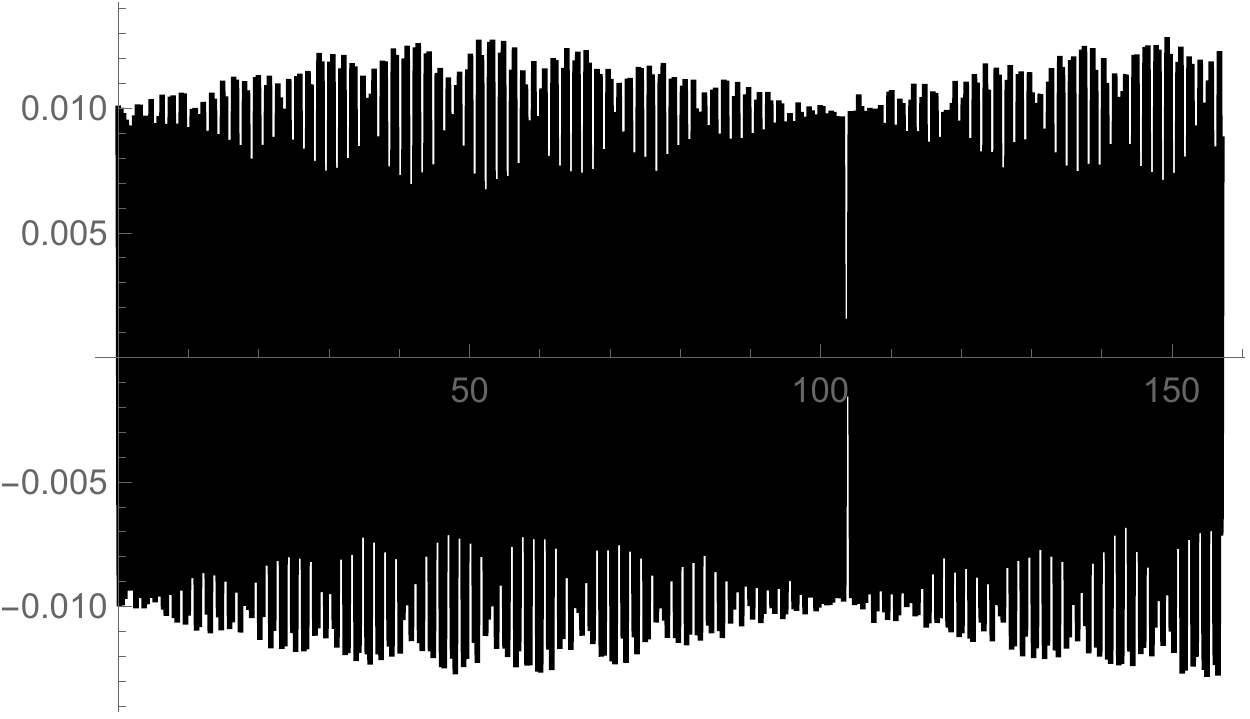}
\quad
\includegraphics[scale=0.42]{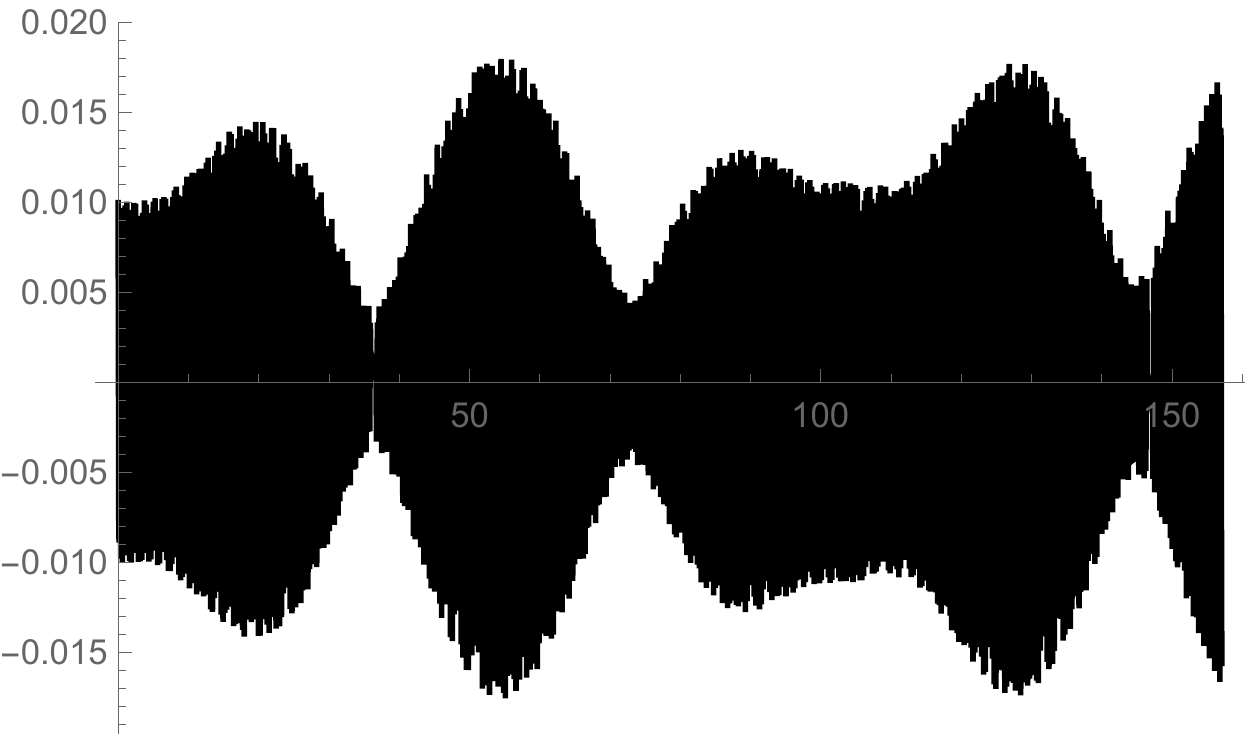}
\caption{The plots of $\varphi_1^5, \ldots, \varphi_5^5$ for problem \eqref{finiteg}, with $f(u)=3u^+$ and $g(x, t)=\sin (2x) \sin[8\sqrt{19}t/(4 + \sqrt{19})]$, on the time interval $[0, T]=[0, 160]$.}
\label{forzato2}
\end{figure}

As an example with a forcing term concentrated on the third mode, it may be seen numerically that a similar behaviour as the one displayed in Figure \ref{forzato2} arises for $g(x, t)=\sin (3x) \sin[18\sqrt{84}t/(9 + \sqrt{84})]$, with the only difference that the second mode remains small.
Again, instability seems not to occur in this context, also for a very large external force.

\subsection{An example of full proof of stability}

It is clear that, if one wishes to have high precision, then a large number $N$ of modes is needed; however, the computational cost of the numerical experiments for the non-smooth nonlinearity $f(u)=\mu u^+$ starts appearing relevant. Therefore, we choose $N=5$ as a compromise between highlighting the phenomena in full generality and maintaining the computational time acceptable (of the order of some hours). Nevertheless, this is enough to show a case of stability which fully fits in our theoretical framework. As a consequence of Theorem \ref{approximation} in the particular case when $f(u)=\mu u^+$, for each mode $n \leq 5$ we obtain the following estimate:
\begin{equation}\label{stimamodo2}
\Vert \varphi_n - \varphi_n^5  \Vert_{L^\infty(0, T)} \leq \sqrt{\frac{\mu+2}{\pi}}\frac{\sqrt{32 E(0)}+ \sqrt{32\pi} \frac{\alpha}{\gamma} \mathcal{I}(1+\frac{\gamma T}{\pi})}{36}\,
\frac{e^{\frac{\mu T}{2\sqrt{2(\mu+2)}}}-1}{\sqrt{n^4+\frac{1}{T^2}+\frac{\mu}{2}}} \qquad\forall0\le t\le T,
\end{equation}
where $\alpha, \gamma$ are as in \eqref{g} and we recall that $\mathcal{I}(\cdot)$ denotes the integer part.
As described in Section \ref{finitodim}, fixed $T > 2$ the strategy is 
to look first at the plots of the solution of the finite-dimensional system \eqref{finiteg}, verifying if \eqref{finitestable} is satisfied. If this is the case, we check if the right-hand side in \eqref{stimamodo2} is sufficiently small so as to enter the setting of Theorem \ref{stable}.
\par
We succeed in this plan for the choices $\eta=0.999$ (and hence $\eta^4 \approx 0.996$), $T=5$ and
\begin{equation}\label{positionsexp}
\mu=0.1, \; N=5, \; u_0(x)=10^{-3} \Big(0.996\sum_{n=1 \atop n \neq 2}^5 \sin (nx) + \sin(2x)\Big), \; u_1(x) \equiv 0, \; g(x, t)=5\cdot 10^{-3} \sin (2x) \sin (t).
\end{equation}
In this case, one has indeed  
$$
L(M) = 0.1, \quad C \approx 0.049, \quad E(0) \approx 7.7 \cdot 10^{-4}, \quad M \leq 0.065
$$
and from \eqref{stimamodo2} it is thus possible to infer that, for every $n \in \{1, 3, 4, 5\}$, it holds
$$
\Vert \varphi_n - \varphi_n^5 \Vert_\infty \leq 7 \cdot 10^{-4}.
$$
Taking into account the picture for the finite-dimensional approximation shown in Figure \ref{plotsS}, this bound 
allows to infer that the second condition in \eqref{finitestable} holds for every $n \in \{1, 3, 4, 5\}$. Hence, Theorem \ref{stable} yields the stability of the considered solution of \eqref{cdbeampositive} until time $T=5$. Once more, stability is meant in the sense of Definition \ref{unstable}.
\begin{figure}[!h]
\center
\includegraphics[scale=0.42]{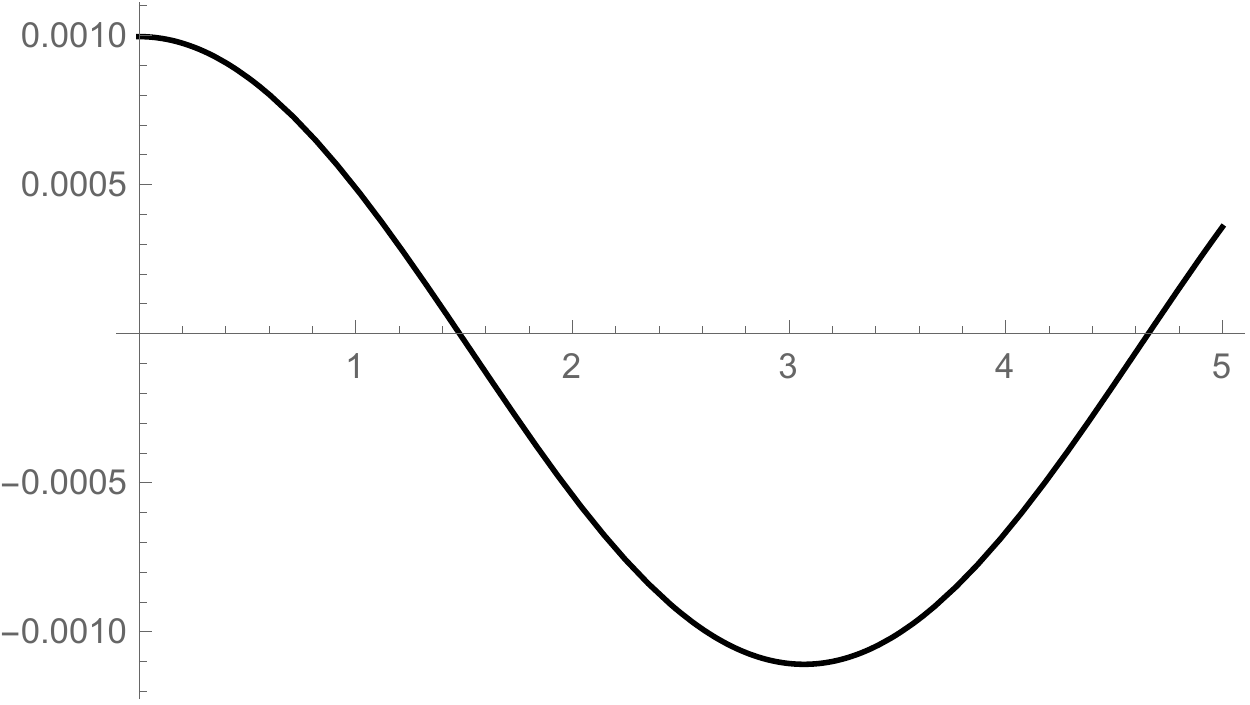}
\quad
\includegraphics[scale=0.42]{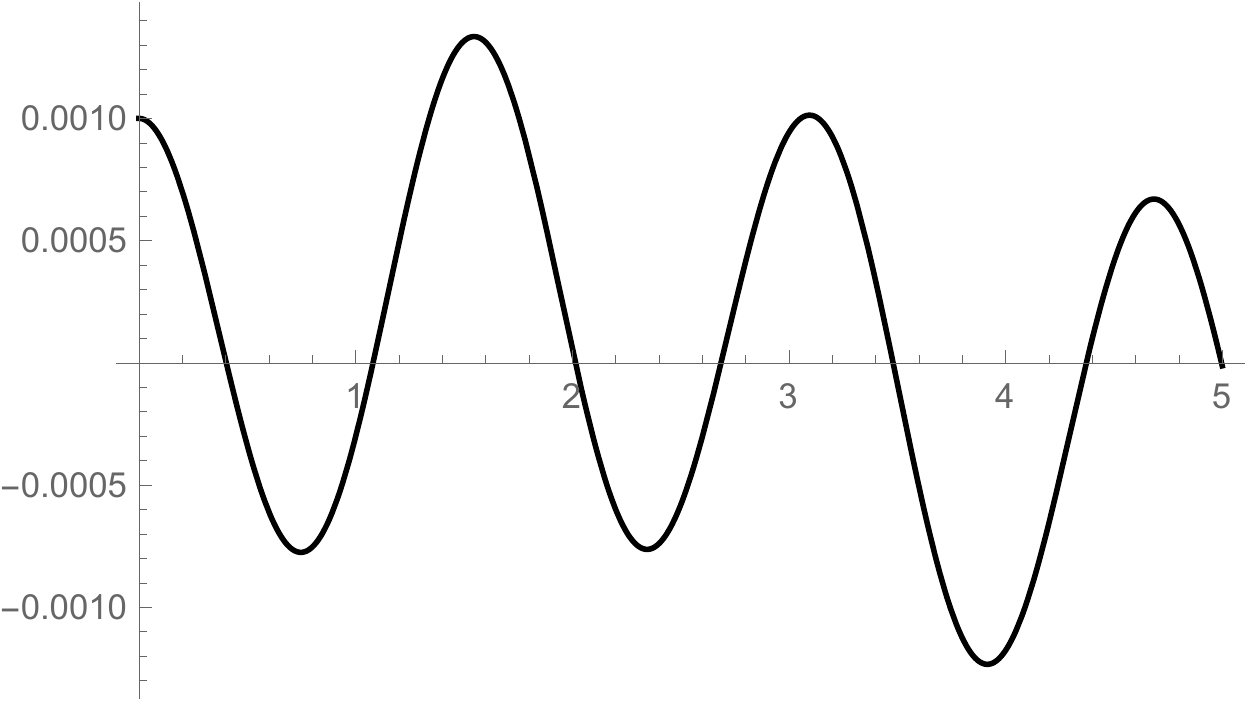}
\\
\vspace{0.3cm}
\includegraphics[scale=0.42]{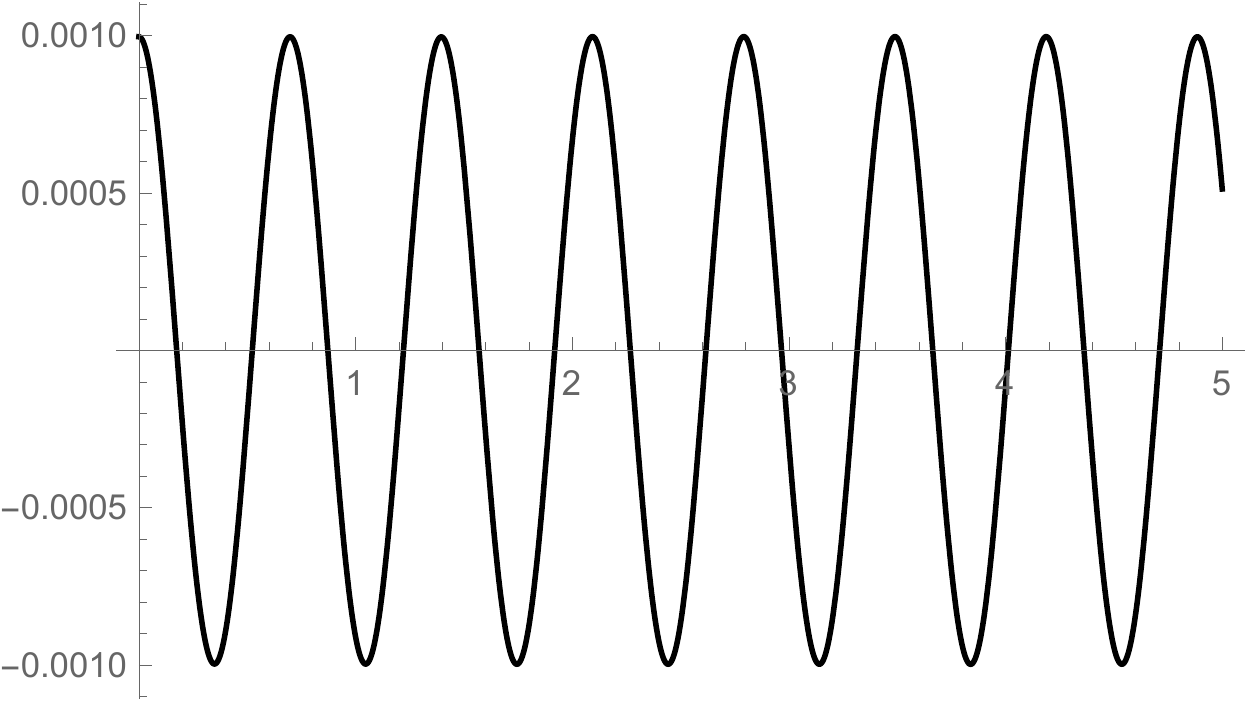}
\quad
\includegraphics[scale=0.42]{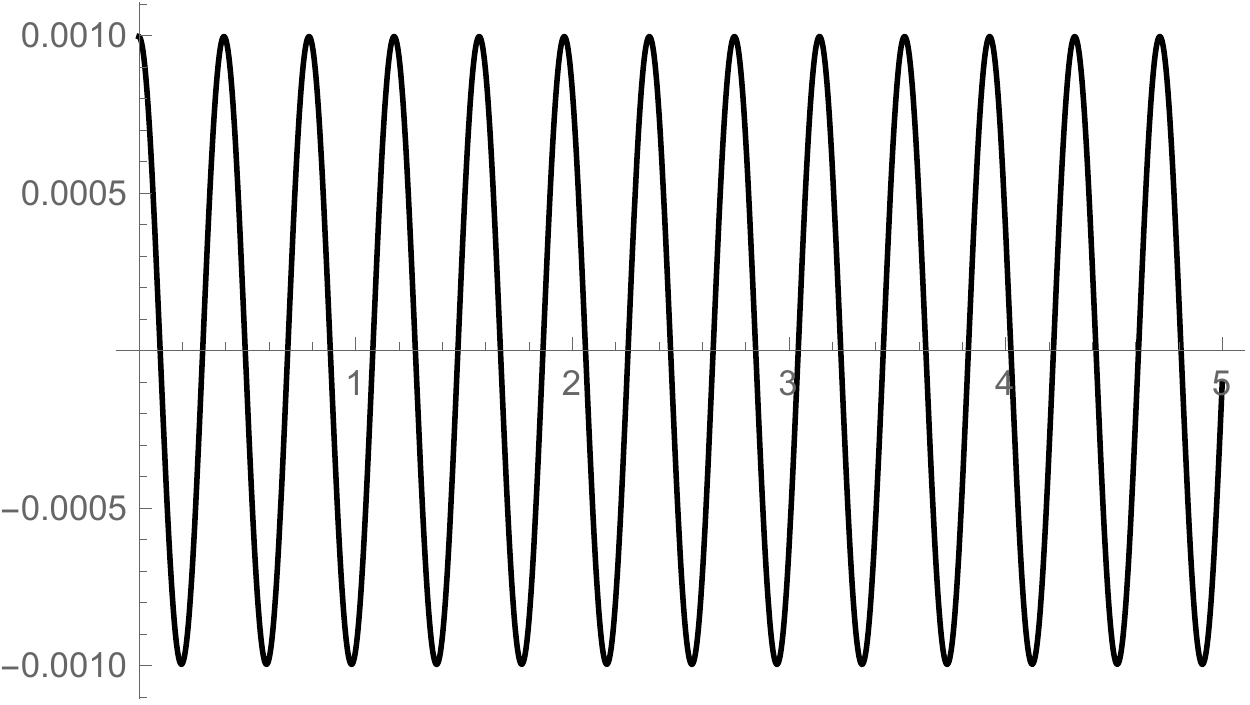}
\quad
\includegraphics[scale=0.42]{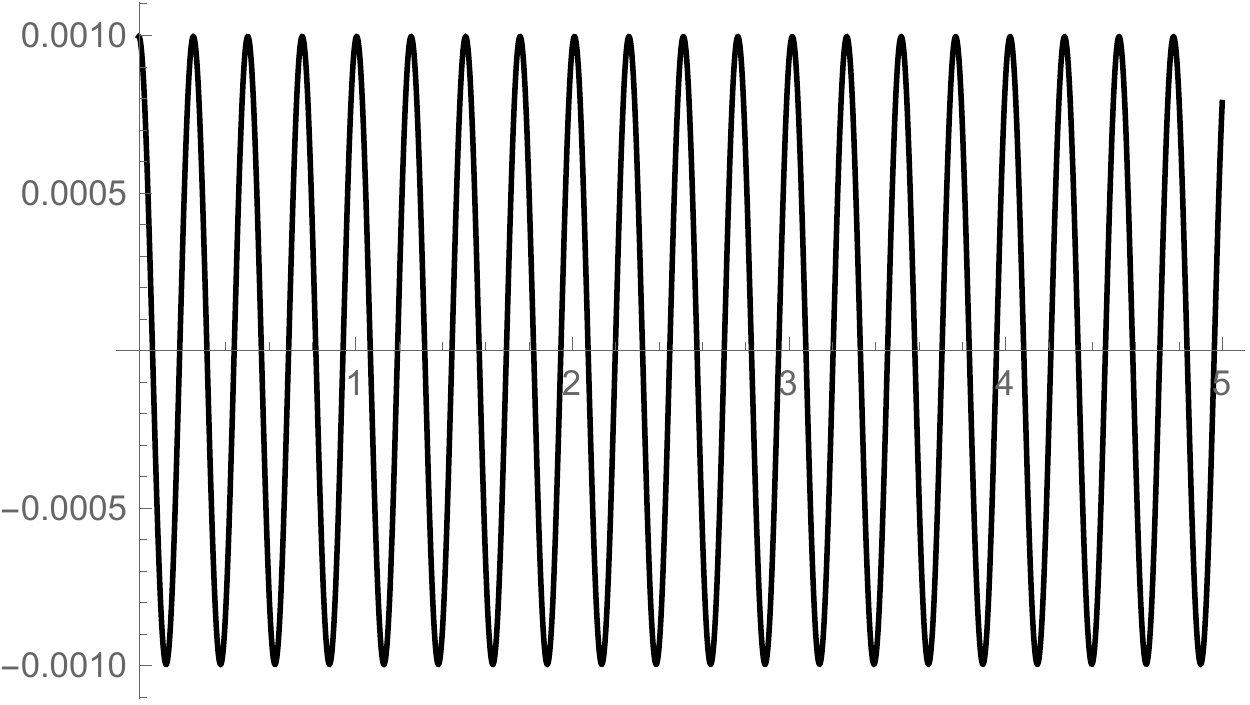}
\caption{The plots of $\varphi_1^5, \ldots, \varphi_5^5$ for problem \eqref{finiteg} -\eqref{positionsexp}, on the time interval $[0, 5]$.}
\label{plotsS}
\end{figure}

\section{The case of a cubic nonlinearity}\label{casecub}

A cubic nonlinearity naturally arises when large deflections of a plate or a beam (modeling a suspension bridge) are involved: in this case, the stretching effects suggest to use variants of the von K\'arm\'an theory, where cubic terms appear \cite{gazwang}.
Indeed, cubic nonlinearities have also been suggested to model the behaviour of the restoring force $f$, based on the following arguments.
First, the hangers may continue their action also after slackening; in this respect, Brownjohn \cite[p.1364]{brown} claims that {\em the hangers
are critical elements in a suspension bridge and for large-amplitude motion their behaviour is not well modelled by either simple on/off stiffness
or invariant connections}. Second, the nonlinearity may also take into account the nonlinear behaviour of the sustaining cables to which the hangers
are connected. Augusti-Sepe \cite{sepe1} view the restoring force at the endpoints of a cross-section of the deck as composed by two connected springs,
the top one representing the action of the sustaining cable and the bottom one (connected with the deck) representing the hangers. Moreover,
the action of the cables is considered the main cause of the nonlinearity of the restoring force by Bartoli-Spinelli \cite[p.180]{bartoli}, who
suggest quadratic and cubic perturbations of a linear behaviour.
In view of these remarks, Plaut-Davis \cite[$\S$ 3.5]{plautdavis} take $f(u)=\alpha_1u+\alpha_2u^3$,
for some $\alpha_1, \alpha_2>0$ depending on the elasticity of the cables and hangers. We neglect the linear term by taking $\alpha_1=0$ and we
study the transfer of energy for \eq{nonlinear} with the nonlinearity
$$
f(u)=u^3\, .
$$

\subsection{Structural properties of the equation}

In this section, we consider the problem
\neweq{cdbeamcubic}
\left\{\begin{array}{ll}
u_{tt}+u_{xxxx}+u^3=0\quad & \mbox{for }(x,t)\in(0,\pi)\times(0,\infty)\\
u(0,t)=u(\pi,t)=u_{xx}(0,t)=u_{xx}(\pi,t)=0\quad & \mbox{for }t\in(0,\infty)\\
u(x,0)=u_0(x)\, ,\quad u_t(x,0)=u_1(x)\quad & \mbox{for }x\in(0,\pi)\, .
\end{array}\right.\endeq
By seeking solutions of \eq{cdbeamcubic} in the form \eq{general}, system \eqref{infinite} reads
\neweq{firstattempt}
\ddot{\phi}_n(t)+n^4\phi_n(t)+\frac{2}{\pi}\int_0^\pi\Big(\sum_{m=1}^\infty\phi_m(t)\sin(mx)\Big)^3\sin(nx)\, dx=0
\endeq
and one has to compute the integral $I$ with the cubic power of the series. To this end,
let us rewrite the integral $I$ in \eq{firstattempt} as
\begin{eqnarray*}
I &=& \int_0^\pi\Big(\phi_n(t)\sin(nx)+\sum_{m\neq n}\phi_m(t)\sin(mx)\Big)^3\sin(nx)\, dx\\
\ &=& \phi_n(t)^3\int_0^\pi\sin^4(nx)\, dx+3\phi_n(t)^2\sum_{m\neq n}\phi_m(t)\int_0^\pi\sin^3(nx)\sin(mx)\, dx\\
\ &\ & +3\phi_n(t)\!\left[\sum_{m\neq n}\phi_m(t)^2\!\int_0^\pi\!\sin^2(nx)\sin^2(mx)dx+2\!\!\sum_{m,l\neq n \atop m>l}\!\phi_m(t)\phi_l(t)\!
\int_0^\pi\!\sin^2(nx)\sin(mx)\sin(lx)dx\right]\\
\ &\ & +\int_0^\pi\Big(\sum_{m\neq n}\phi_m(t)\sin(mx)\Big)^3\sin(nx)\, dx\, .
\end{eqnarray*}

We can make this slightly more explicit thanks to Lemma \ref{calculus} in the Appendix at the end of the paper; in this way, \eq{firstattempt} becomes
\neweq{secondattempt}
\ddot{\phi}_n(t)+A_n(t)+B_n(t)\phi_n(t)-\frac34 \phi_{3n}(t)\phi_n(t)^2+\frac34 \phi_n(t)^3=0
\endeq
where $A_n(t)$ is a homogeneous polynomial of degree 3 with respect to the Fourier components $\phi_m(t)$ (with $m\neq n$) and
$$
B_n(t)=n^4+\frac32\sum_{m\neq n}\Big(\phi_m(t)+\phi_{2n-m}(t)-\phi_{2n+m}(t)\Big)\phi_m(t)\, ;
$$
notice that $\sum_{m\neq n}\phi_{2n-m}(t)\phi_m(t)$ only contains a finite number of terms (possibly none), that is, the ones for which $m<2n$ (there
are none if $n=1$). We are not interested in the explicit expression of the coefficients of equation \eq{secondattempt}, but we performed these computations with the aim of emphasizing some interactions between the modes. The presence of coefficients such as $\phi_{2n+m}(t)$ or $\phi_{3n}(t)$ is clearly due to the specific nonlinearity $f(u)=u^3$ considered in \eq{cdbeamcubic}.\par
Let us now concentrate on the term $A_n(t)$, which does not depend on the unknown $\phi_n$ and therefore acts as a forcing term in \eq{secondattempt}.

\begin{definition}
Let $p, q, r, s\in\N$ be all different. We say that:\par
$\bullet$ the $q$-th mode influences the $p$-th mode if $A_p$ contains the term $\phi_q^3$;\par
$\bullet$ the couple of modes $(q, r)$ influences the $p$-th mode if $A_p$ contains one of the two terms $\phi_q^2\phi_r$ or $\phi_q\phi_r^2$;\par
$\bullet$ the triple of modes $(q, r, s)$ influences the $p$-th mode if $A_p$ contains the term $\phi_q\phi_r\phi_s$.
\end{definition}

Roughly speaking, if some modes influence the $p$-th mode, it may happen that $\varphi_p\not\equiv0$ even if
$(\phi_p(0),\dot{\phi}_p(0))=(0,0)$. The following result explains in which way the modes influence each other.

\begin{proposition}\label{influence}
The following statements hold true:\par
$\bullet$ if $p=3q$, then the $q$-th mode influences the $p$-th mode;\par
$\bullet$ if $q<r$ and $p\in\{2r+q,2r-q,r+2q,|r-2q|\}$, then the couple of modes $(q,r)$ influences the $p$-th mode;\par
$\bullet$ if $q<r<s$ and $p\in\{s+r+q,s+r-q,s-r+q,|s-r-q|\}$ then the triple of modes $(q,r,s)$ influences the $p$-th mode.
\end{proposition}

The proof can be easily deduced from Lemma \ref{calculus} in the Appendix. The statement of Proposition \ref{influence} is a consequence of the
considered nonlinearity $f(u)$: the fact that the $q$-th mode influences the $3q$-th mode depends on the fact that $f(u)$ is cubic.\par
We briefly comment on the meaning of Proposition \ref{influence}. In principle, a risk may arise from the truncation procedure, the one of confusing a physiological transfer of energy with instability. Explicitly, when the initial data are concentrated on a single mode, i.e., $u_0(x)= A\sin(jx)$ for some integer $j$ and $u_1(x)\equiv 0$, in view of Proposition \ref{influence} one sees no coupling and no physiological energy exchanges if $N<3j$: as in the analysis of a fish-bone model for suspension bridges \cite{bergaz}, the Galerkin system \eqref{finiteg} reduces to the single Duffing equation
$$
\ddot\varphi_j^N(t) + j^4 \varphi_j^N(t) + \frac{3}{4} (\varphi_j^N)^3(t) = \alpha \sin (\gamma t)\, ,\quad \varphi_j^N(0)=A\, ,\quad\dot\varphi_j^N(0)=0\, .
$$
On the contrary, as soon as the initial datum is slightly perturbed (for instance, considering $u_0(x) = A \sin (jx) + \delta \sum_{n \neq j} \sin (n x)$ for a small number $\delta$), one passes from a picture with only one nonzero component to a situation where all the Fourier components interact and oscillate (again by Proposition \ref{influence}), and this may be identified as instability. However, Definition \ref{unstable} rules out this risk, since the influences described in Proposition \ref{influence} are all observed almost instantly, so that the Wagner effect is not accentuated (nor delayed) and this behaviour can be classified as physiological.
According to Definition \ref{unstable}, these energy transfers do not affect the stability analysis.
\par
Proposition \ref{influence} can be ``iterated'' and, for instance, the $q$-th mode influences also the $9q$-th mode;
however, after a few iterations the effects are undergone by very high modes and again they take place after a very small amount of time, so that the above discussion remains valid also in this case.
\par
A straightforward consequence of Proposition \ref{influence} is that both the odd and the even modes are reluctant to an energy transfer, as stated
in the following result.

\begin{theorem}\label{evenoddcompleto}
If $\{a_n\}_{_n}\in\ell^2_4$, $\{b_n\}_{_n}\in\ell^2_2$, and
$$
u_0(x)=\sum_{n=0}^\infty a_n\sin\big((2n+1)x\big)\ ,\quad u_1(x)=\sum_{n=0}^\infty b_n\sin\big((2n+1)x\big)\,,
$$
then there exist functions $\phi_{2n+1}\in C^2(\R_+)$ such that the solution of \eqref{cdbeamcubic} is given by
$$
u(x,t)=\sum_{n=0}^\infty \phi_{2n+1}(t)\sin\big((2n+1)x\big)\, .
$$
Similarly, if
$$
u_0(x)=\sum_{n=0}^\infty a_n\sin\big(2nx\big)\ ,\quad u_1(x)=\sum_{n=0}^\infty b_n\sin\big(2nx\big)\,,
$$
then there exist functions $\phi_{2n}\in C^2(\R_+)$ such that the solution of \eqref{cdbeamcubic} is given by
$$
u(x,t)=\sum_{n=0}^\infty \phi_{2n}(t)\sin\big(2nx\big)\, .
$$
\end{theorem}
\begin{proof}
The statement follows by noticing that, in view of Proposition \ref{influence}, a set of modes which have all the same parity may only influence modes
with the same parity, that is, the other modes solve an equation which has only the zero solution by uniqueness.
\end{proof}

We incidentally observe that also the modes which are multiples of a fixed integer form a closed set with respect to the energy exchange, see Proposition \ref{influence}.

\subsection{Numerical results}

We show some numerical experiments considering $N=12$ modes and $T=16$. We limit ourselves to the autonomous case, namely $g(x, t) \equiv 0$. Contrary to Section \ref{casepos}, where the half-linearity of the problem makes us conjecture that no instability phenomena occur,
the general rule for the cubic nonlinearity appears richer. In particular, it seems that each prevailing mode $j$ possesses a critical amplitude $M_j$ with the following properties:
\begin{itemize}
\item[-] if the initial value of the prevailing Fourier component $\varphi_j$ is below $M_j$, then all the residual components oscillate regularly around their initial datum, up to the possible exceptions represented by the modes which undergo a physiological transfer of energy (mainly $\varphi_{3j}$, if $3j \leq N$, and all the couples influenced by $\varphi_j, \varphi_{3j}$). However, these physiological exchanges occur almost instantly and generate regular oscillations;
\item[-] if the initial amplitude of $\varphi_j$ is larger than $M_j$, then there exists a residual mode which suddenly increases its amplitude in an uncontrolled way, displaying an energy transfer according to Definition \ref{unstable}.
\end{itemize}
Hence, a change of the behaviour of the residual components, with tendency to instability according to Definition \ref{unstable}, is likely to be observed for initial values of $\varphi_j$ around $M_j$, which represents an instability threshold for the mode $j$ before time $T$.
\par
We focus on the particular situation when the second mode is prevailing. In this case, it may be seen that $M_2 \approx 6.2$. Just to remark quantitatively the different behaviour below or above this threshold, in Table \ref{amplitude} we report the approximate maximum amplitude of the oscillations of the first 9 modes of the solution of \eqref{firstattempt}
with $u_1(x)=0$ and $u_0(x)$ as in the first column (up to an initial amplitude of 0.01 on the residual modes). We observe the differences taking the initial amplitude of the second component equal, respectively, to $0.5 M_2$ and $M_2$. 

\begin{table}[!ht]
\begin{center}
\begin{tabular}{|c|c|c|c|c|c|c|c|c|c|}
\hline
$u_0(x) \approx$ & $\|\phi_1\|_\infty$ & $\|\phi_2\|_\infty$ & $\|\phi_3\|_\infty$ & $\|\phi_4\|_\infty$ & $\|\phi_5\|_\infty$ & $\|\phi_6\|_\infty$ & $\|\phi_7\|_\infty$ & $\|\phi_8\|_\infty$ & $\|\phi_9\|_\infty$ \\
\hline
$3.1\sin(2x)$ & $0.012$ & $3.1$ & $0.011$ & $0.01$ & $0.01$  & $0.01$ & $0.01$  & $0.01$ & $0.01$ \\
\hline
$6.2\sin(2x)$ & $1.2$ & $6.2$ & $0.24$ & $0.052$ & $0.051$ & $0.09$  & $0.013$ & $0.01$ & $0.01$  \\
\hline
\end{tabular}
\vspace{0.2 cm}
\caption{For $t\in (0,16)$, approximate amplitude of oscillation of the first $9$ modes.}\label{amplitude}
\end{center}
\end{table}
The nonlinear response on dilating the amplitude of the initial datum is evident: doubling the initial datum of the prevailing component, the first and the third Fourier coefficients grow of two and one orders of magnitude, respectively. In the second row, we also notice the physiological transfer of energy on $\varphi_6$, which, as seen in Figures \ref{plots1cubo}-\ref{plots3cubo}, has a stronger (and instantaneous) $L^\infty$-growth than $\varphi_4, \varphi_5$, for instance.
\par
We now go one step further and show in details the transition from stability to instability with the corresponding pictures of the Fourier components. We take again
$$
u_0(x) = M \sin (2x) + \delta \sum_{n=1 \atop n\neq 2}^{12} \sin (nx), \quad M \gg \delta, \quad u_1 \equiv 0.
$$
Thus, the second mode is prevailing and we fix
$$
M \approx 6 \quad \textrm{ and } \quad \delta=0.01.
$$
We here note the striking change of picture for the autonomous Galerkin system passing from the stable situation $M=6$ (Figure \ref{plots1cubo}) to $M=6.1$ (Figure \ref{plots2cubo}), where the system ``prepares'' itself to display instability, until $M=6.2$ (Figure \ref{plots3cubo}), where instability appears. We conclude that $M_2 \approx 6.2$ is an instability threshold before time $T=16$. In the next three figures, we display this process in its entirety,  in particular:
\\
-- in Figure \ref{plots1cubo}, we see that the energy is immediately spread onto all the residual components, which start oscillating quite regularly around their initial value, with the exception of the $6$-th component. Indeed, according to Proposition \ref{influence}, $\varphi_6^{12}$ physiologically receives some energy from $\varphi_2^{12}$ and grows instantaneously of approximately eight times. The whole picture appears as a stability one;
\\
-- in Figure \ref{plots2cubo}, we observe an abrupt change in the oscillations of $\varphi_1^{12}$ (and, simultaneously, of $\varphi_3^{12}$, that is the component undergoing a physiological influence by $\varphi_1^{12}$), which starts growing with an exponential-type behaviour. Figure \ref{plots2cubo} describes again a stability situation, since the first condition in \eqref{finitegrande} is not fulfilled, but the picture has changed considerably with respect to Figure \ref{plots1cubo} and the first component is ``preparing'' itself to capture - in a non-physiological way - some energy from the second mode. The other residual components which grow do not reach significant amplitudes with respect to the prevailing mode. We underline that $\varphi_1^{12}$ has grown of $20$ times, with respect to Figure \ref{plots1cubo}, in correspondence of a small variation ($1.67 \% $) of the initial value of $\varphi_2^{12}$;
\\
-- in Figure \ref{plots3cubo}, the solution becomes unstable according to Definition \ref{finitestability}, since the first mode has captured enough energy in order to fulfill the two conditions in \eqref{finitegrande}. In comparison with Figure \ref{plots2cubo}, notice again that a small variation of the amplitude of the prevailing mode at the beginning ($1.64 \%$) leads to a gain of more than five times in the amplitude of the residual mode displaying instability.
\par
We here observe that the higher modes have a very large frequency but basically oscillate with the same amplitude as their initial datum, with no evidence of possibly modifying their behaviour in a neighbouring time interval; for the sake of brevity, in Figures \ref{plots2cubo} and \ref{plots3cubo} we will sometimes omit to depict all the modes, focusing on the most significant ones.
\newpage

\begin{figure}[!ht]
\center
\includegraphics[scale=0.31]{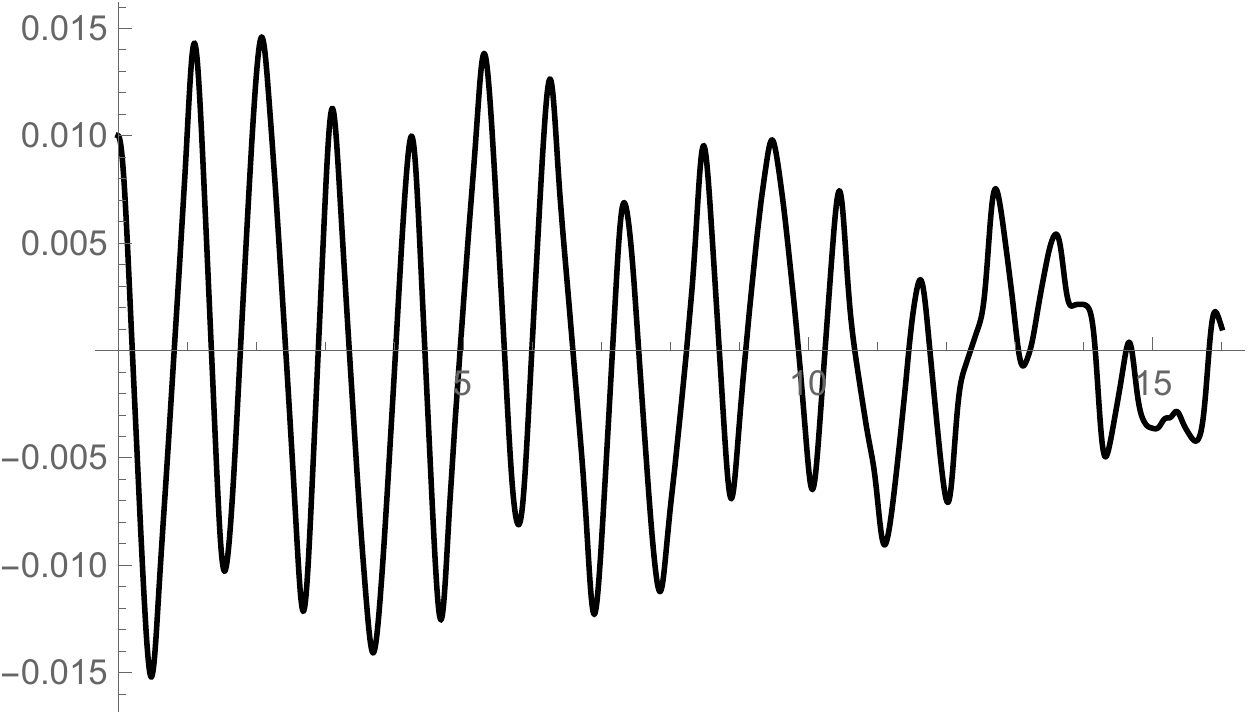}
\;
\includegraphics[scale=0.31]{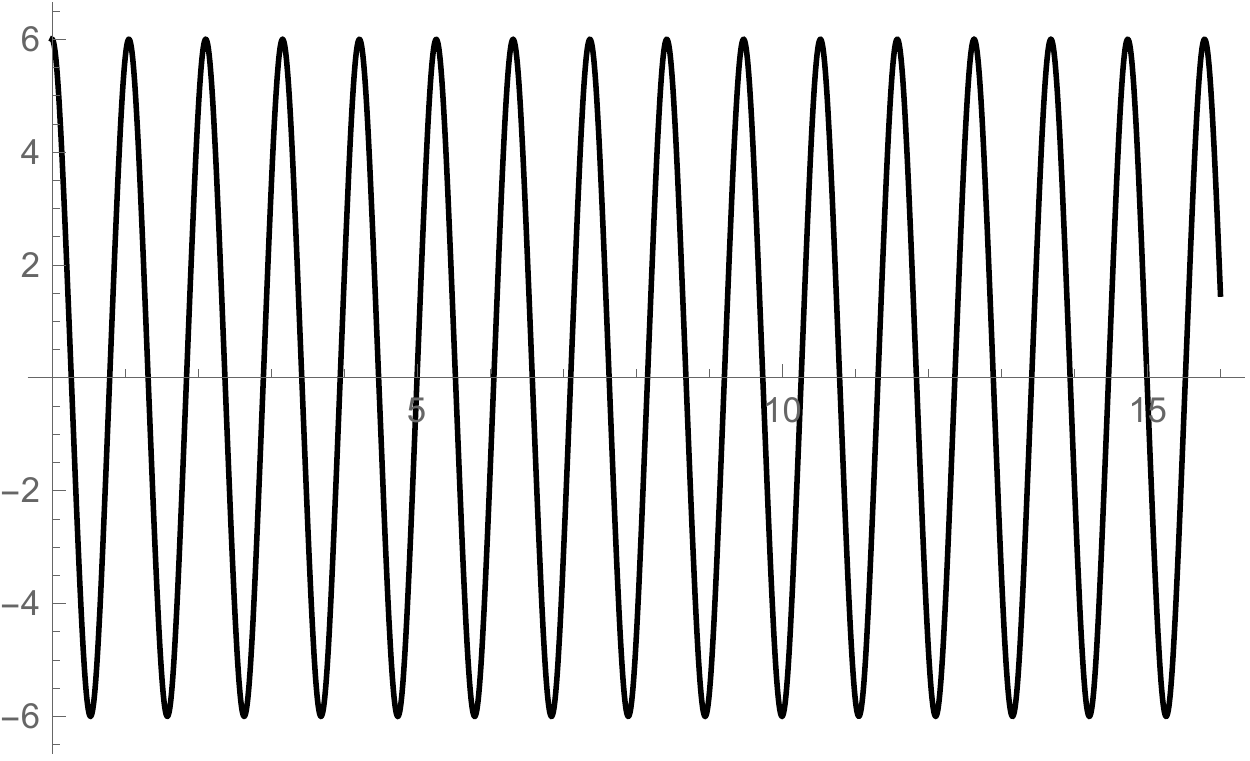}
\;
\includegraphics[scale=0.31]{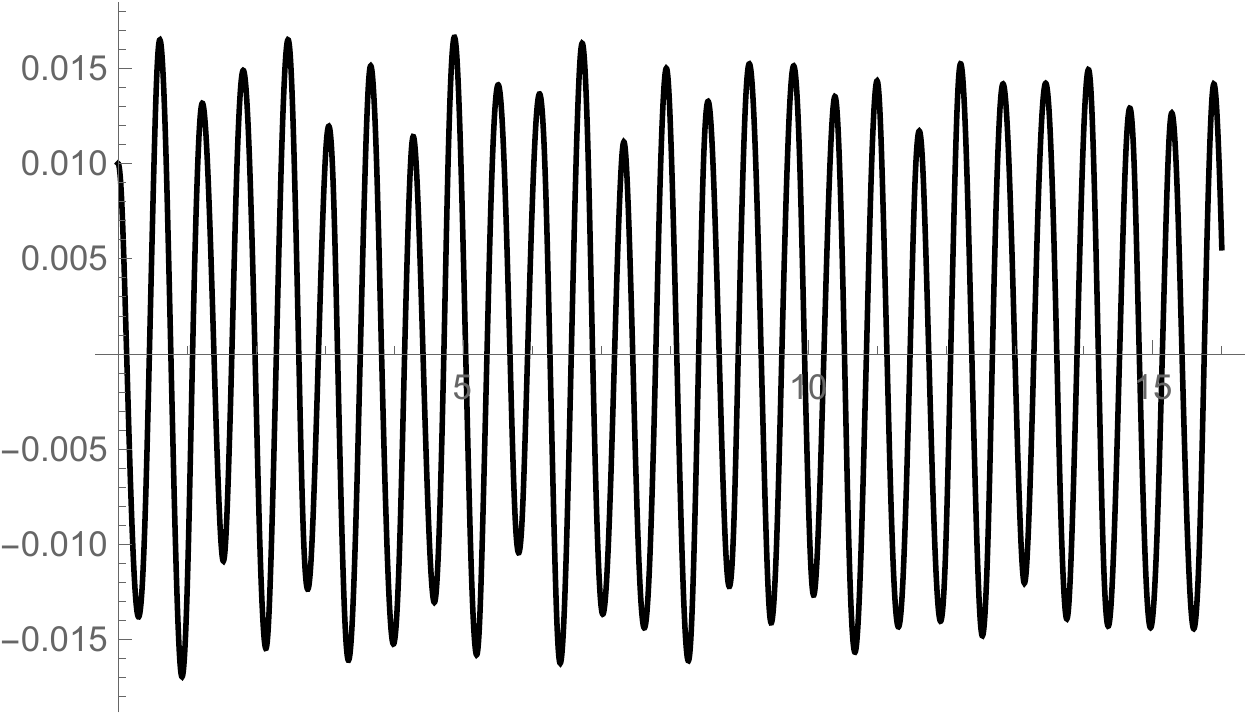}
\;
\includegraphics[scale=0.31]{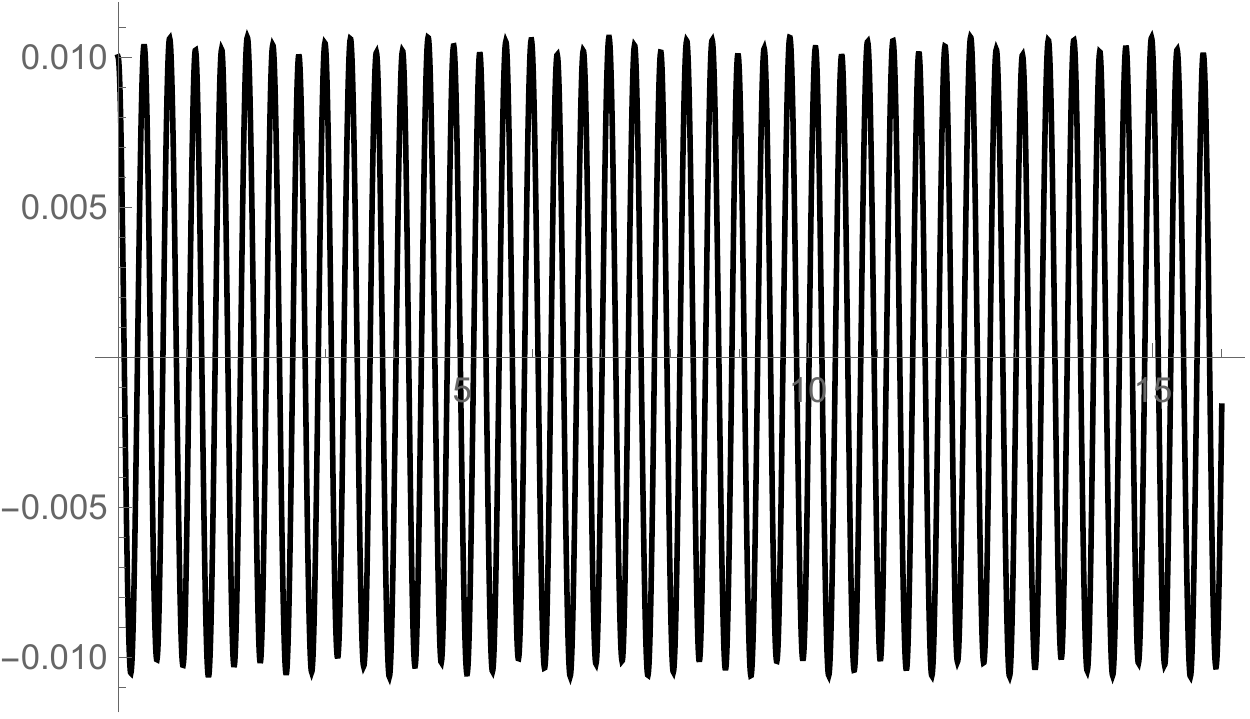}
\vspace{0.3cm}
\\
\includegraphics[scale=0.31]{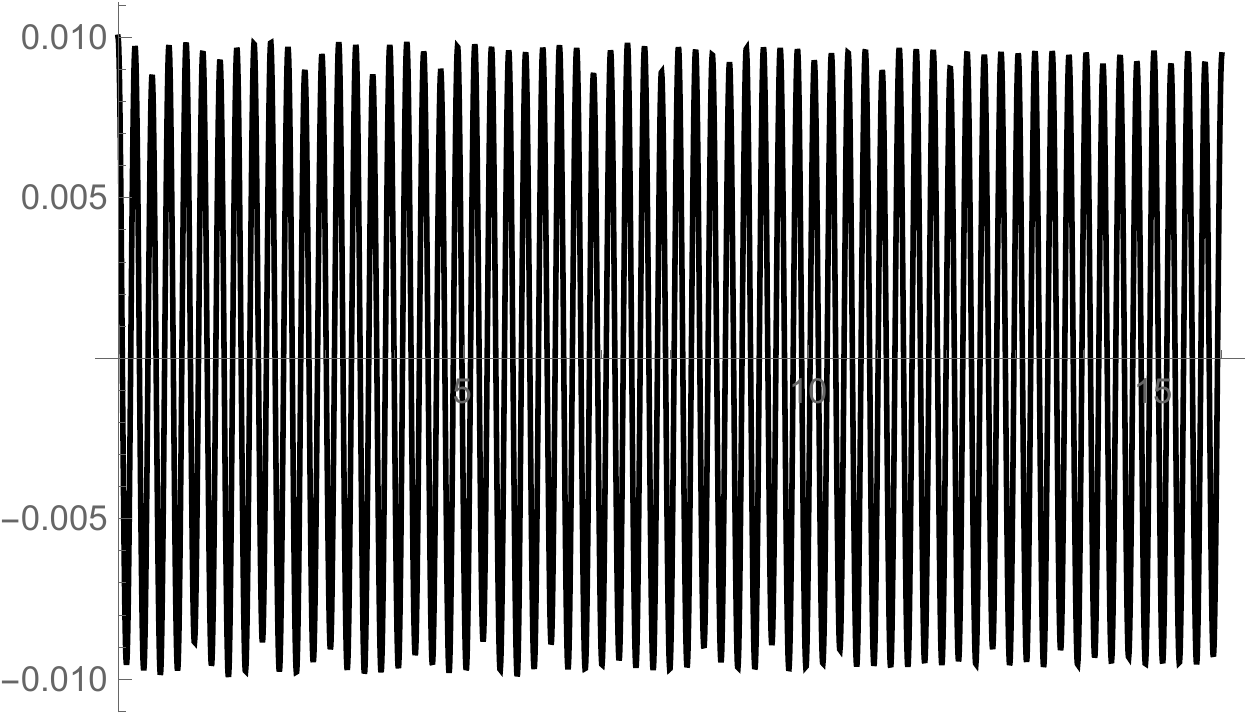}
\;
\includegraphics[scale=0.31]{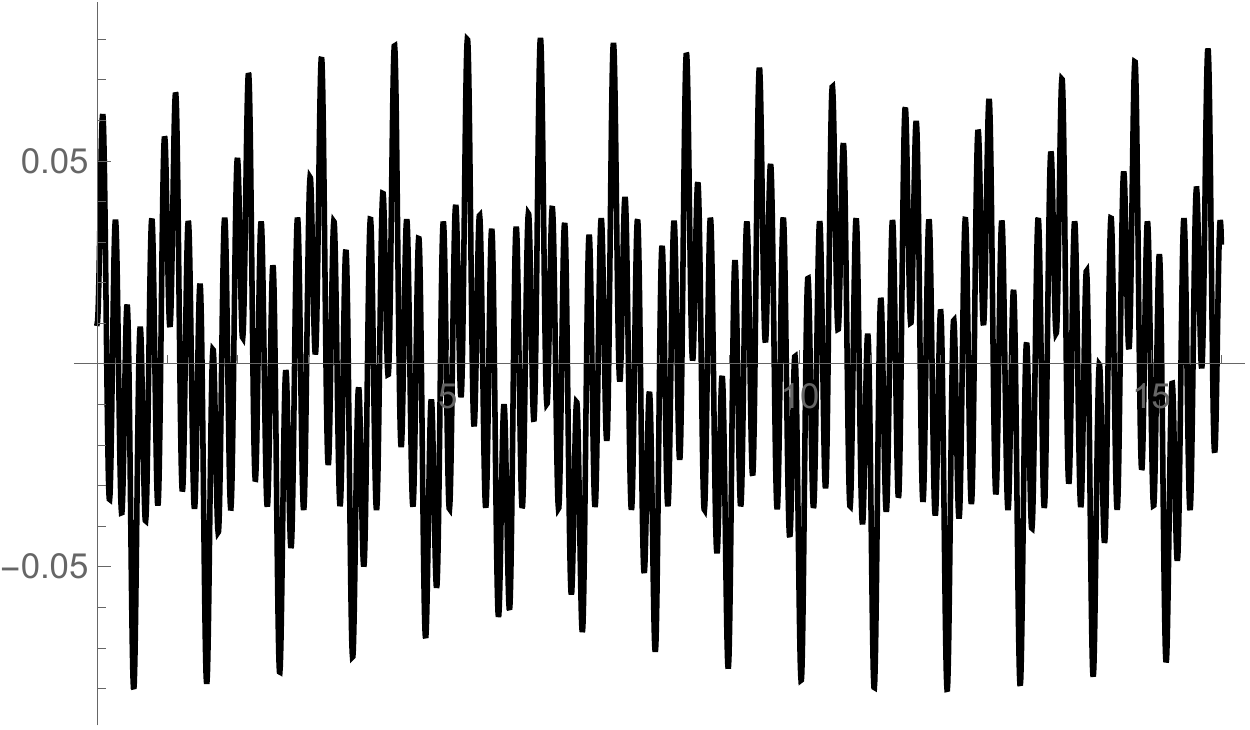}
\;
\includegraphics[scale=0.31]{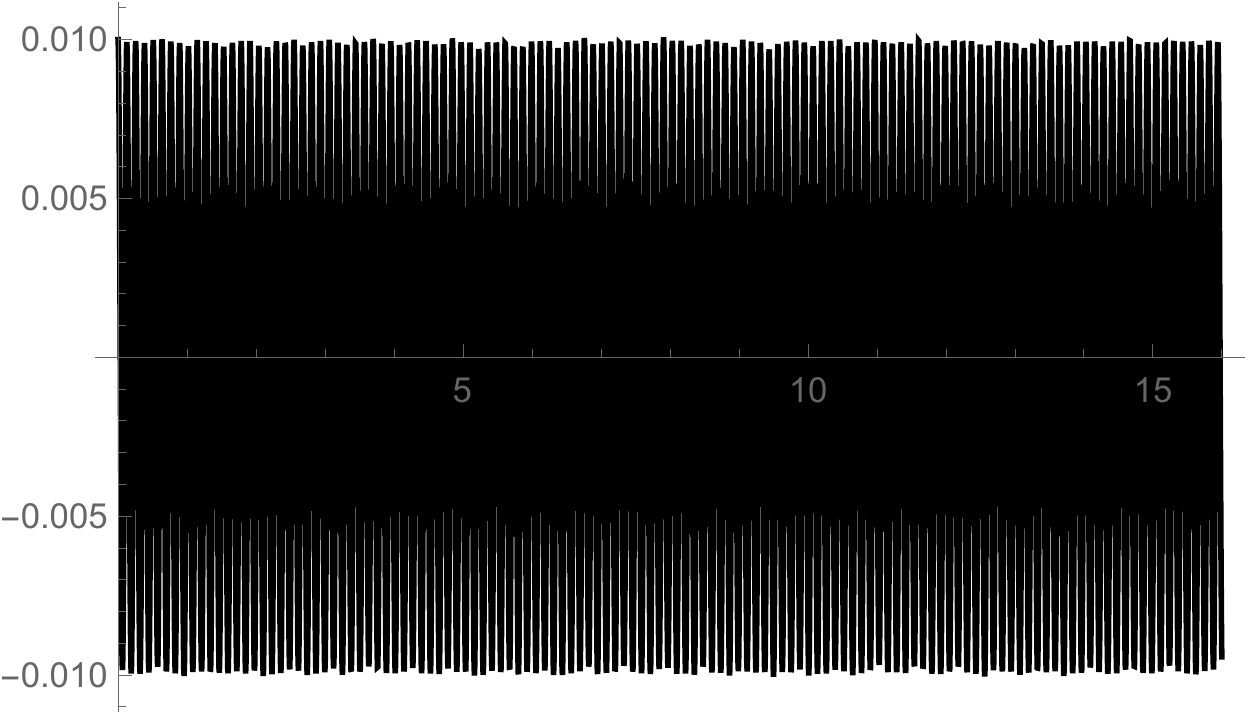}
\;
\includegraphics[scale=0.31]{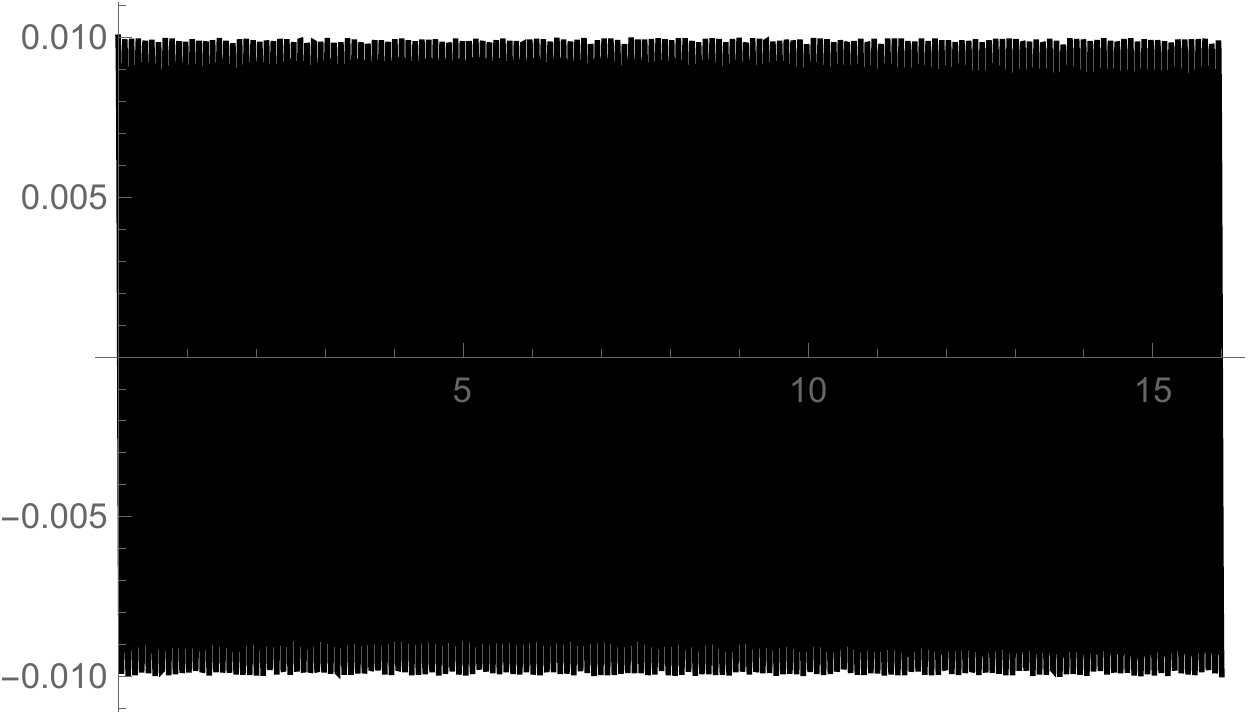}
\vspace{0.3cm}
\\
\includegraphics[scale=0.31]{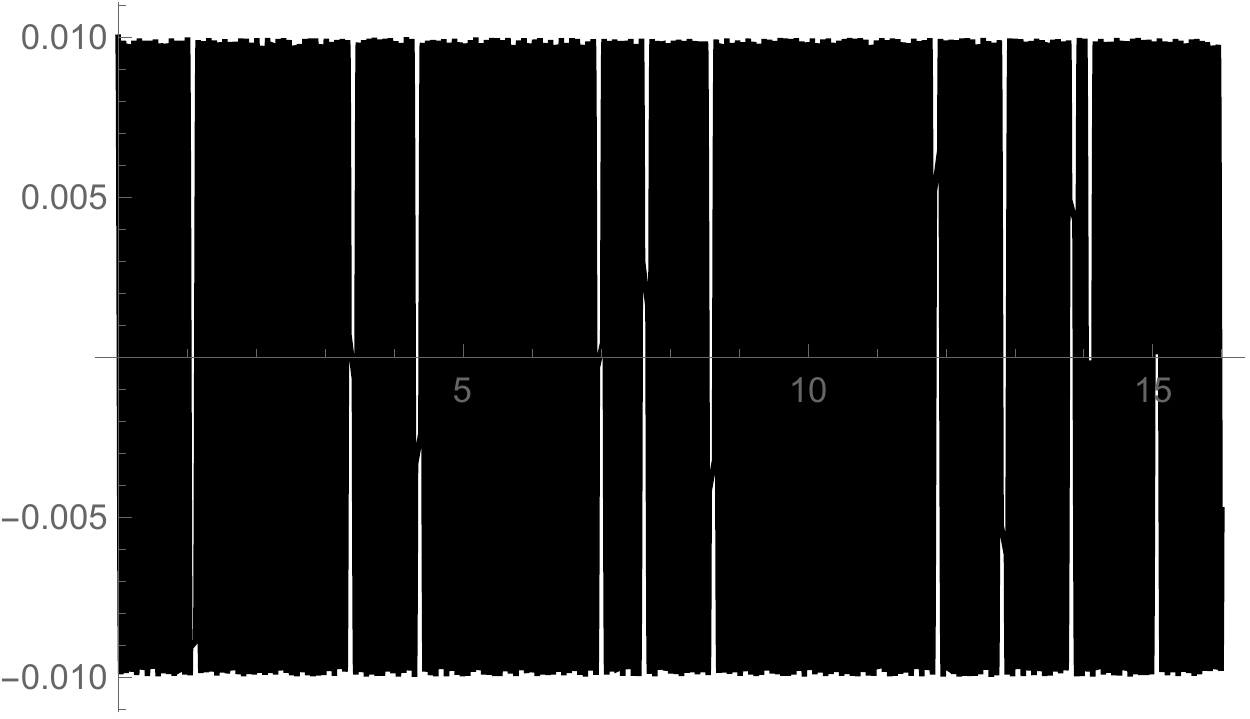}
\;
\includegraphics[scale=0.31]{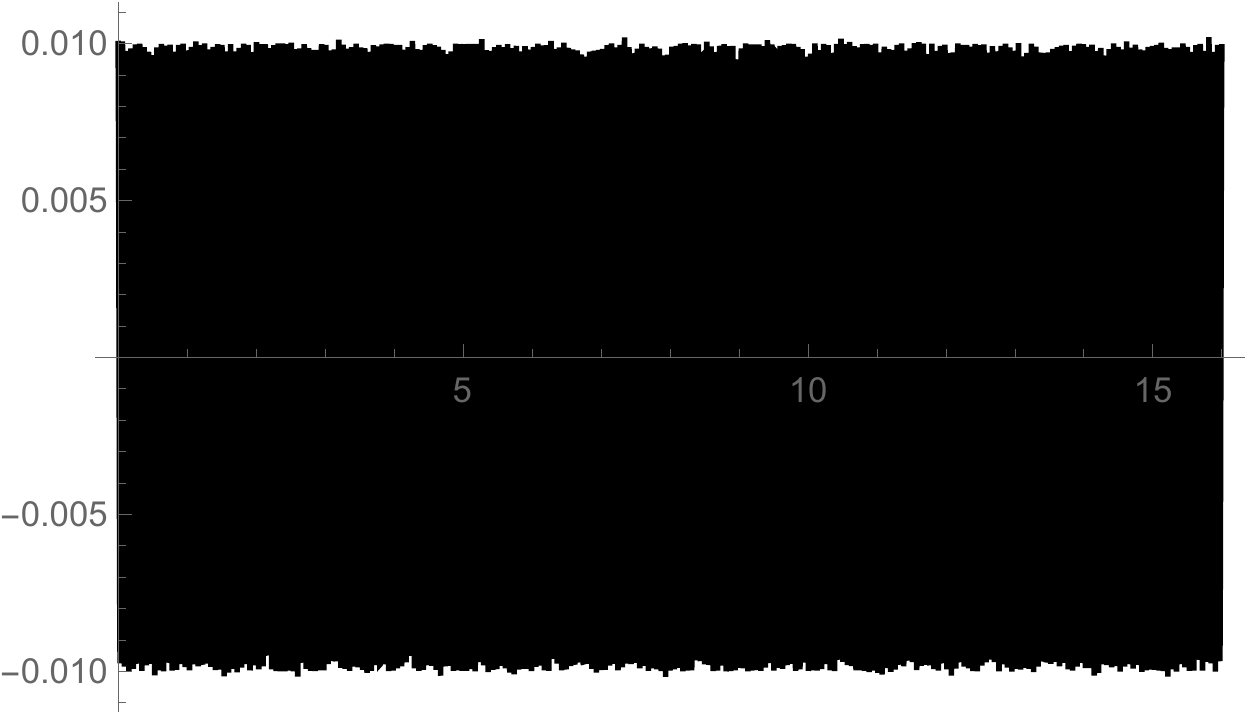}
\;
\includegraphics[scale=0.31]{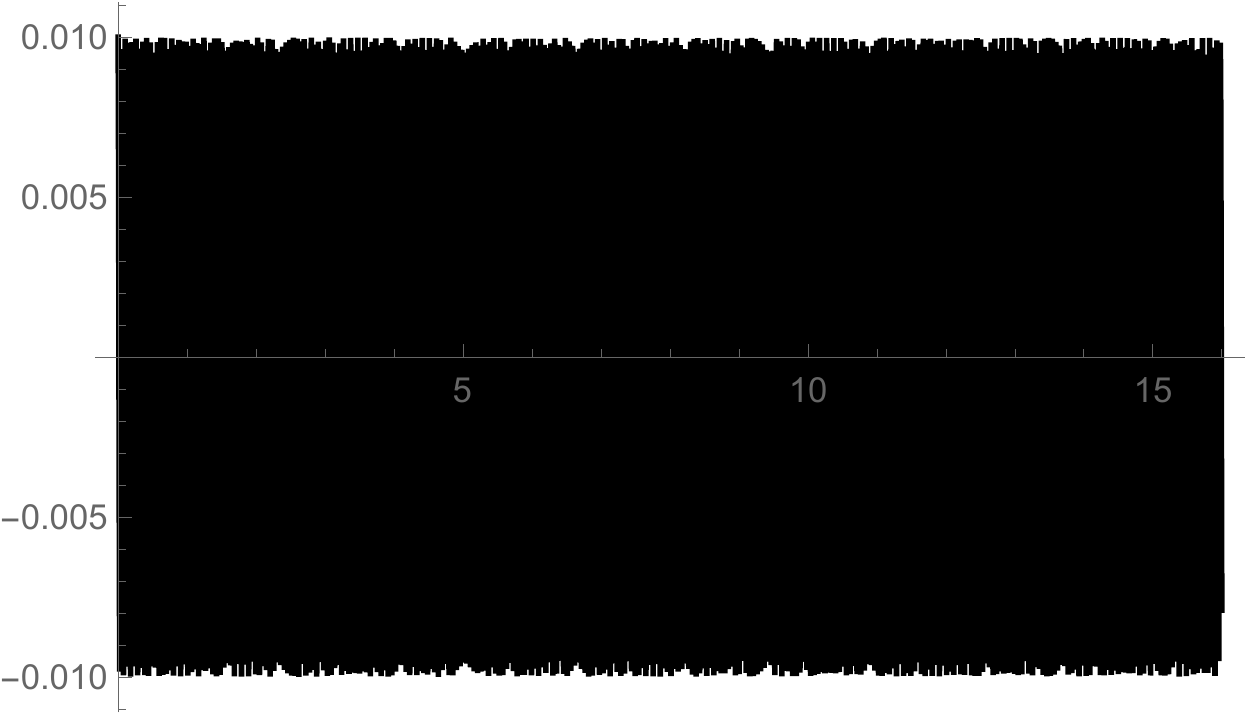}
\;
\includegraphics[scale=0.31]{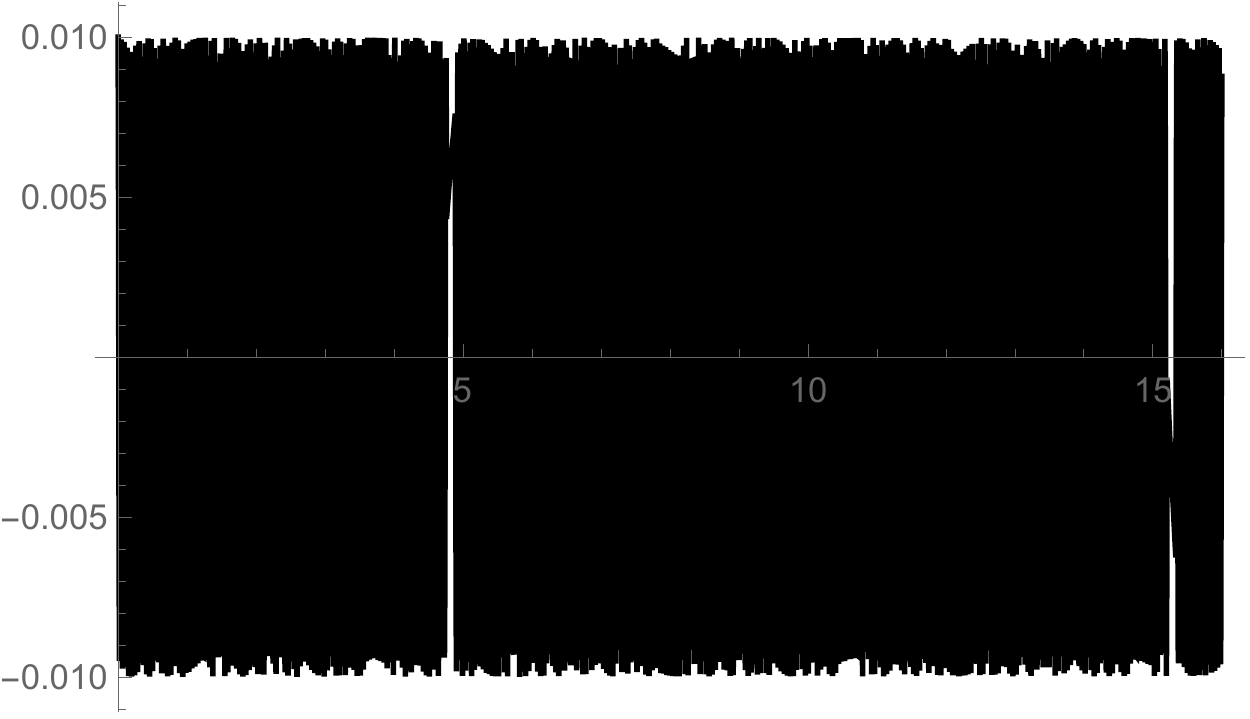}
\vspace{0.3 cm}
\caption{The plots of $\varphi_1^{12}, \ldots, \varphi_{12}^{12}$ for \eqref{firstattempt}, with $u_0(x)=6 \sin (2x) + 0.01 \sum_{n \neq 2} \sin (nx)$.}
\label{plots1cubo}
\vspace{0.3 cm}
\includegraphics[scale=0.31]{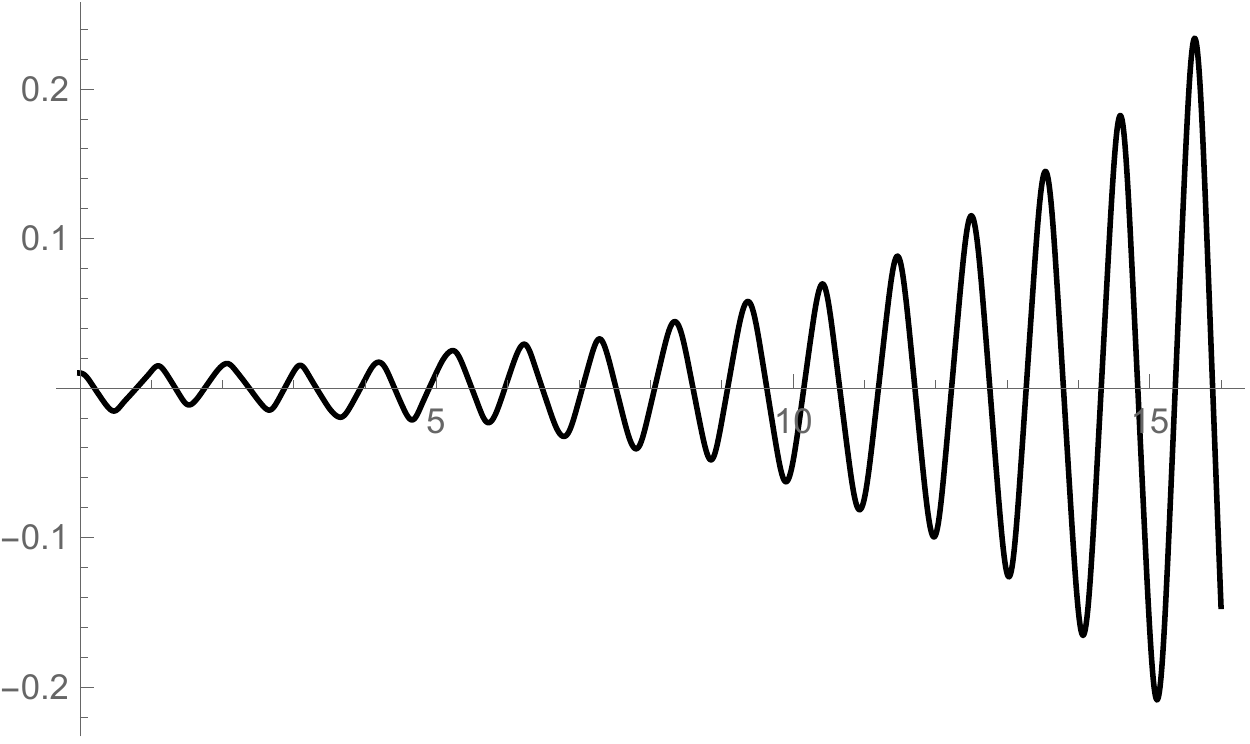}
\;
\includegraphics[scale=0.31]{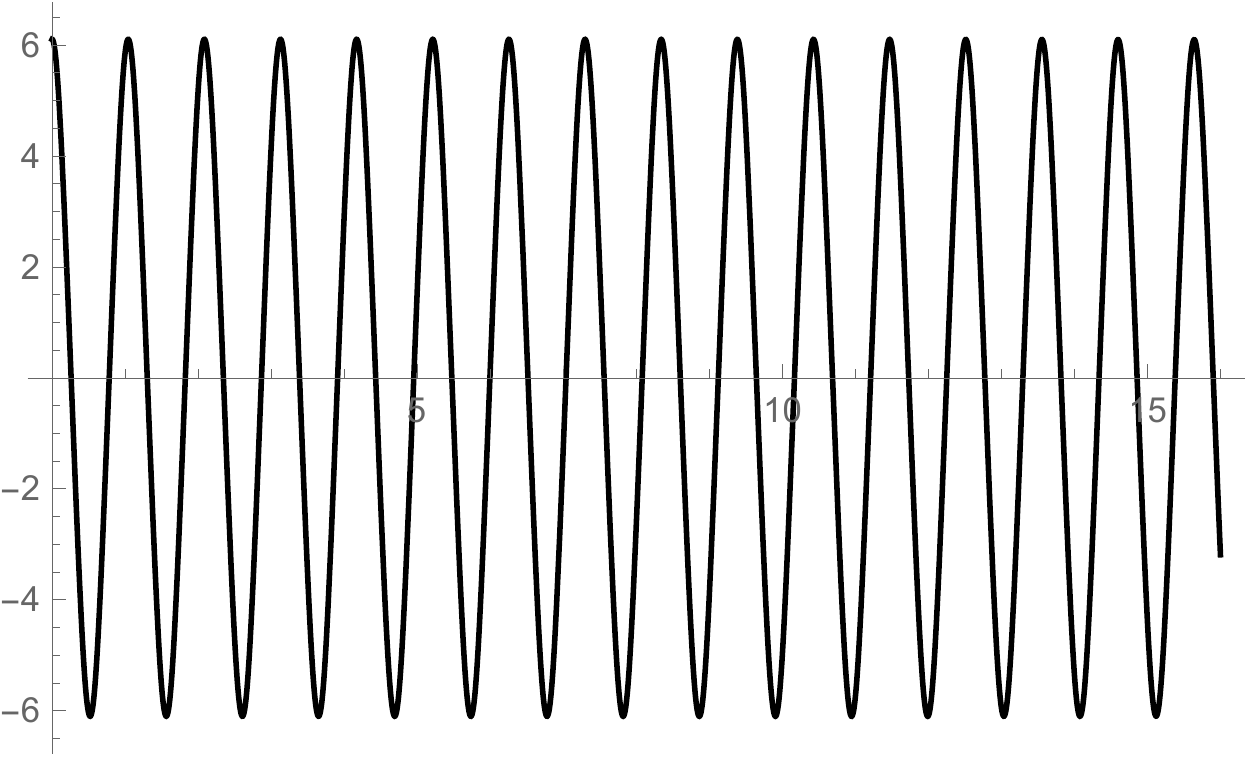}
\;
\includegraphics[scale=0.31]{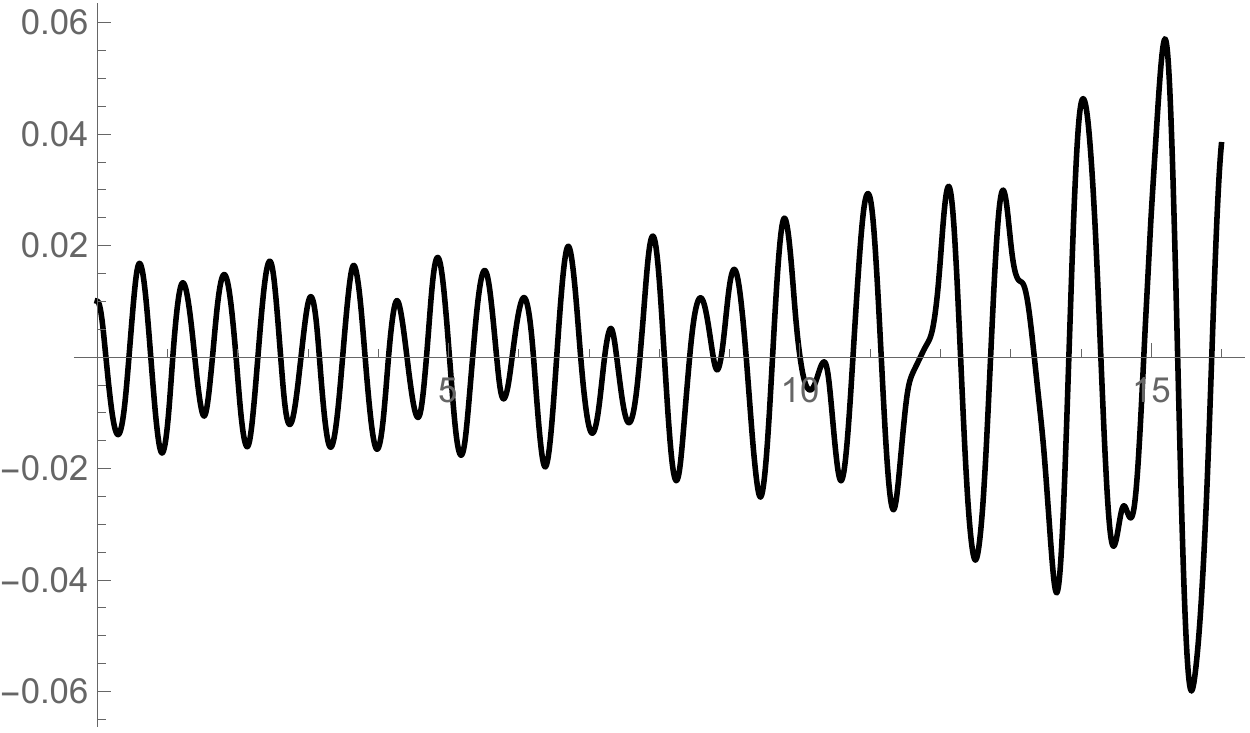}
\;
\includegraphics[scale=0.31]{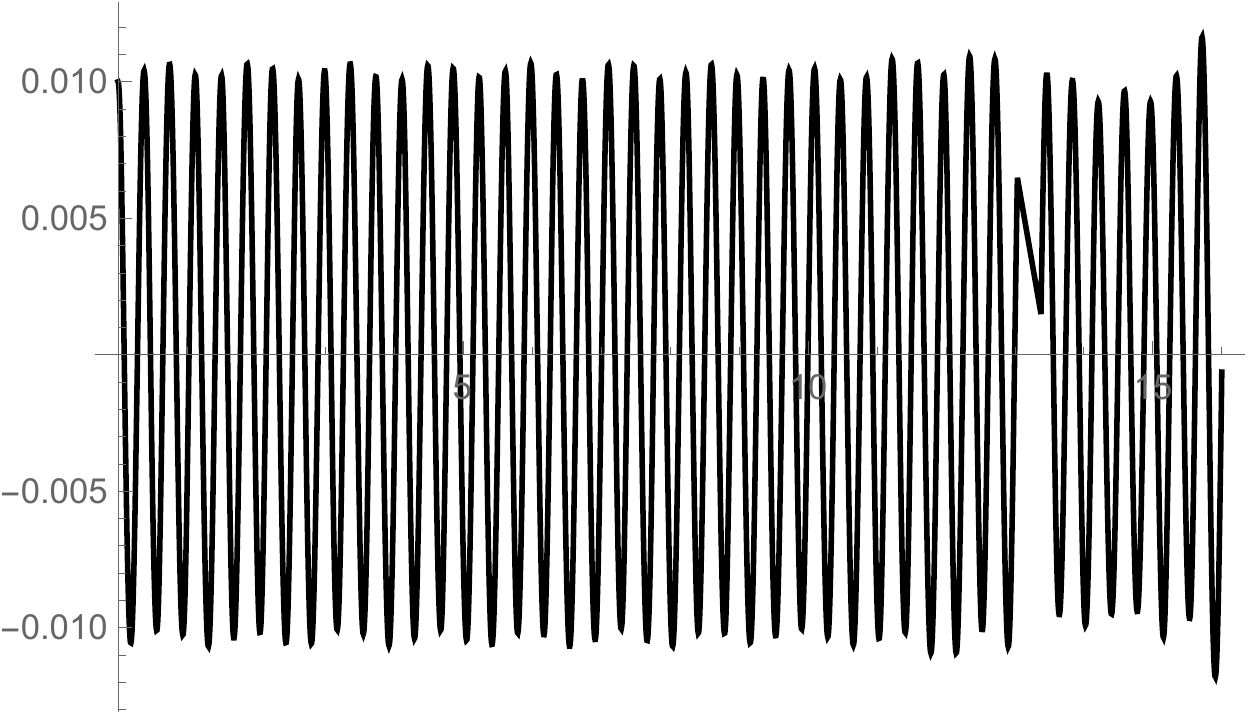}
\vspace{0.3cm}
\\
\includegraphics[scale=0.31]{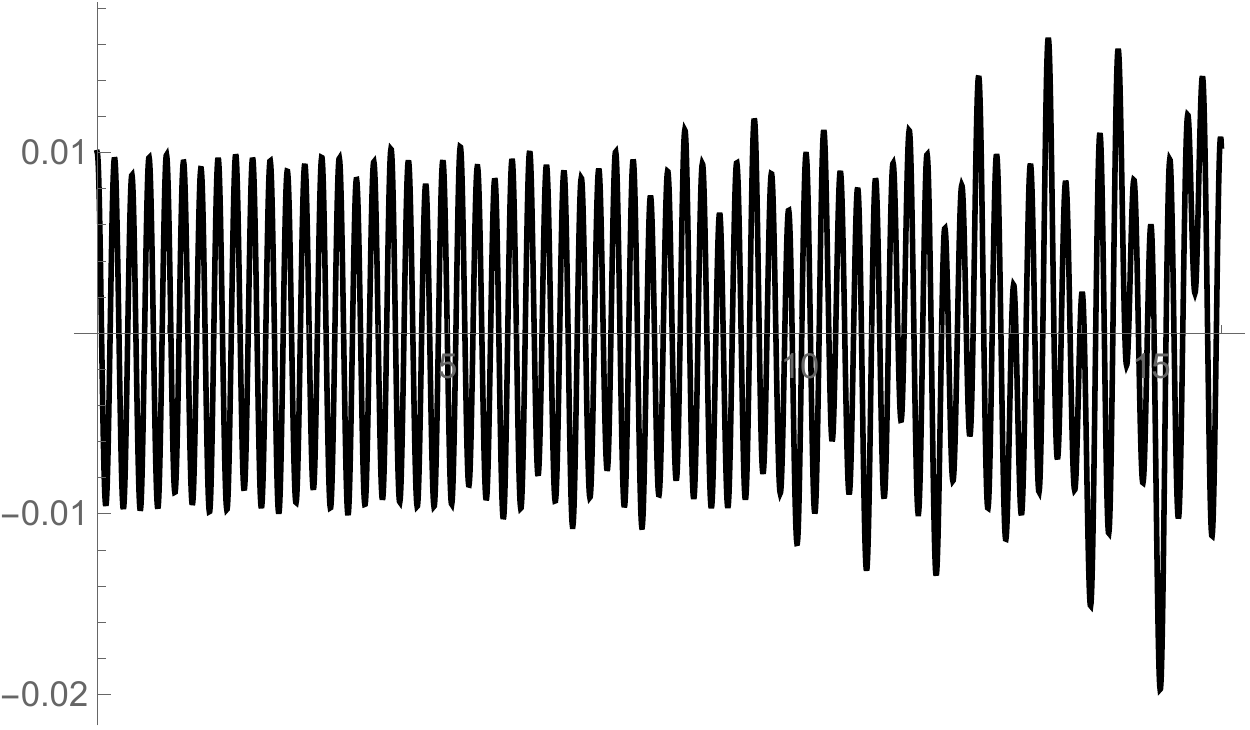}
\;
\includegraphics[scale=0.31]{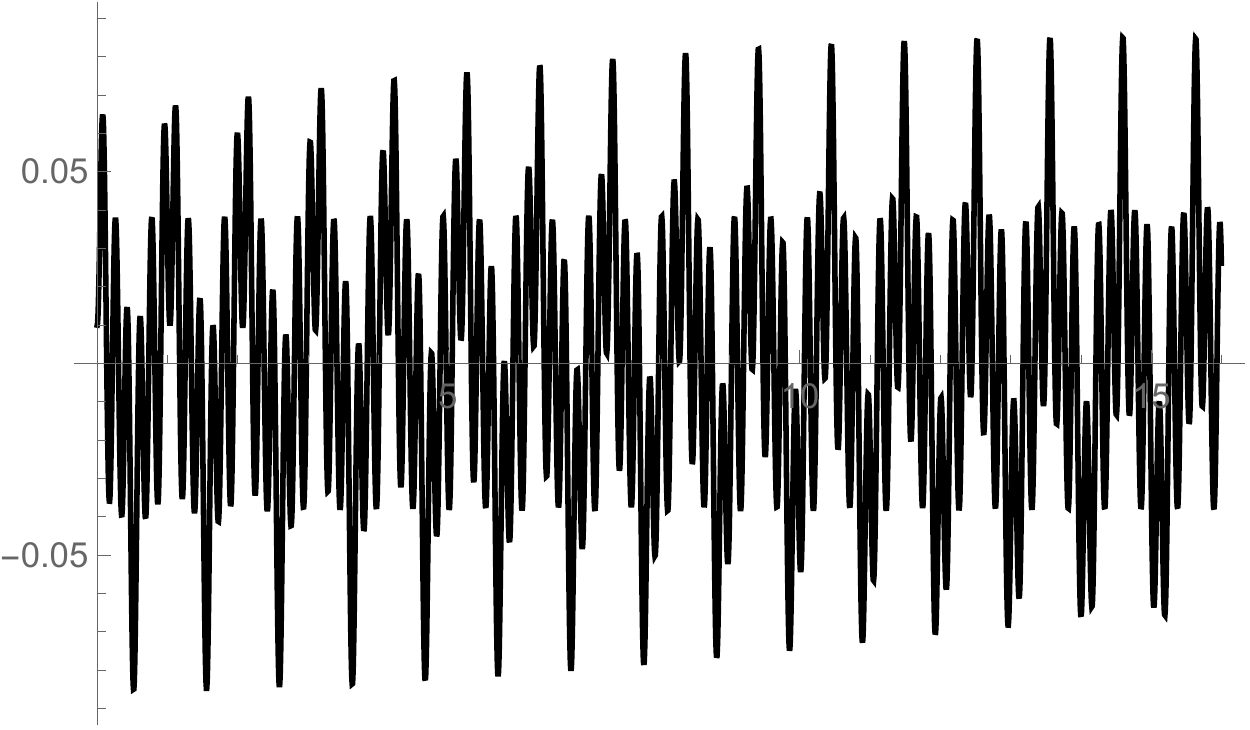}
\;
\includegraphics[scale=0.31]{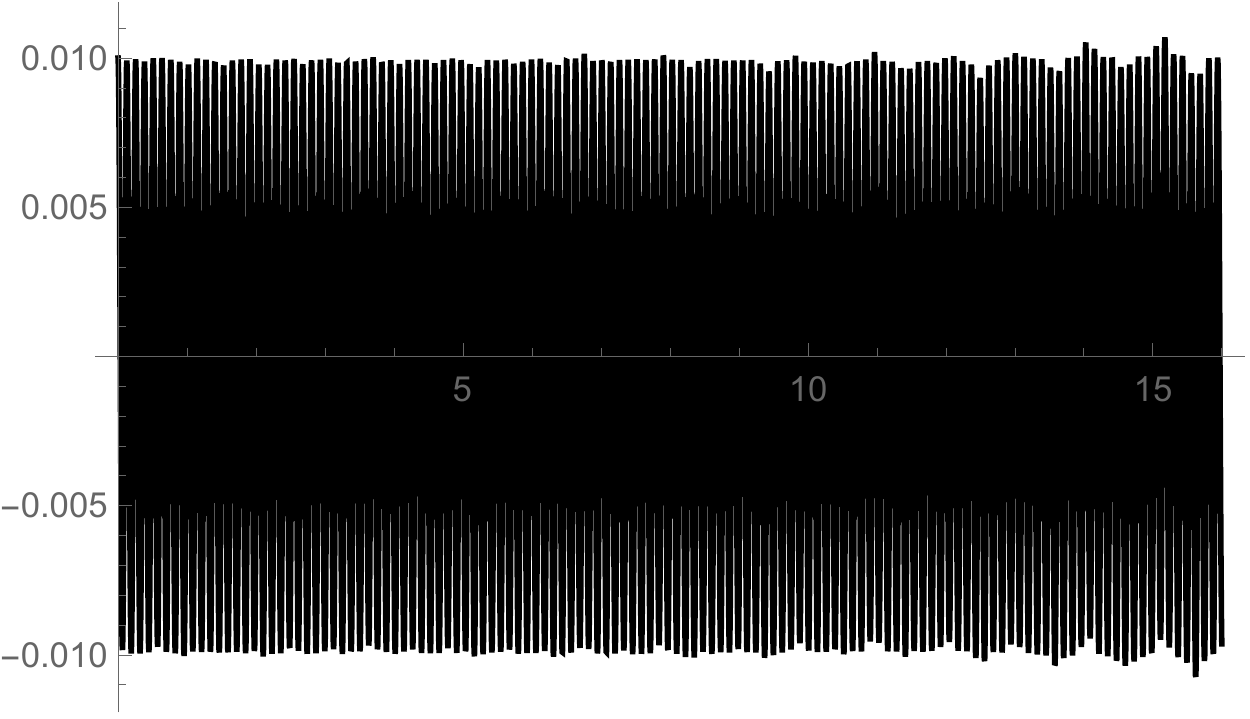}
\;
\includegraphics[scale=0.31]{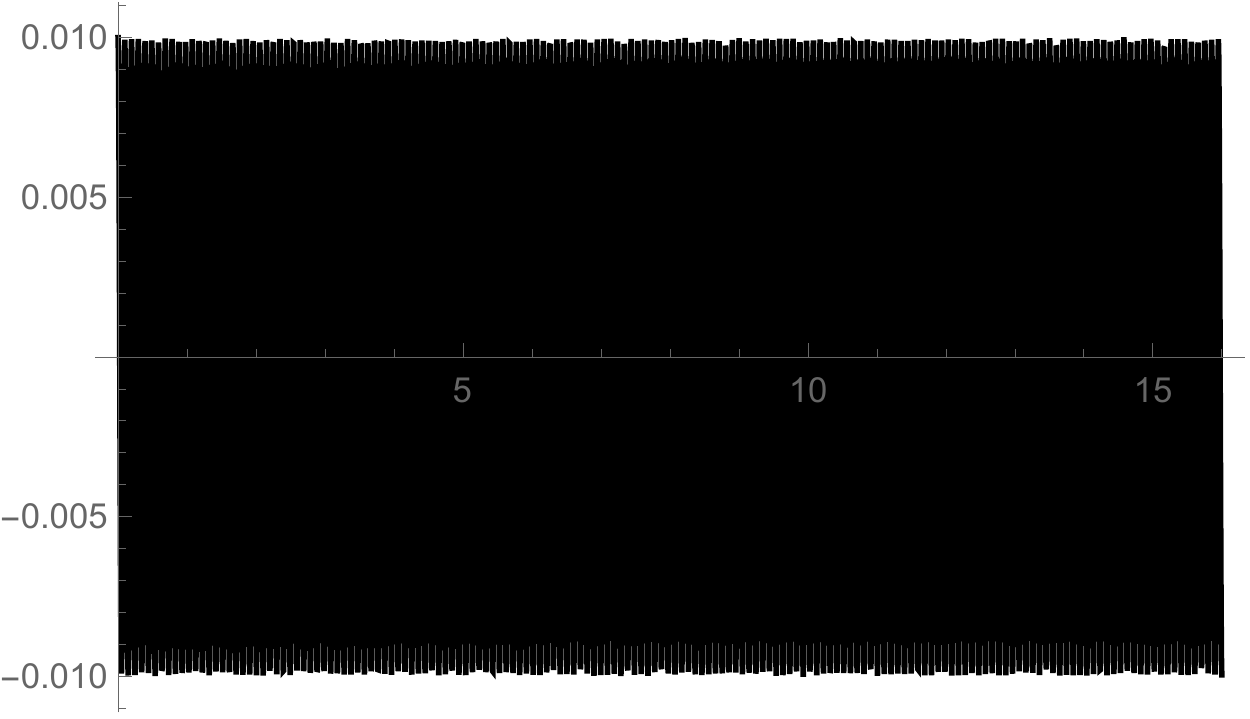}
\vspace{0.3 cm}
\caption{The plots of $\varphi_1^{12}, \ldots, \varphi_{8}^{12}$ (the last four modes not being significant) for problem \eqref{firstattempt}, with $u_0(x)=6.1 \sin (2x) + 0.01 \sum_{n \neq 2} \sin (nx)$.}
\label{plots2cubo}
\vspace{0.3 cm}
\includegraphics[scale=0.31]{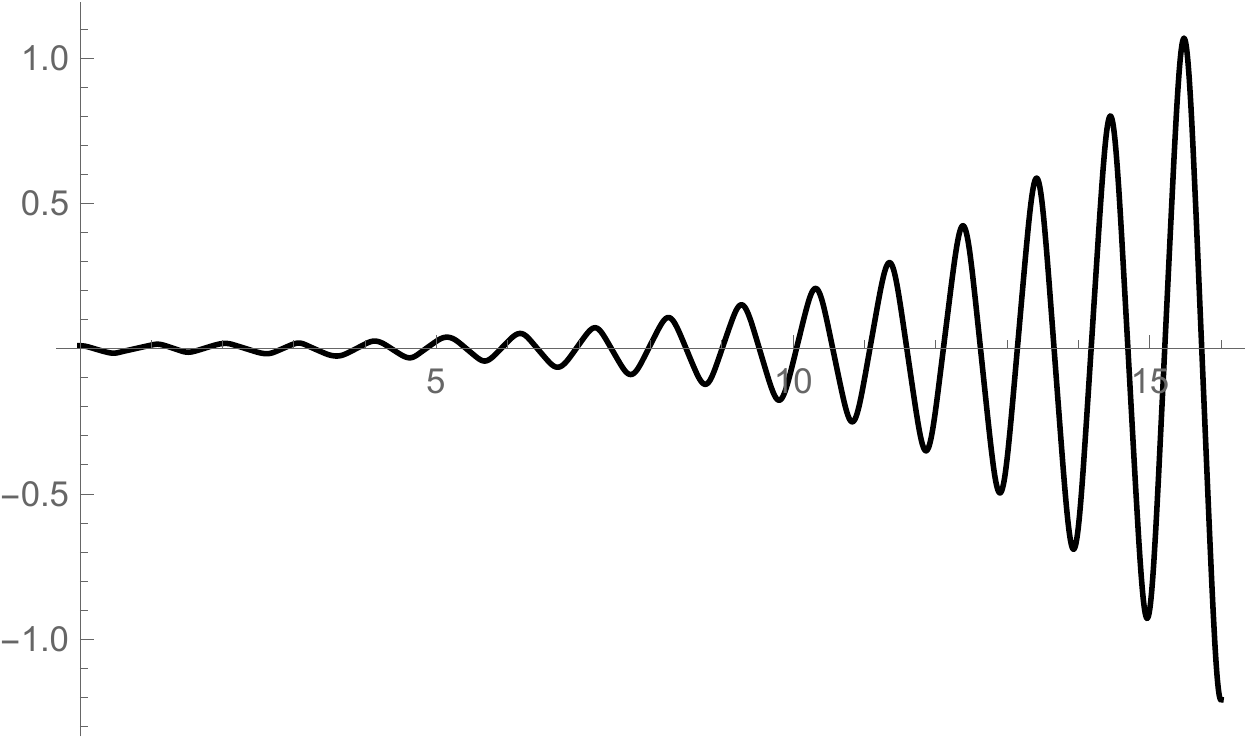}
\;
\includegraphics[scale=0.31]{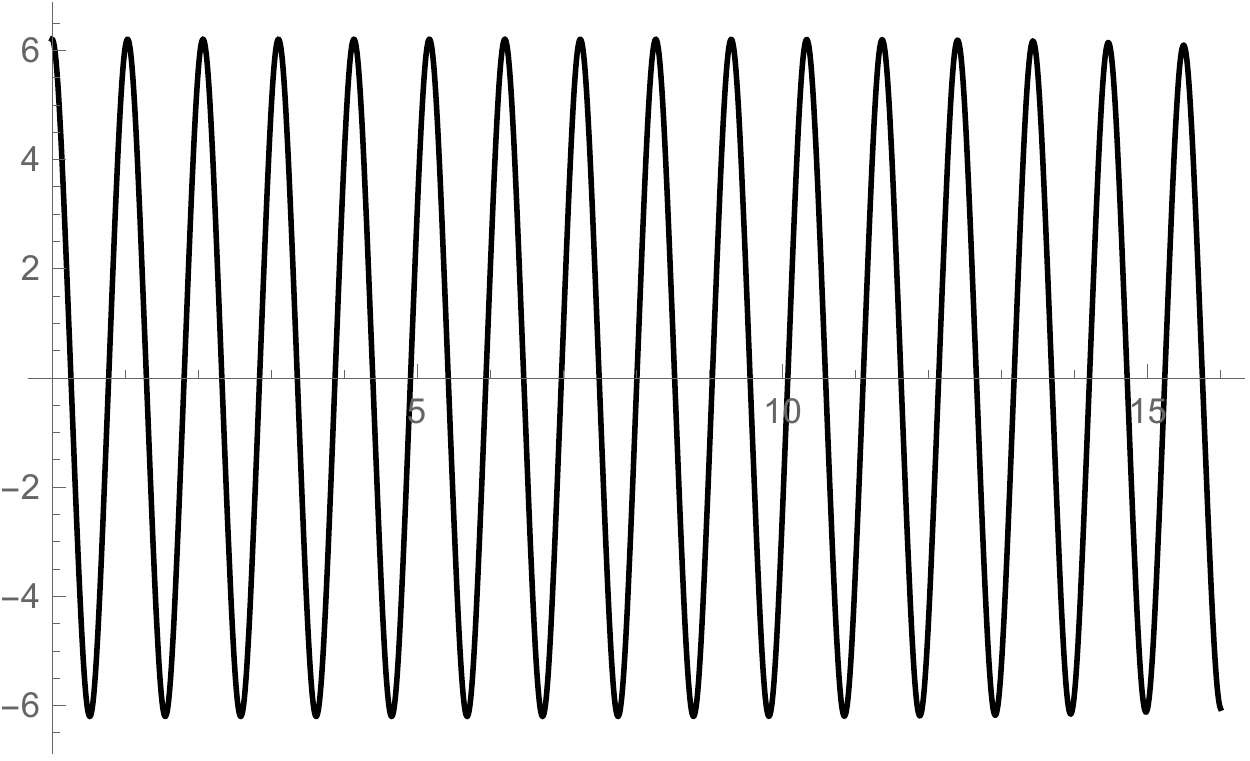}
\;
\includegraphics[scale=0.31]{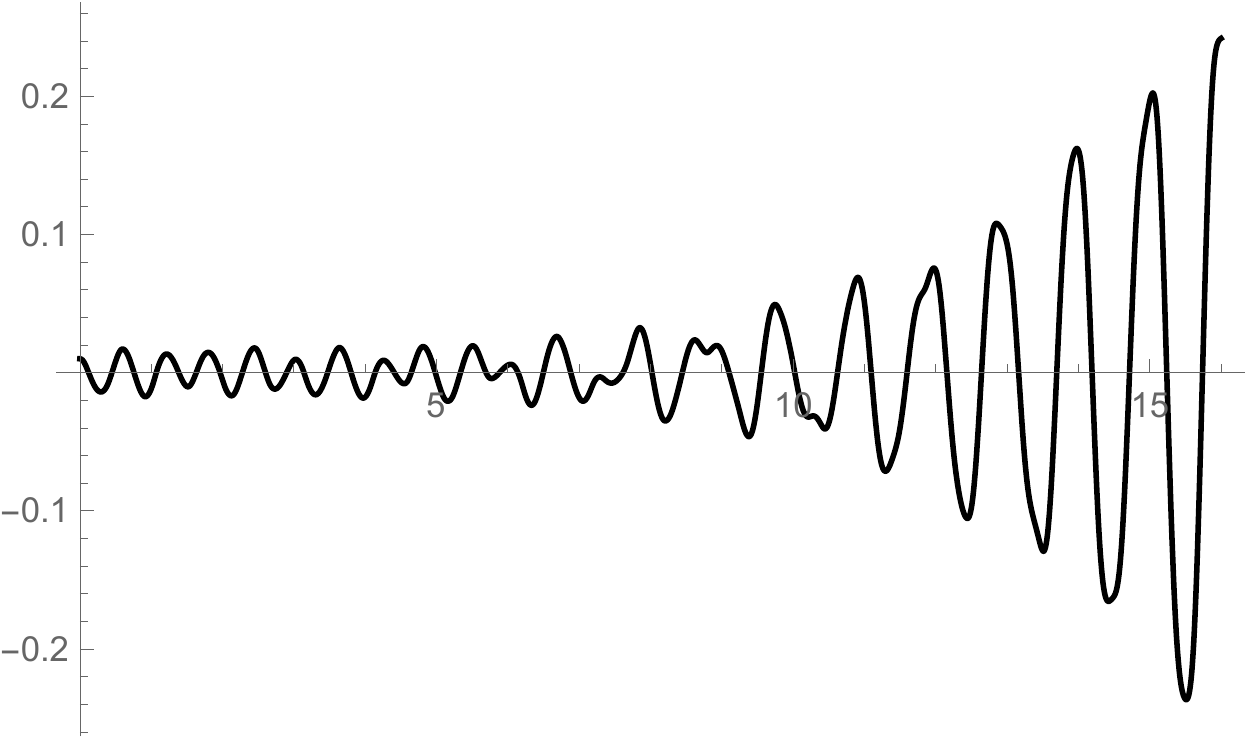}
\;
\includegraphics[scale=0.31]{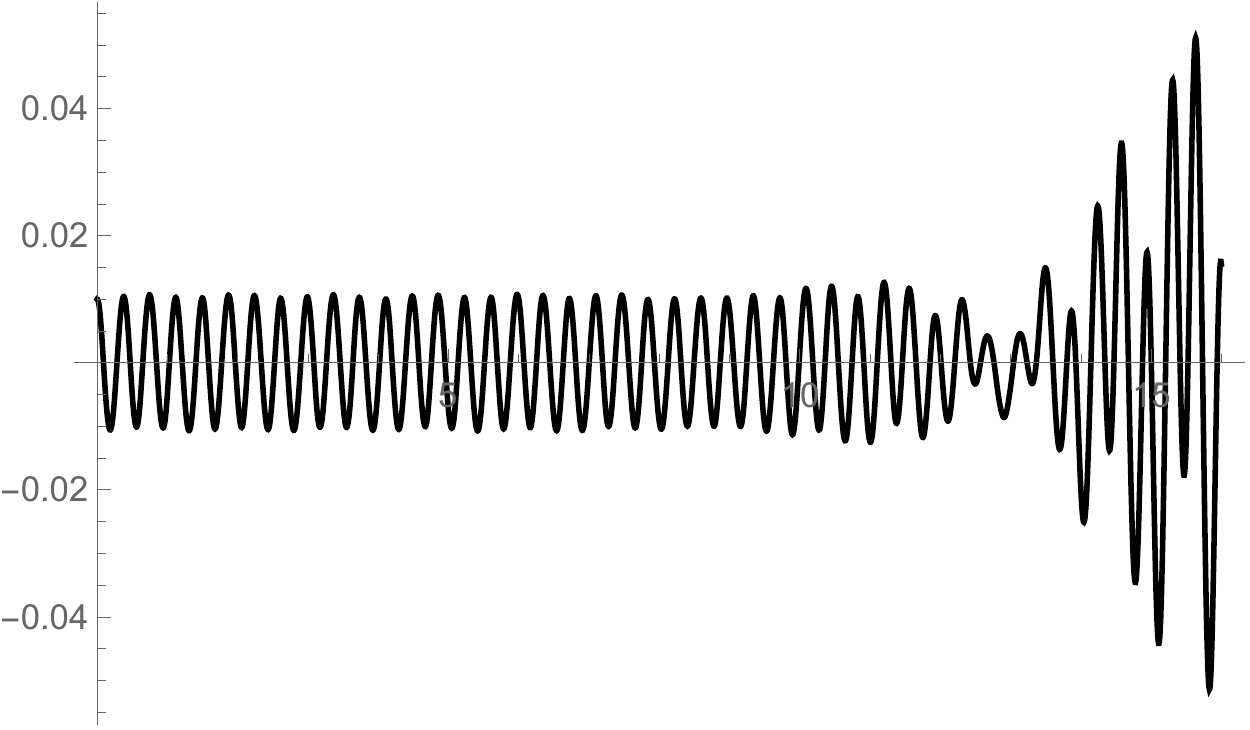}
\vspace{0.3cm}
\\
\includegraphics[scale=0.31]{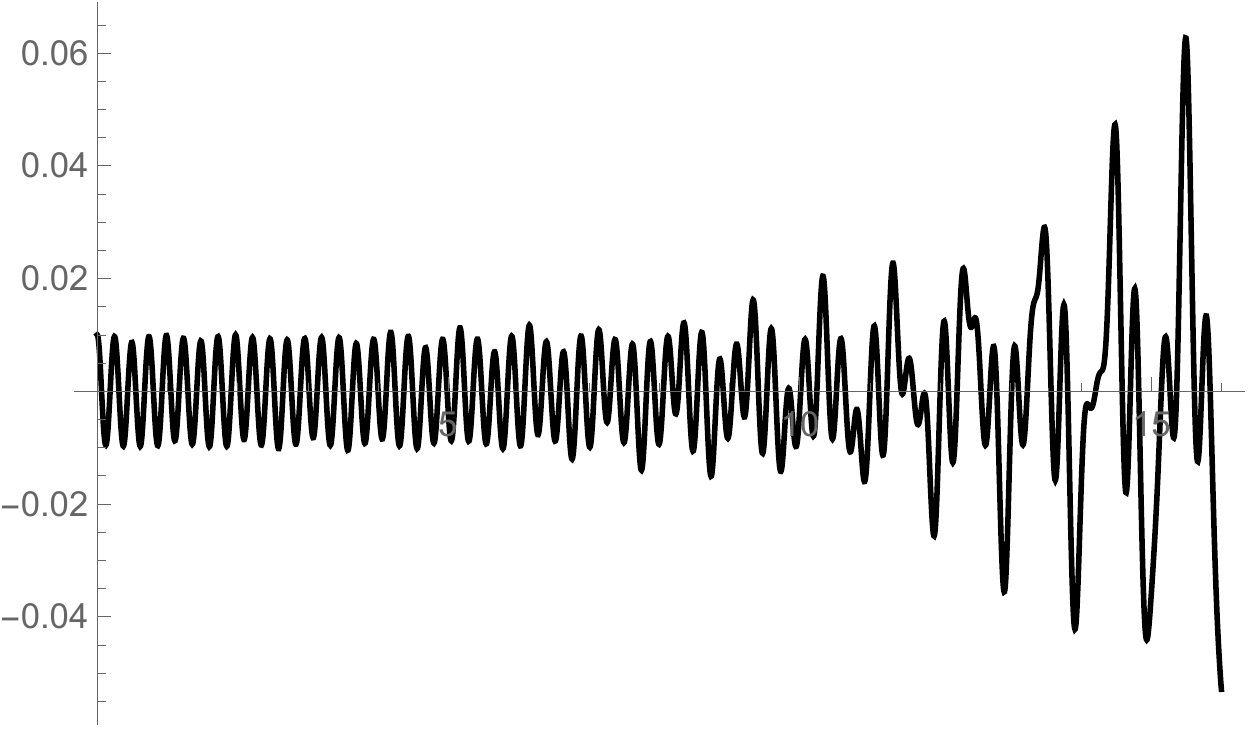}
\;
\includegraphics[scale=0.31]{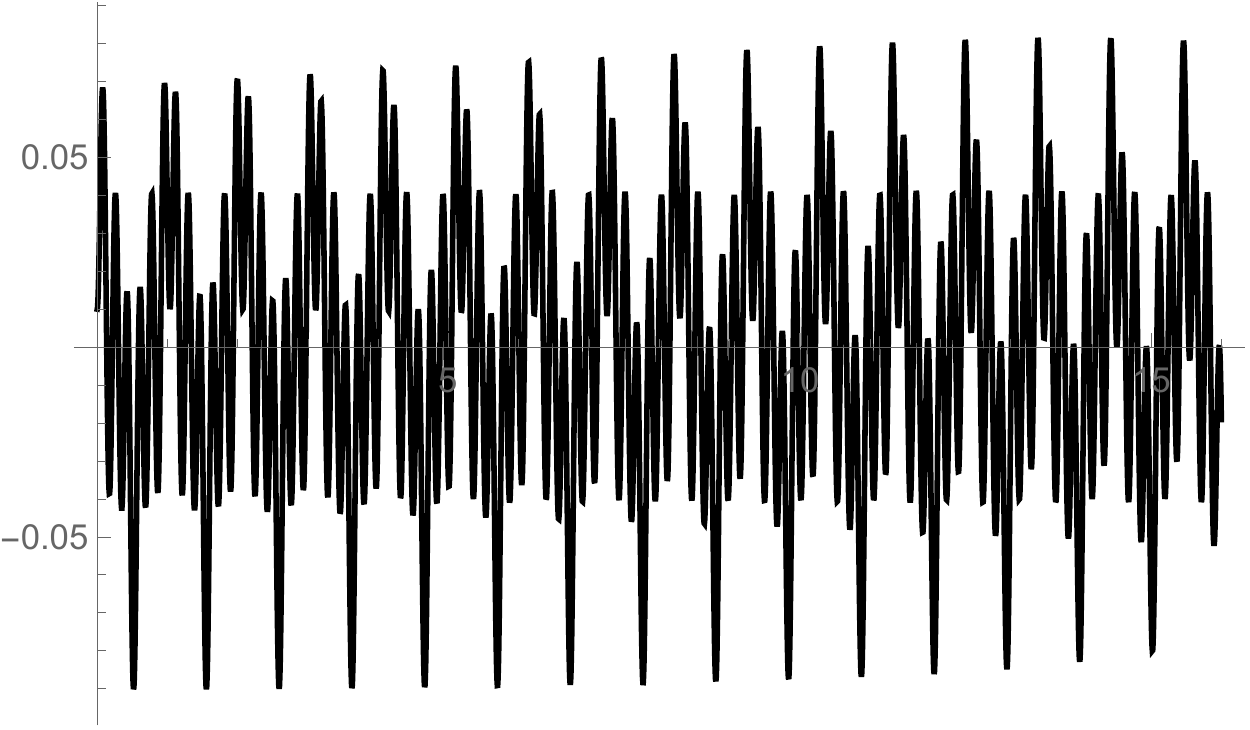}
\;
\includegraphics[scale=0.31]{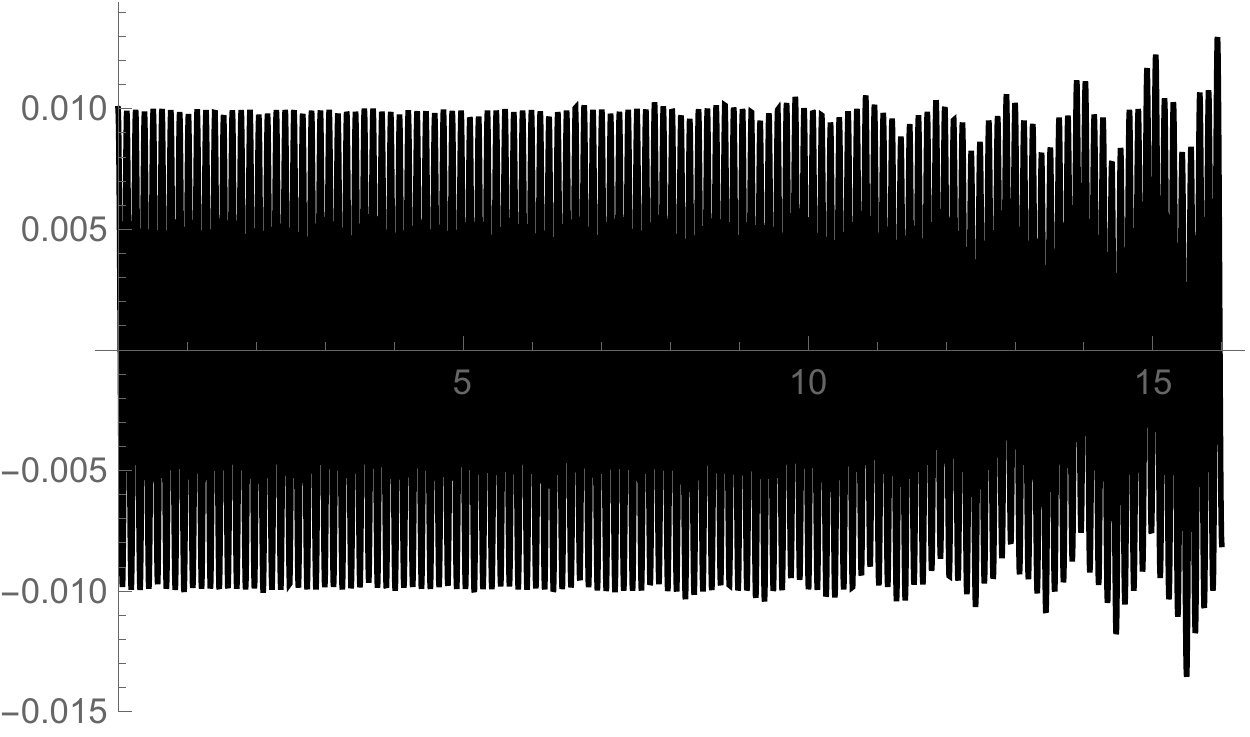}
\;
\includegraphics[scale=0.31]{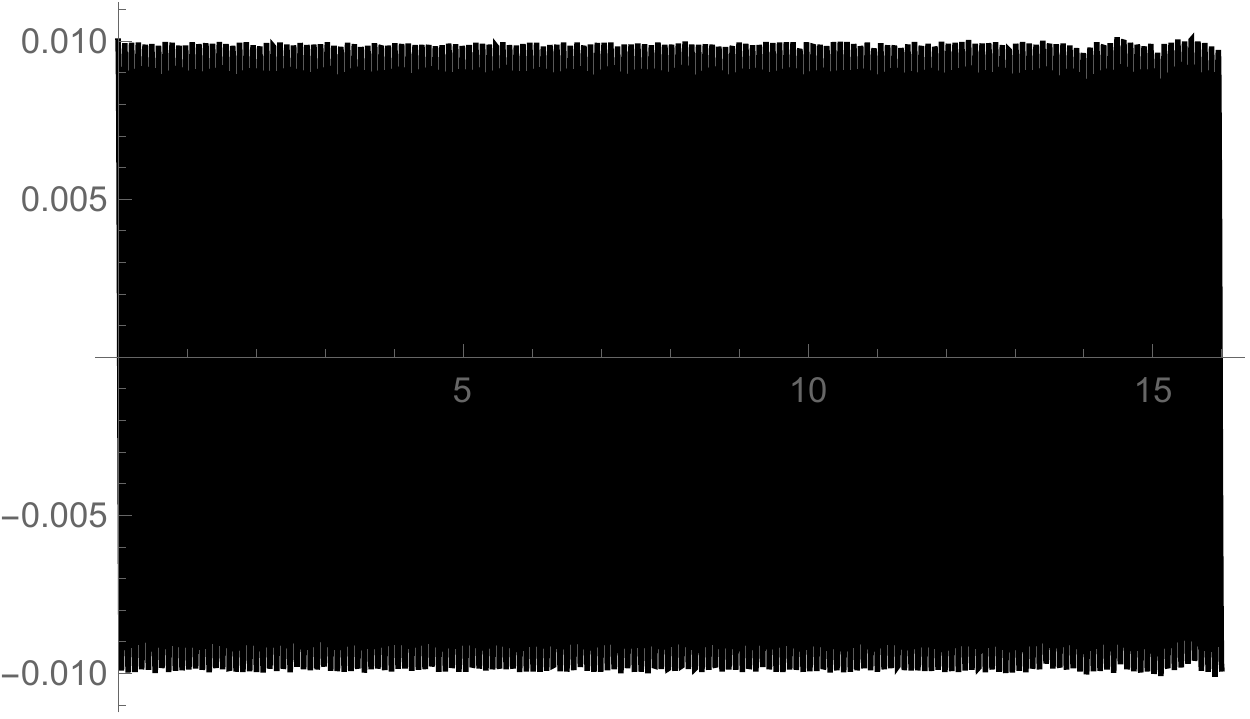}
\caption{We show the plots of $\varphi_1^{12}, \ldots, \varphi_{8}^{12}$ (again, the last four modes are not significant) for problem \eqref{firstattempt}, with $u_0(x)=6.2 \sin (2x) + 0.01 \sum_{n \neq 2} \sin (nx)$.}
\label{plots3cubo}
\end{figure}

With reference to the comments after Definition \ref{unstable}, we now show two experiments where the sixth mode is prevailing; here we point out that, under small variations of the amplitude of the prevailing mode, \emph{more} residual modes may fulfill \eqref{finitegrande}, and the picture appears ``catastrophic''. Indeed, in Figure \ref{tuttimodicubo1}, where $\varphi_6^{12}(0)=50.097$, the solution is stable but its modes other than $6, 12$ are showing a significant change in their behaviour. Increasing the initial amplitude to $\varphi_6^{12}(0)=50.098$, it may be observed that only the first mode enters the setting of Definition \ref{unstable} (we do not show the picture) while, when reaching $\varphi_6^{12}(0)=50.099$, all the modes which previously absorbed energy grow substantially (Figure \ref{tuttimodicubo2}).

\begin{figure}[!ht]
\center
\includegraphics[scale=0.31]{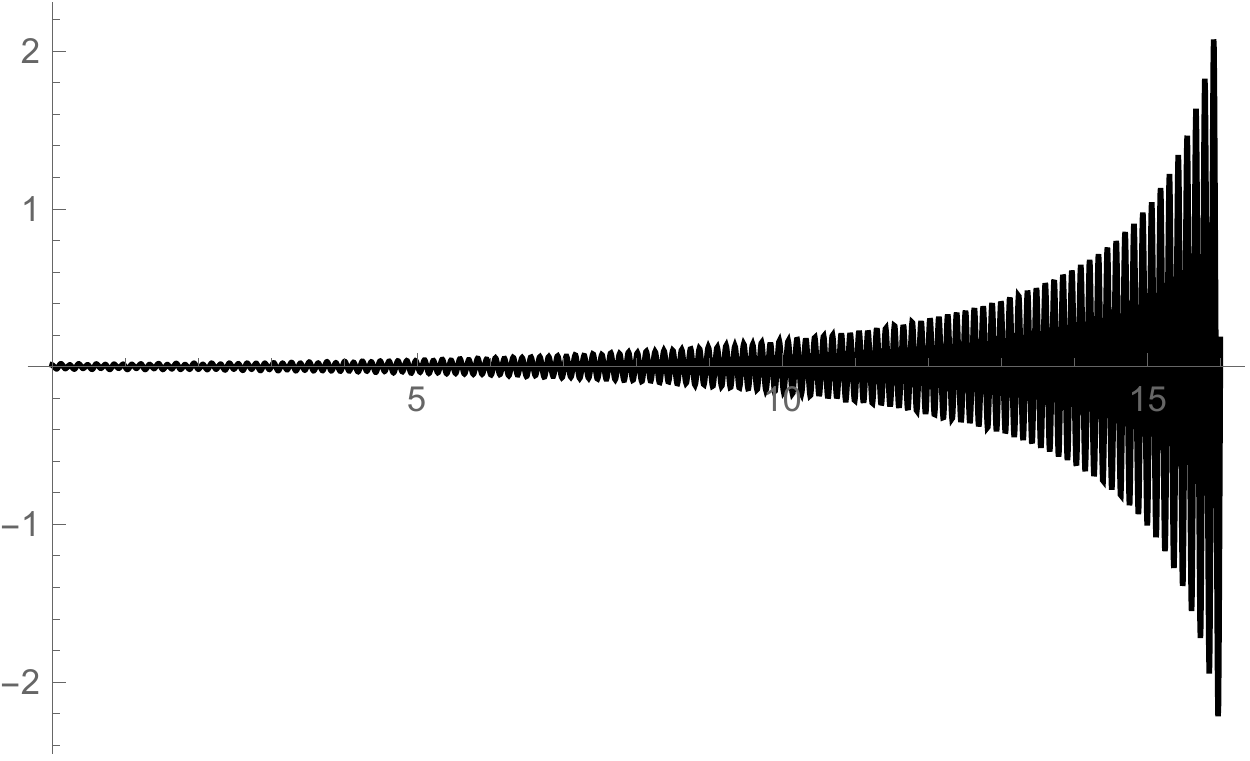}
\;
\includegraphics[scale=0.31]{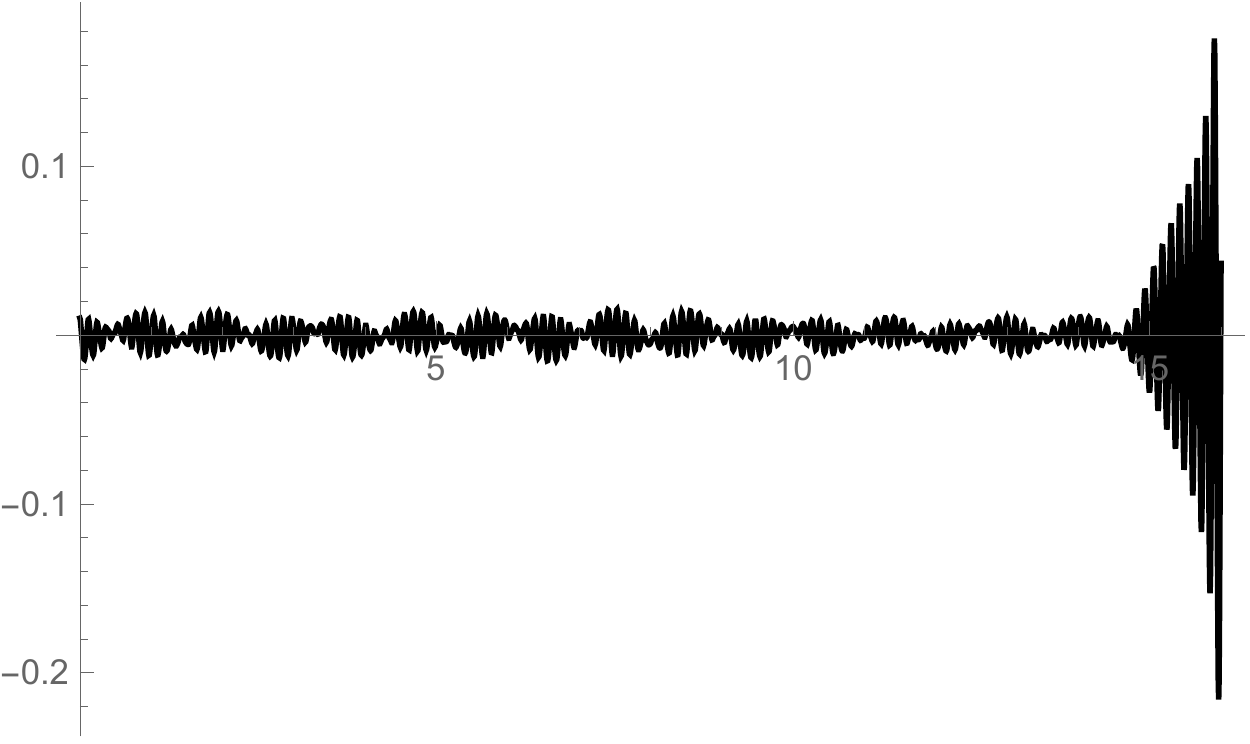}
\;
\includegraphics[scale=0.31]{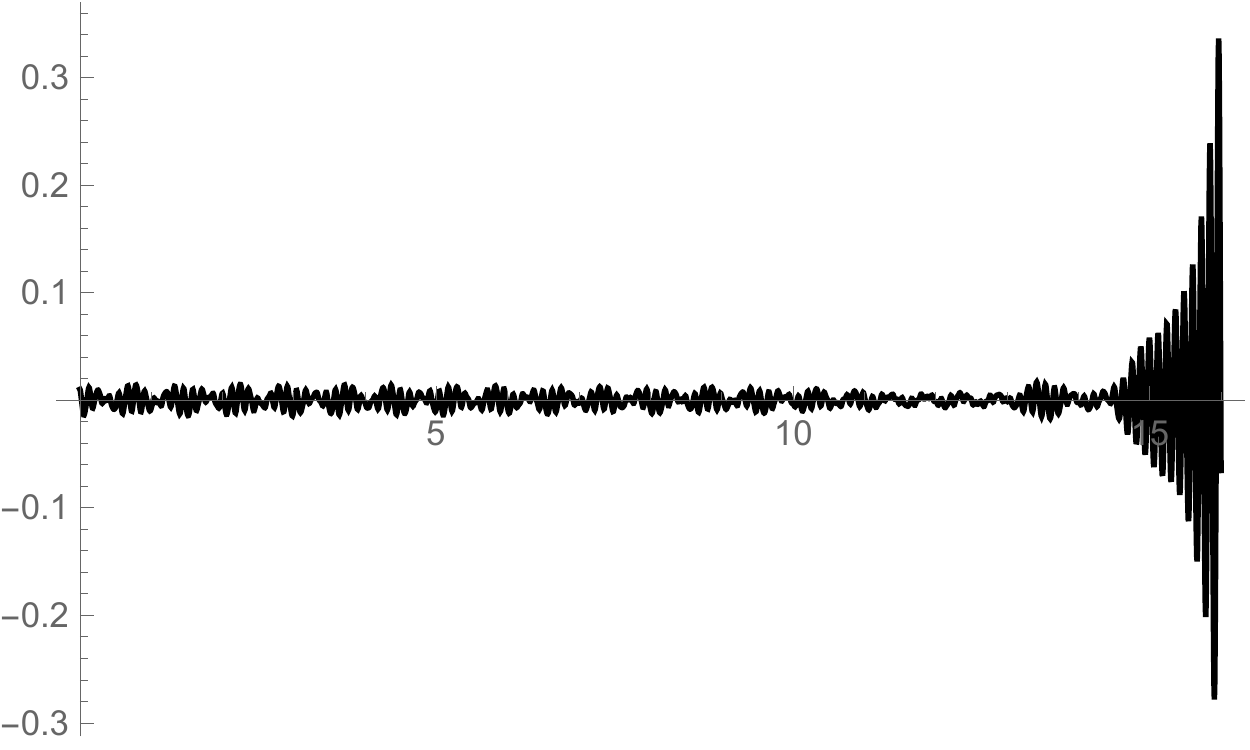}
\;
\includegraphics[scale=0.31]{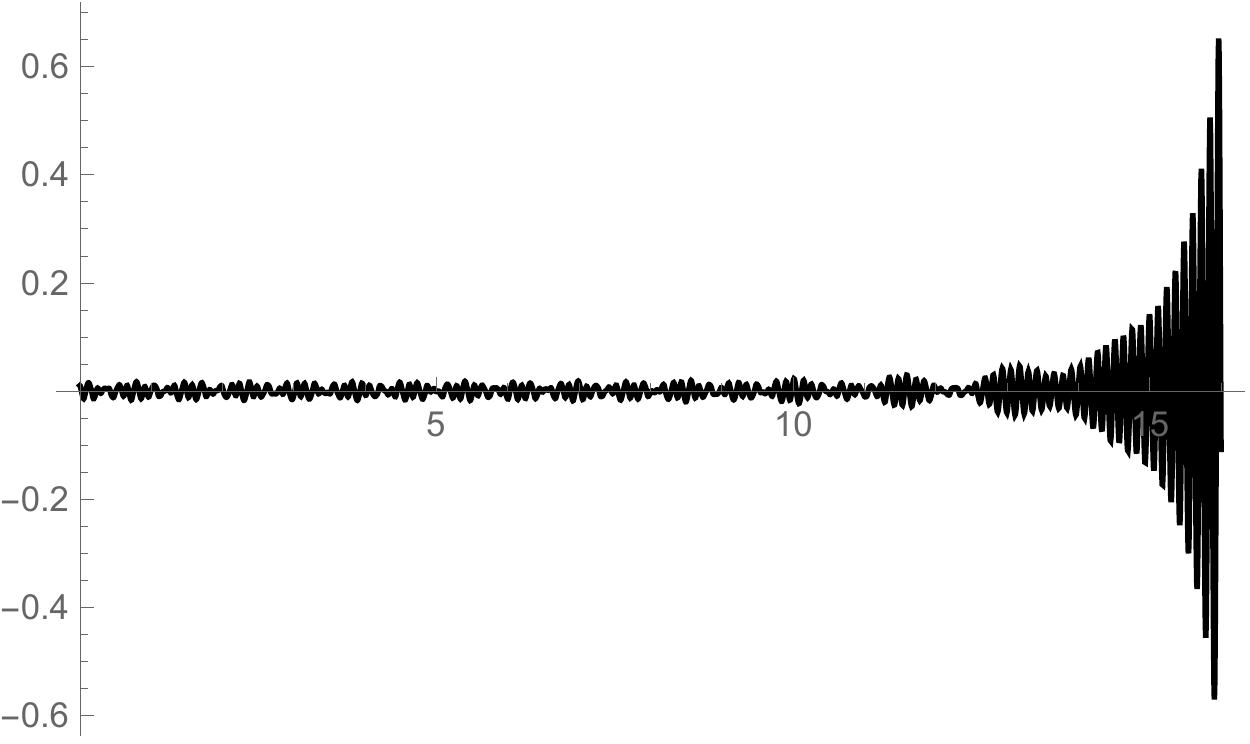}
\vspace{0.3cm}
\\
\includegraphics[scale=0.31]{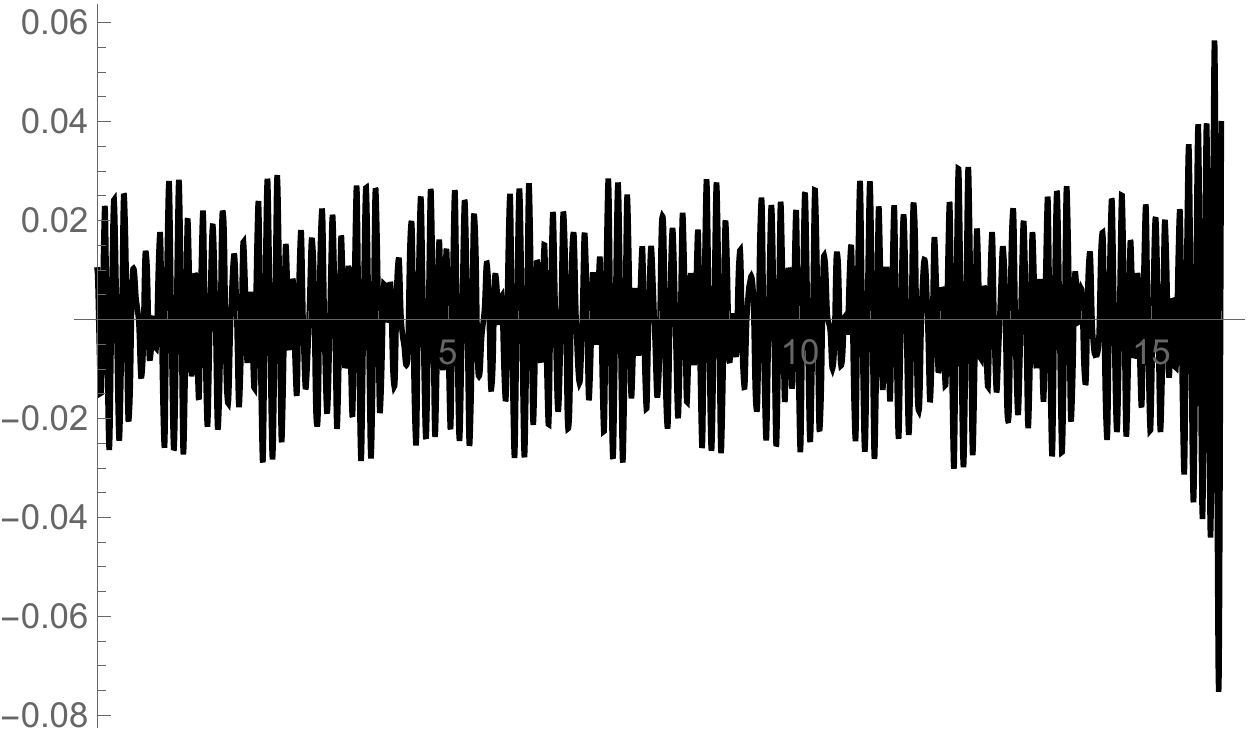}
\;
\includegraphics[scale=0.31]{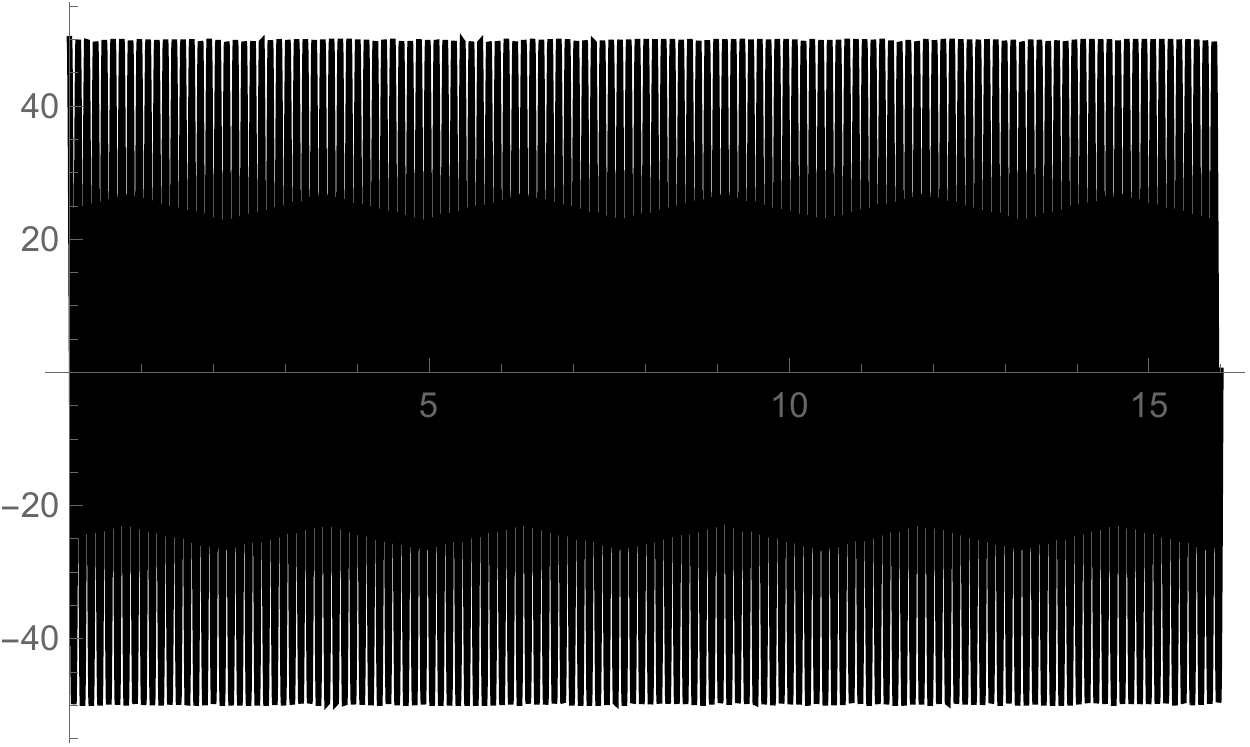}
\;
\includegraphics[scale=0.31]{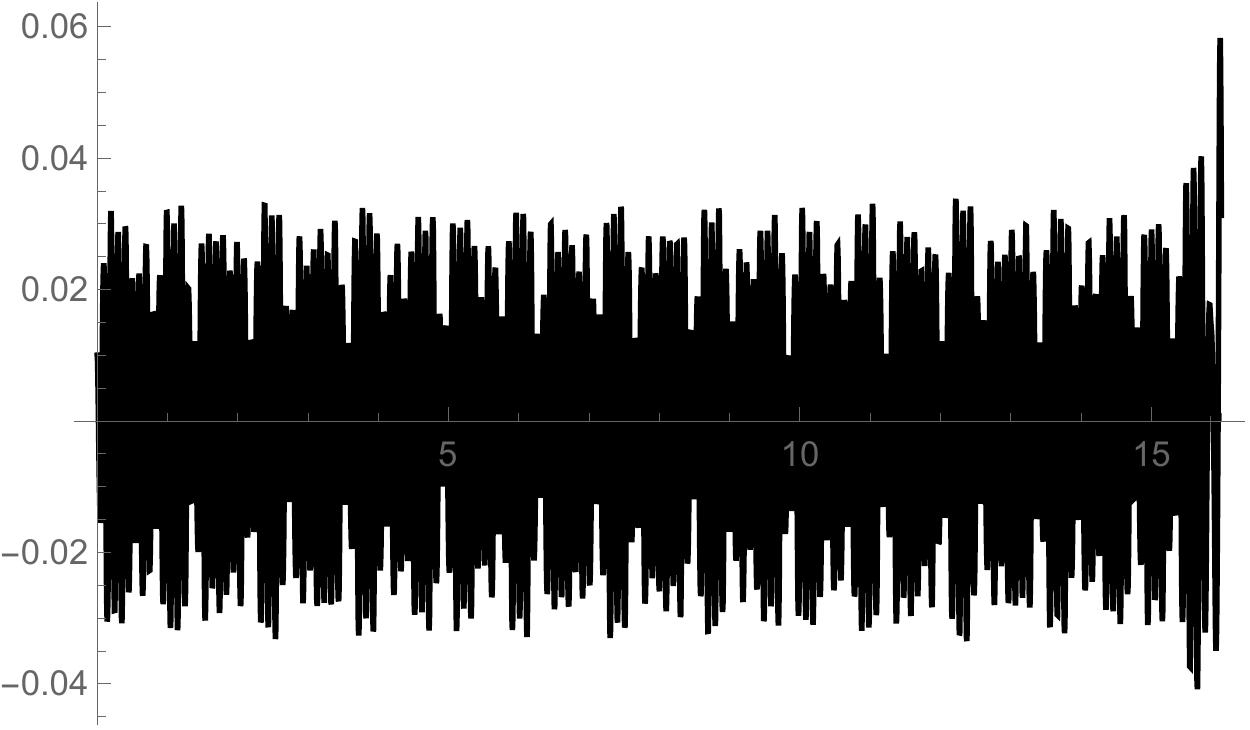}
\;
\includegraphics[scale=0.31]{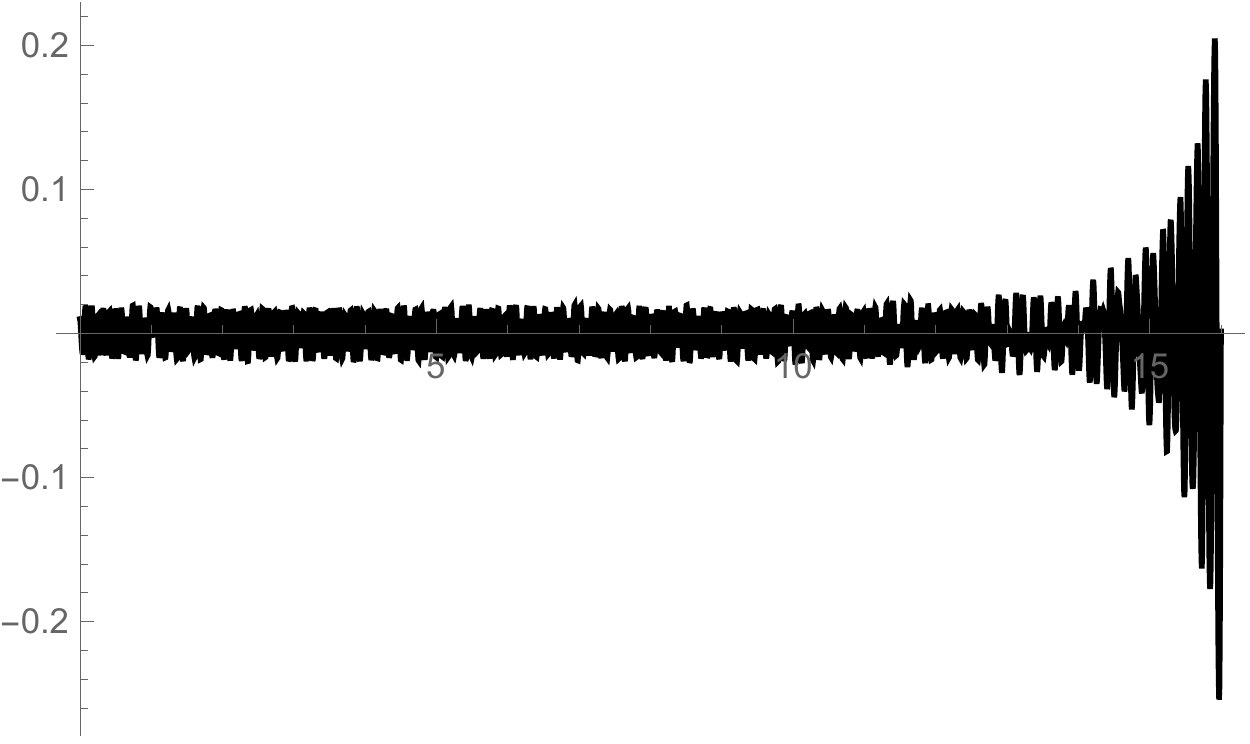}
\vspace{0.3cm}
\\
\includegraphics[scale=0.31]{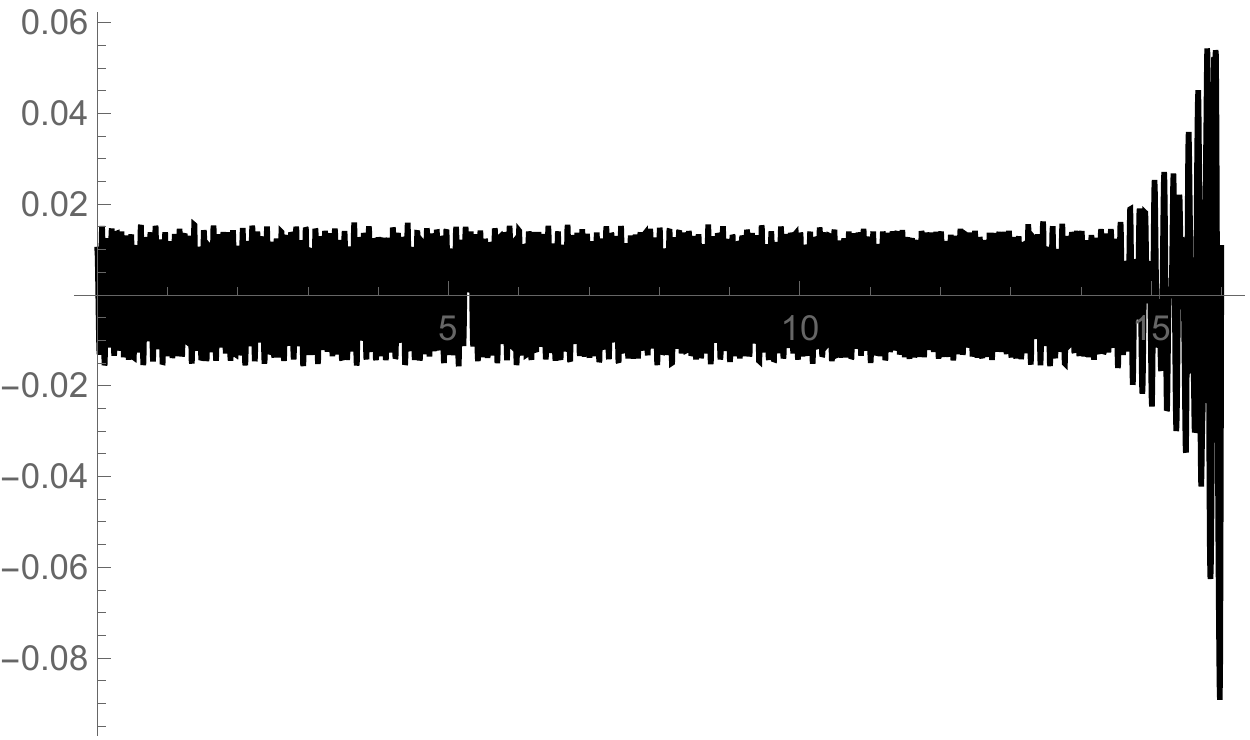}
\;
\includegraphics[scale=0.31]{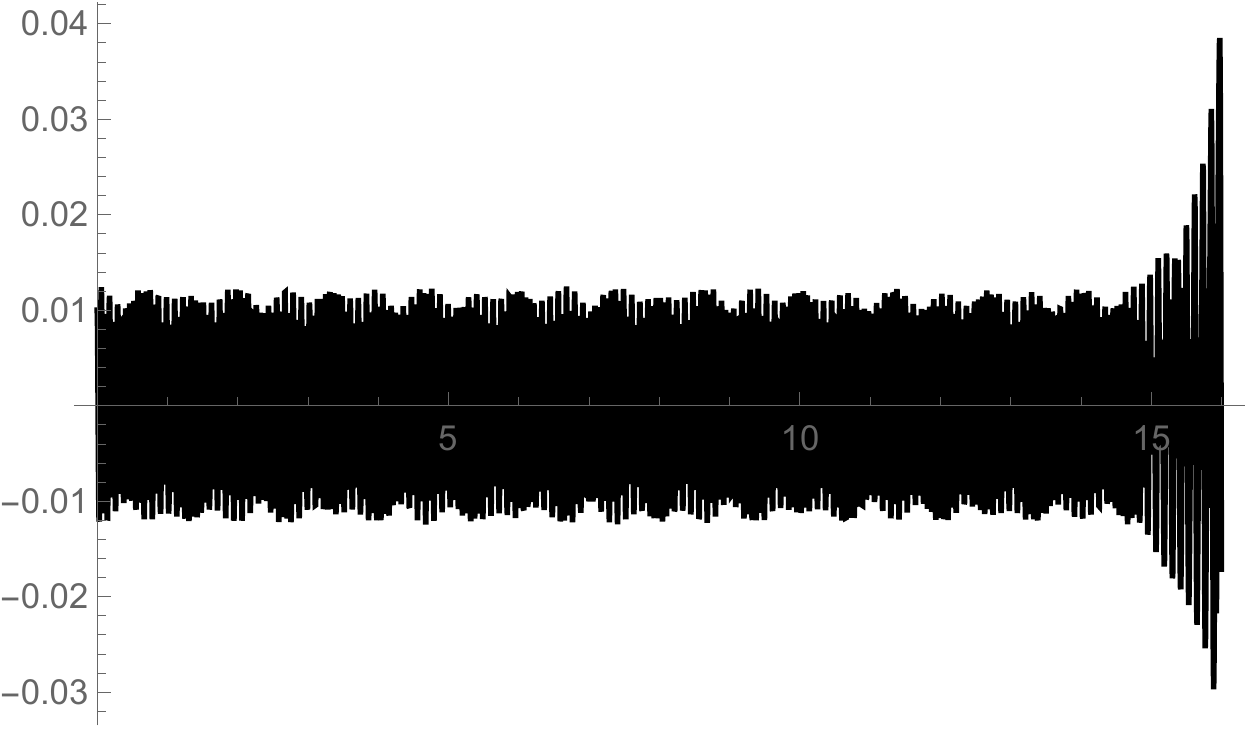}
\;
\includegraphics[scale=0.31]{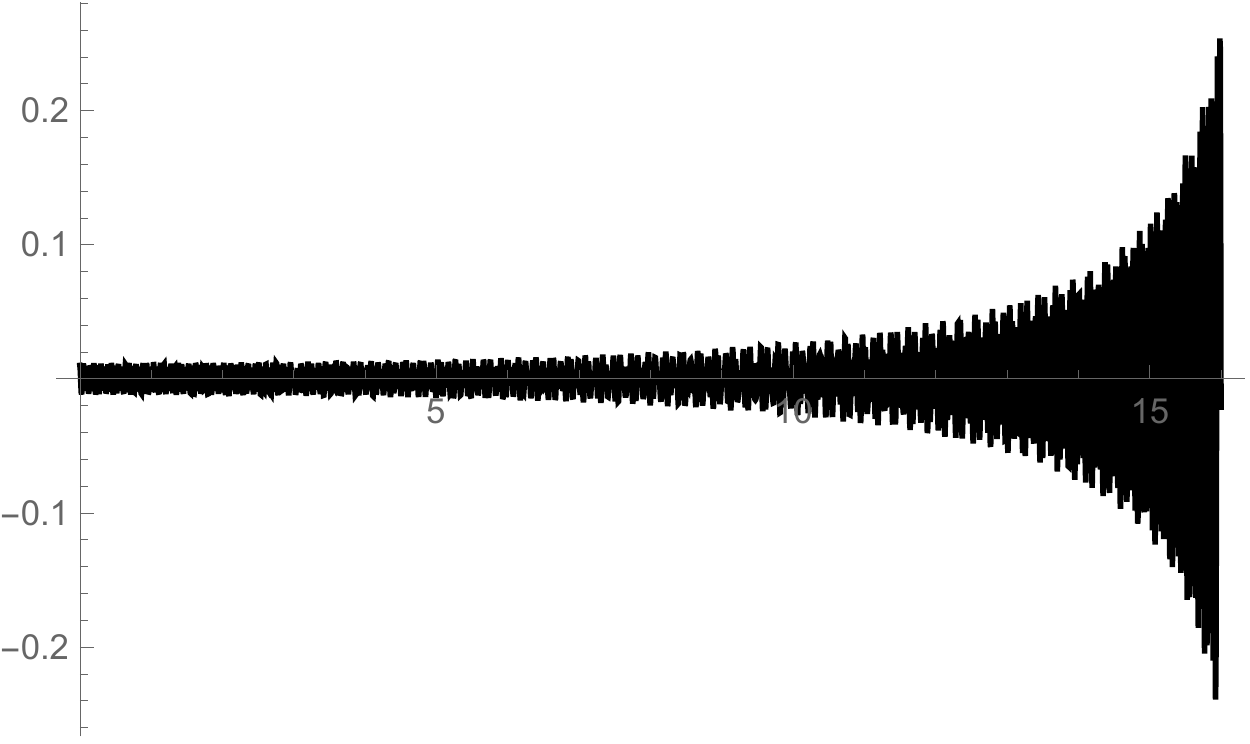}
\;
\includegraphics[scale=0.31]{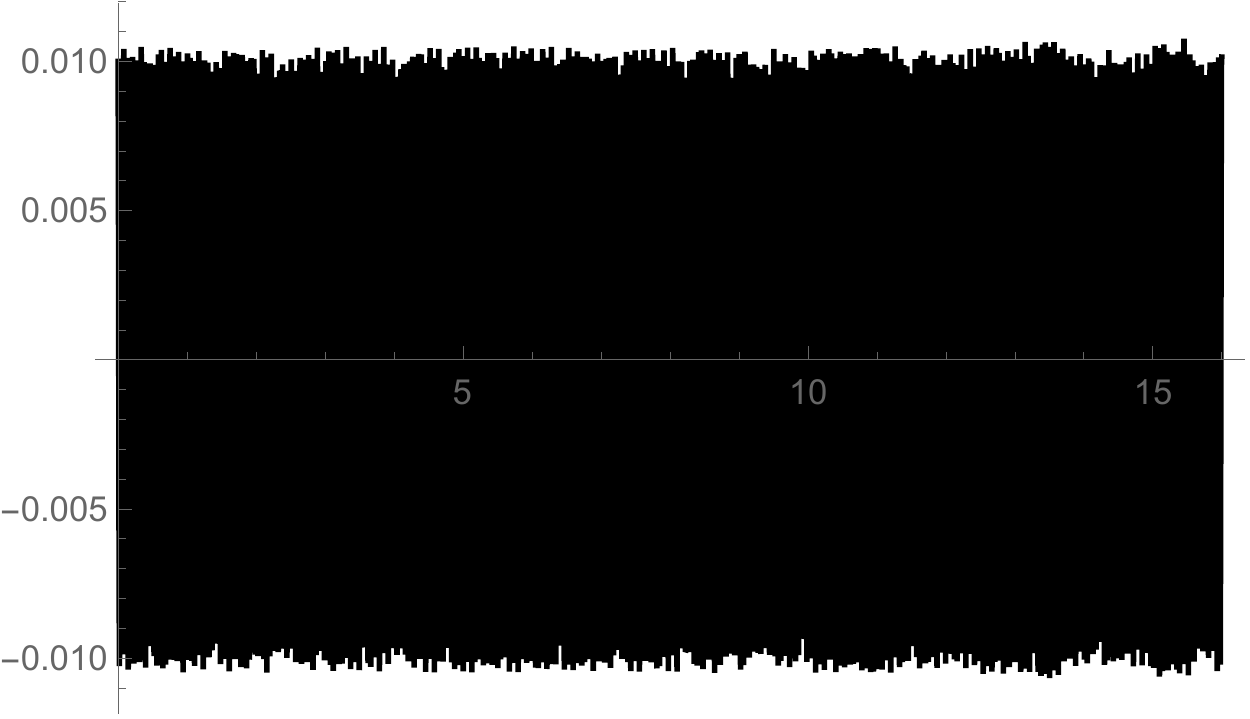}
\caption{We depict the solution of problem \eqref{firstattempt} for $u_0(x)=50.097 \sin (6x) + 0.01 \sum_{n \neq 6} \sin (nx)$ and $u_1(x) \equiv 0$.}
\label{tuttimodicubo1}
\end{figure}

\begin{figure}[!ht]
\center
\includegraphics[scale=0.31]{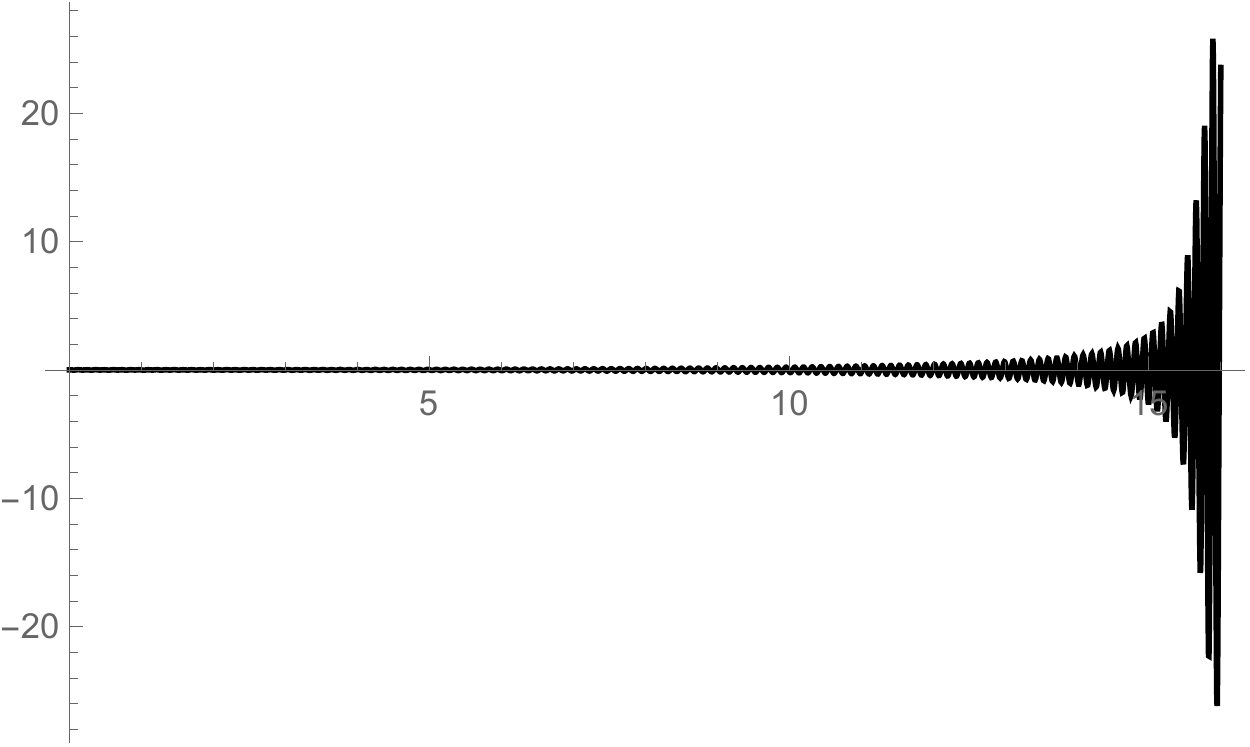}
\;
\includegraphics[scale=0.31]{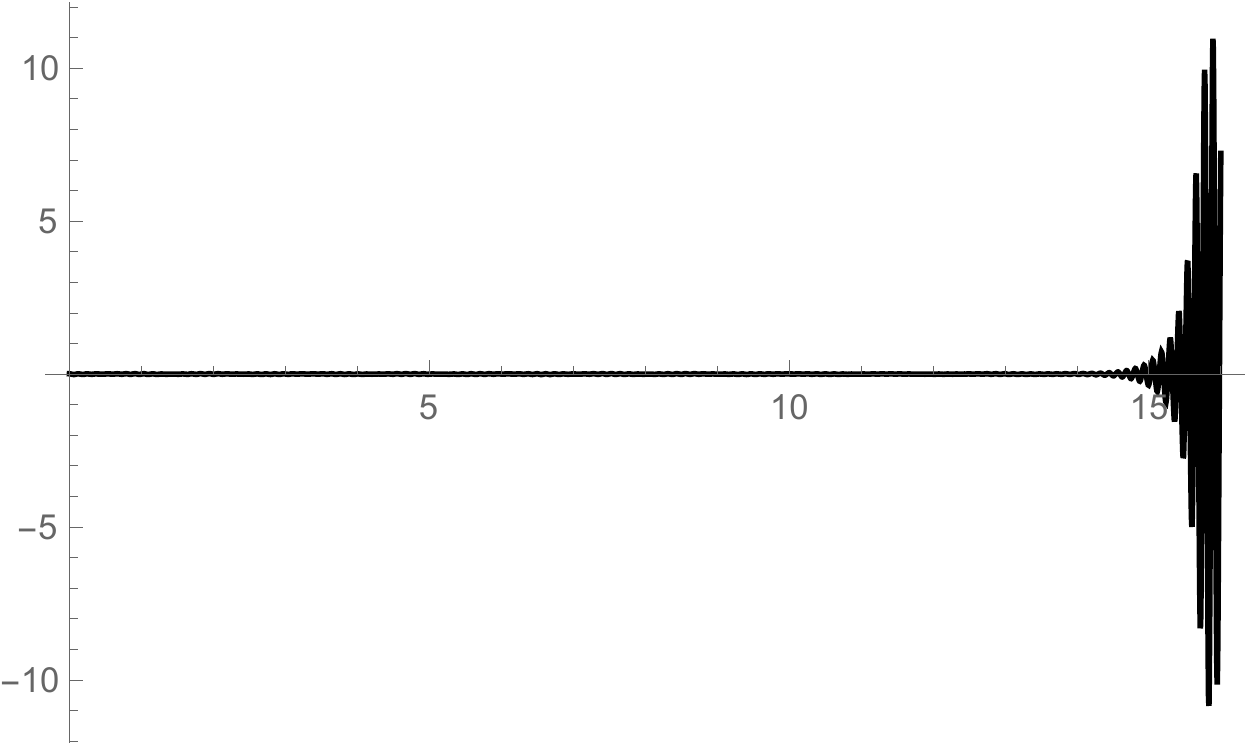}
\;
\includegraphics[scale=0.31]{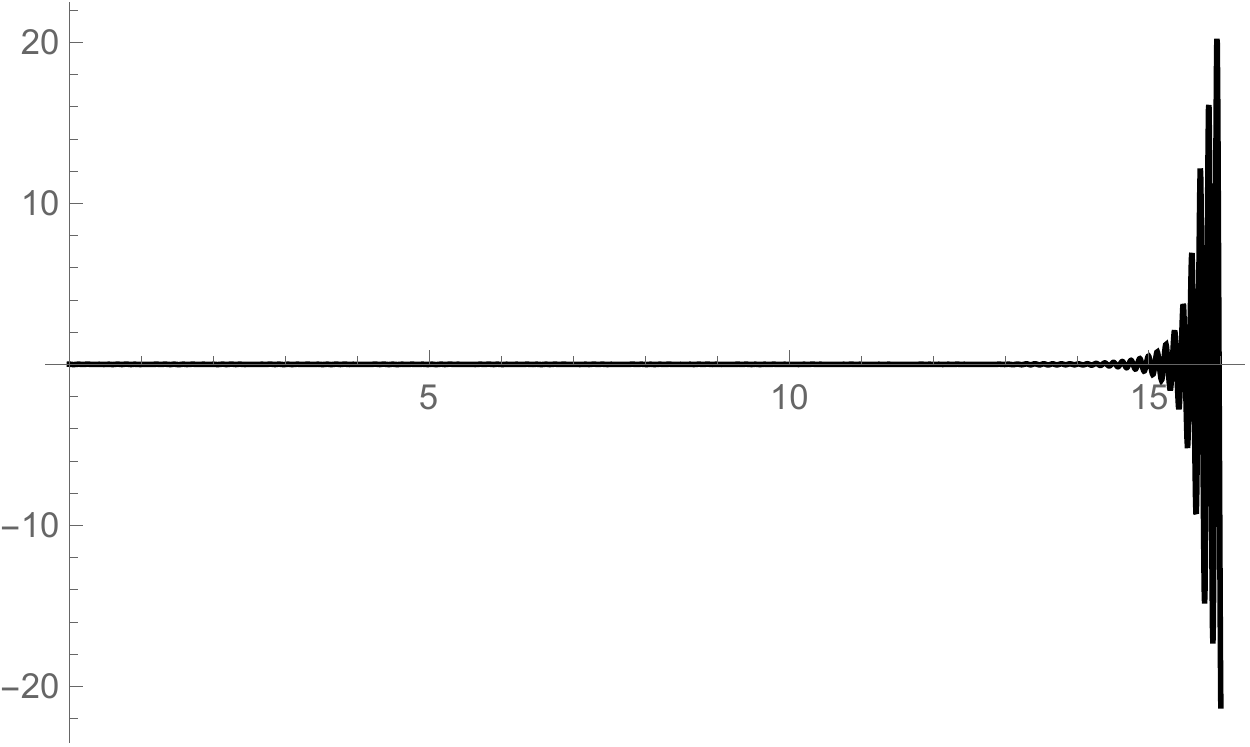}
\;
\includegraphics[scale=0.31]{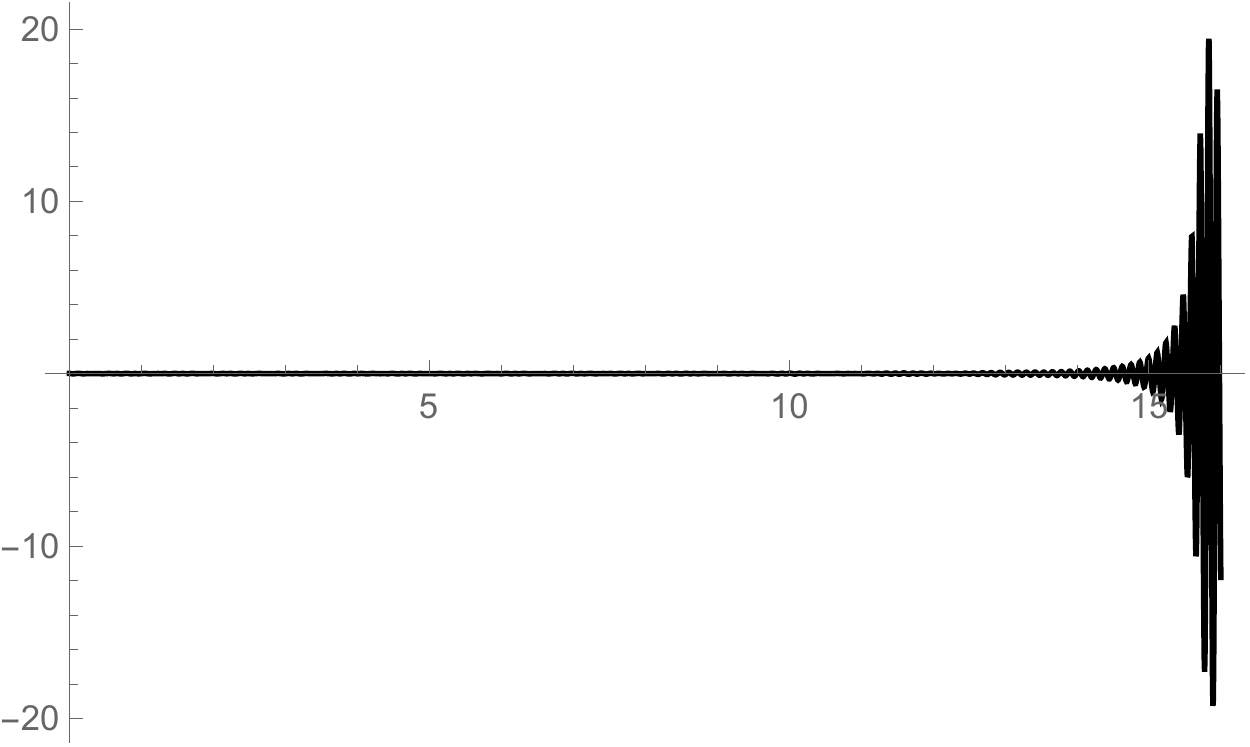}
\vspace{0.3cm}
\\
\includegraphics[scale=0.31]{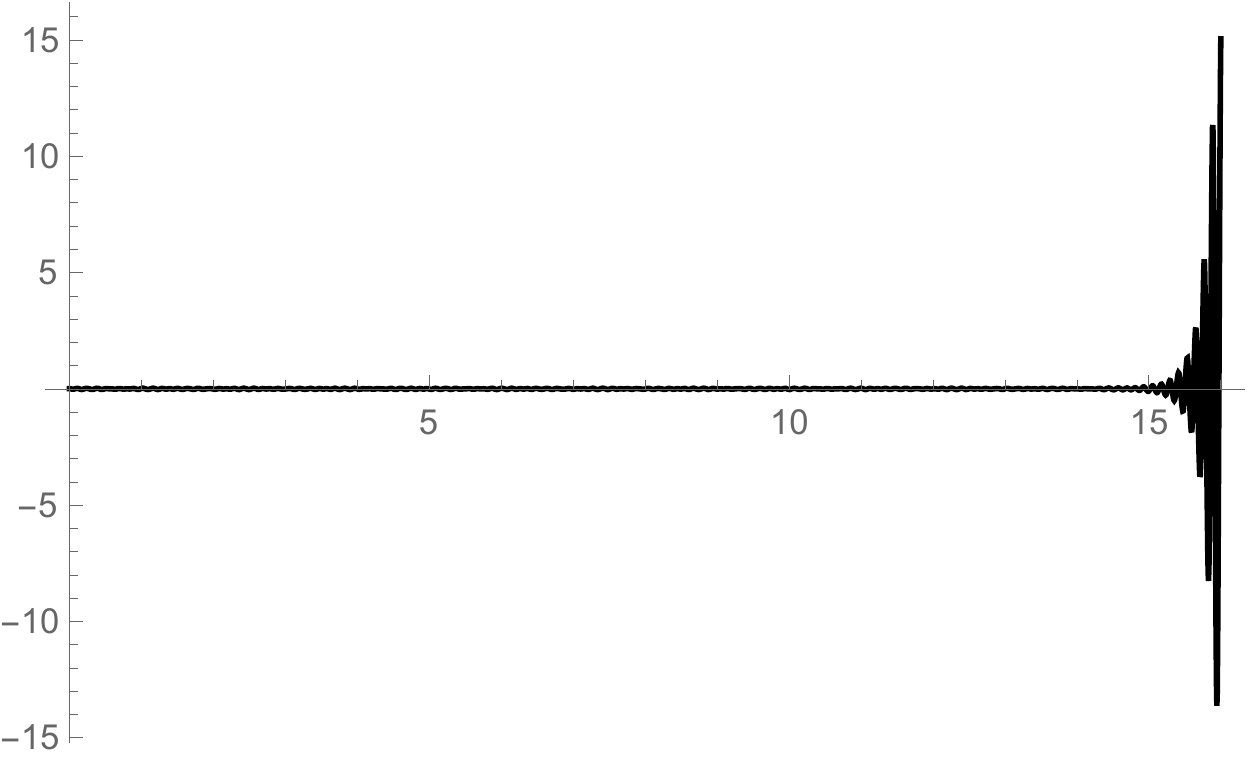}
\;
\includegraphics[scale=0.31]{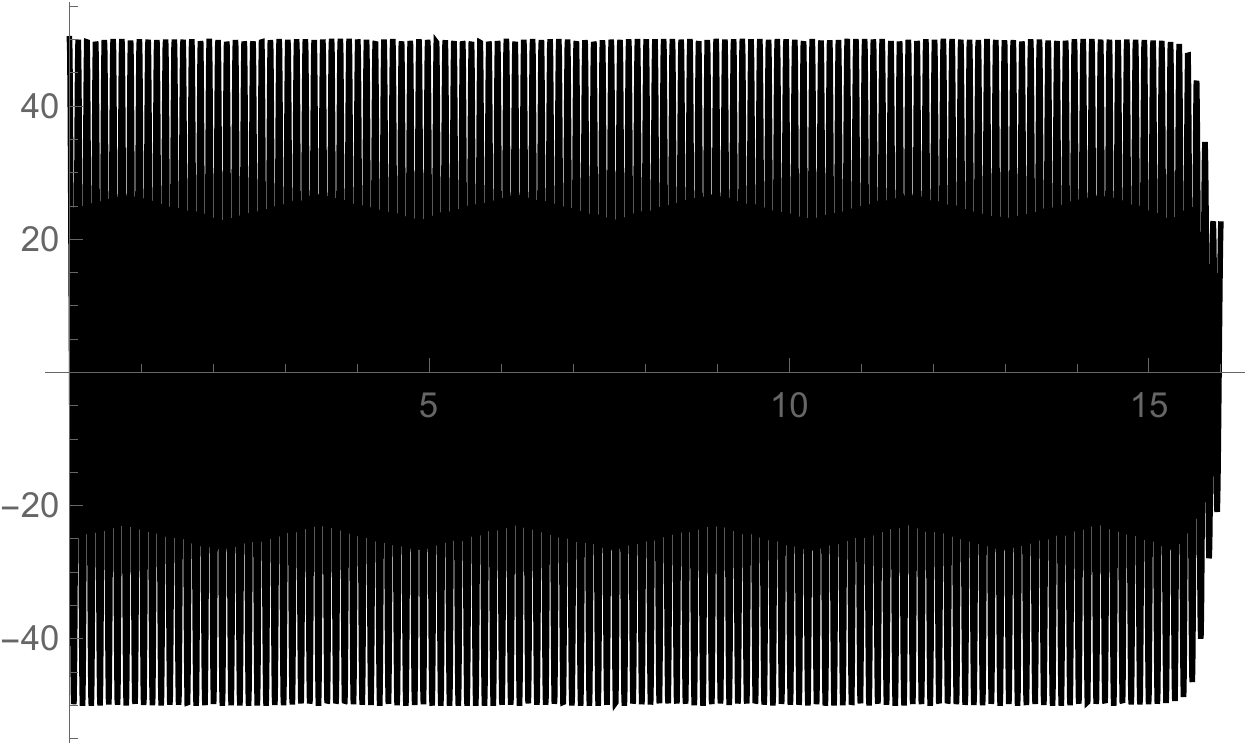}
\;
\includegraphics[scale=0.31]{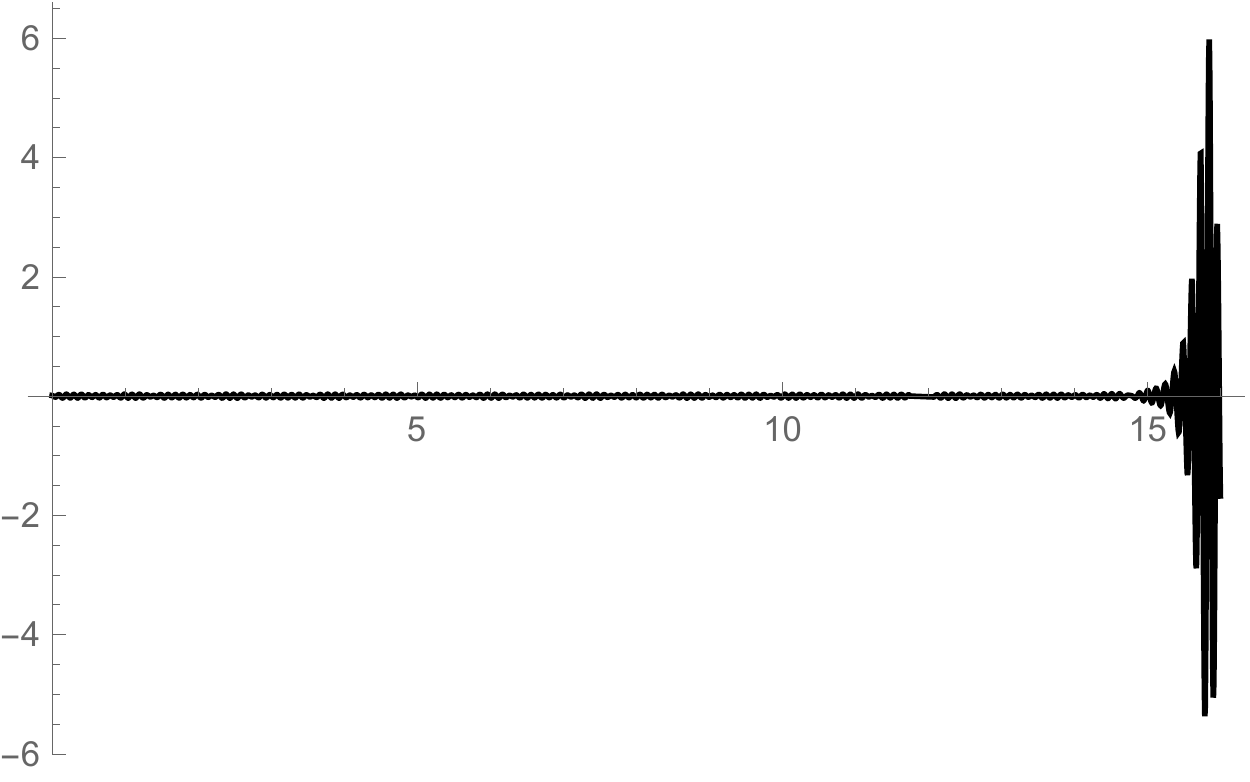}
\;
\includegraphics[scale=0.31]{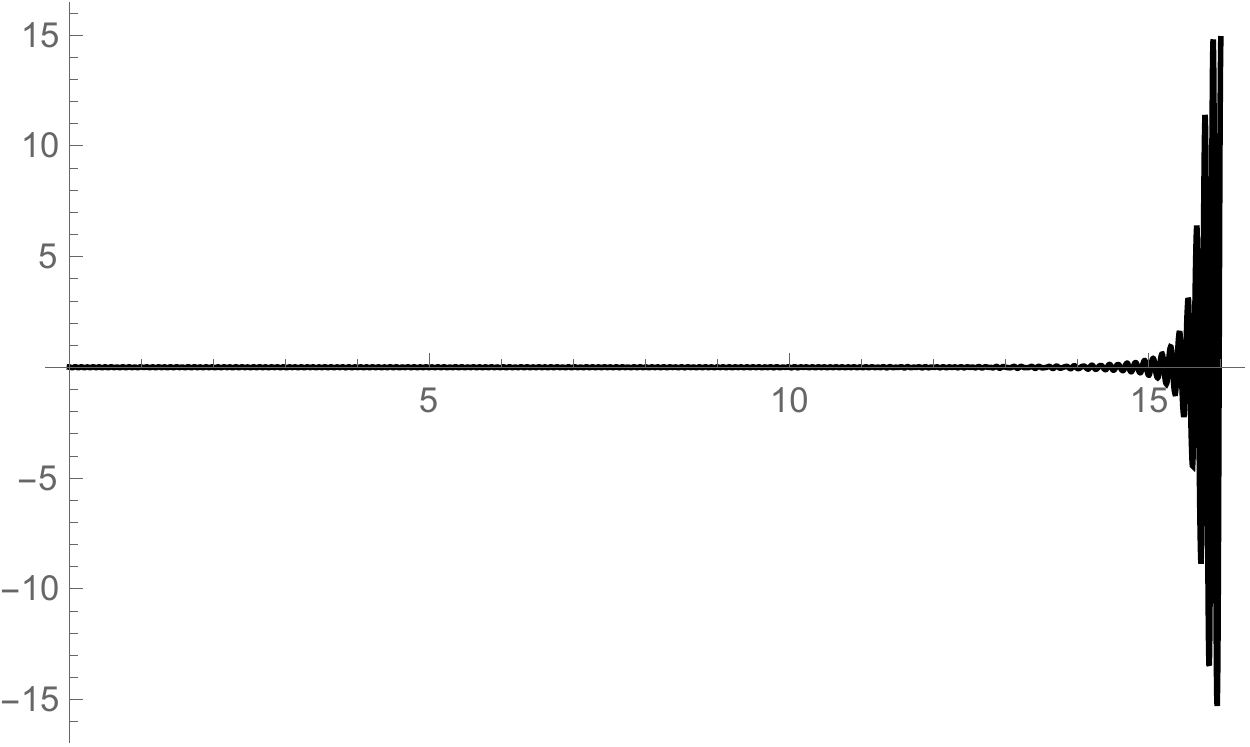}
\vspace{0.3cm}
\\
\includegraphics[scale=0.31]{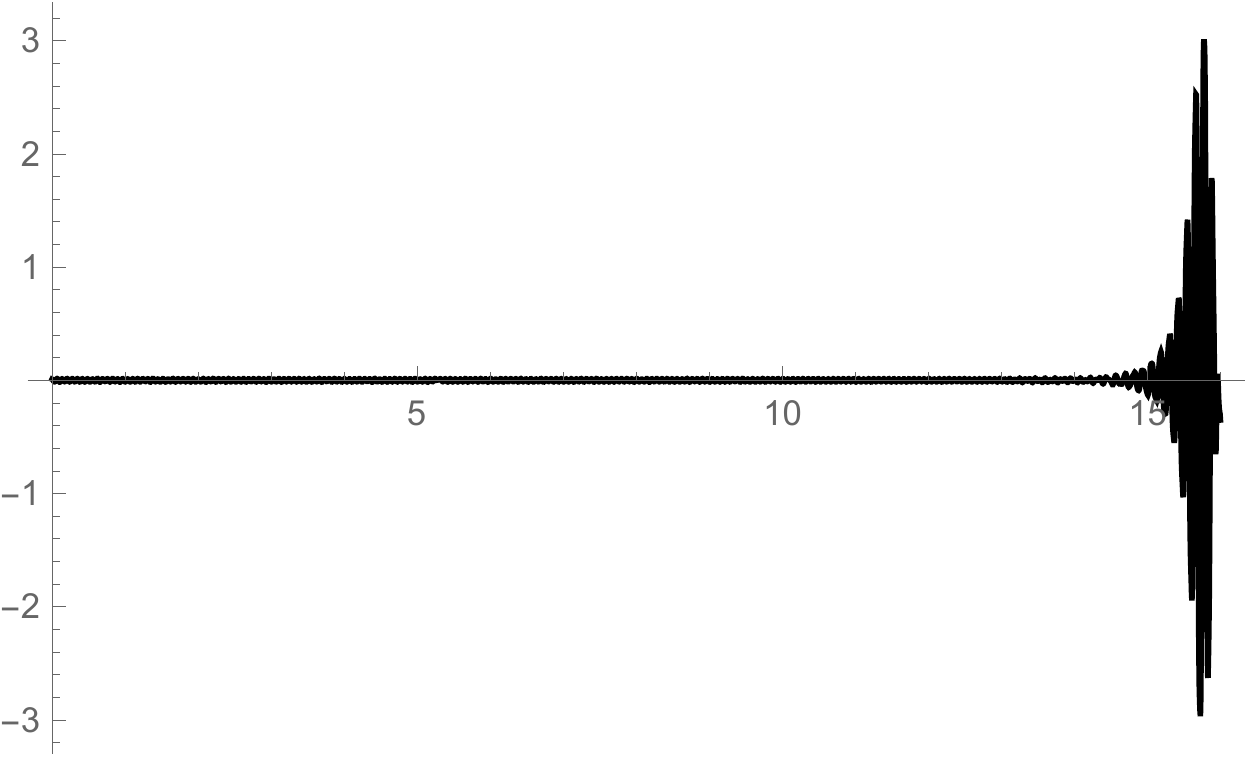}
\;
\includegraphics[scale=0.31]{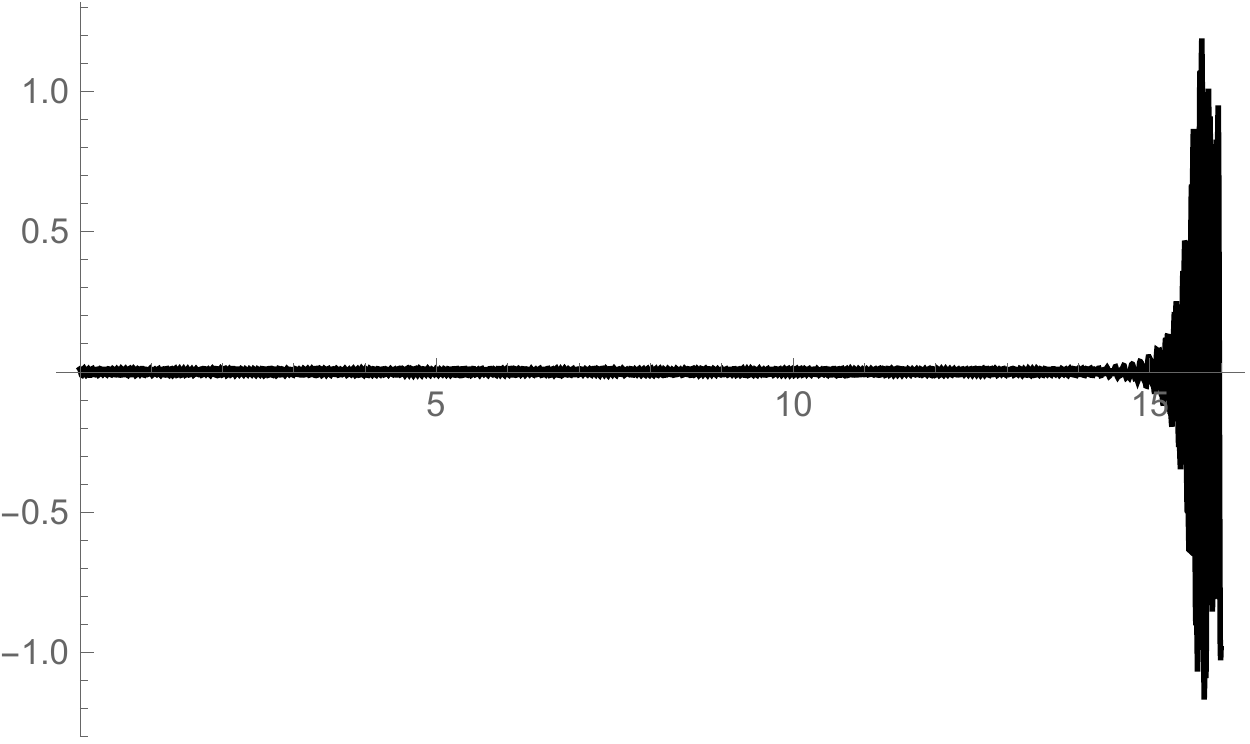}
\;
\includegraphics[scale=0.31]{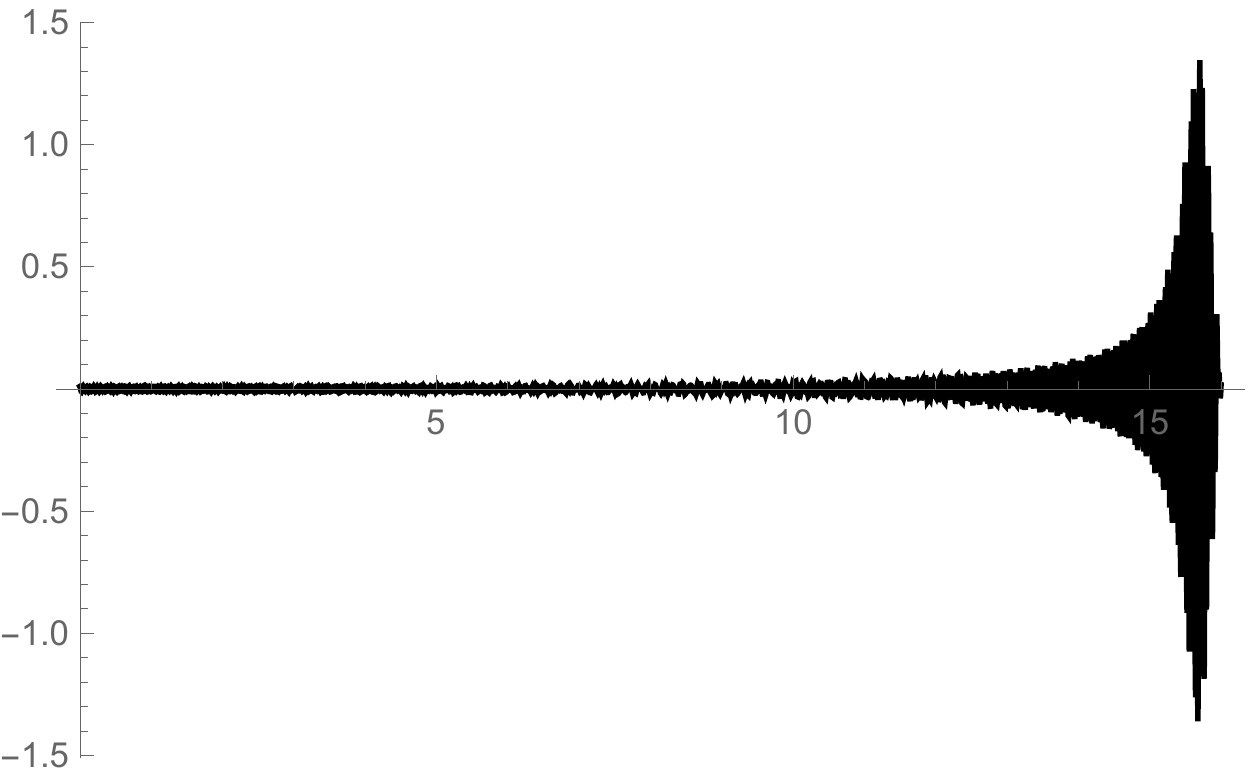}
\;
\includegraphics[scale=0.31]{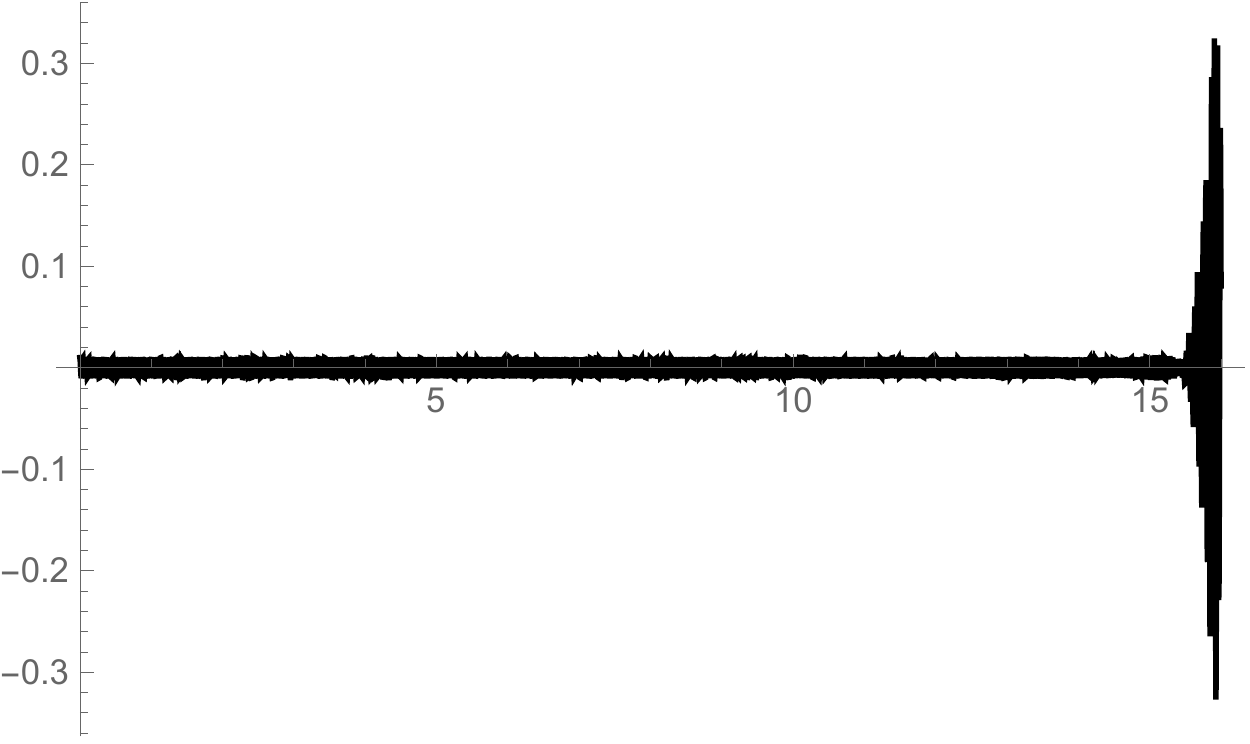}
\caption{We depict the solution of problem \eqref{firstattempt} for $u_0(x)=50.099 \sin (6x) + 0.01 \sum_{n \neq 6} \sin (nx)$ and $u_1(x) \equiv 0$.}
\label{tuttimodicubo2}
\end{figure}

Finally, in Table \ref{soglie} we report some critical instability thresholds $M_j$ (until time $T=16$) observed for each prevailing mode $j=1, \ldots, 6$. Not to overload the content, we choose to report approximate values of the thresholds with possible error of $\pm 0.1$; the table has been constructed by performing a careful analysis (in some cases, varying the initial amplitude with a step of $10^{-3}$) to highlight that there is always a unique Fourier component, indicated in the third line, that starts fulfilling \eqref{finitegrande}. Here we have used again $N=12$ modes and the residual components have been set equal to $0.01$ for $t=0$.
In the last line of the table, we indicate which are the modes which ``resist'' to the energy transfer, namely the ones manifesting regular oscillations of the same amplitude as their initial datum even if the whole system is going towards instability. We remark that the thresholds are obtained for $T=16$ and may vary when changing the observation time $T$.
\par
Notice that, once the threshold is overcome, we do not know if instability always occurs or if, as for linear stability, stable and unstable behaviours are likely to alternate. However, we guess that, for very large energies, the probability for the solutions to be unstable considerably increases.

\begin{table}[!ht]
\begin{center}
\begin{tabular}{|c|c|c|c|c|c|c|c|c|c|c|}
\hline
\textrm{ Prevailing mode } & 1 & 2 & 3 & 4 & 5 & 6 \\ 
\hline
\textrm{ Approximate amplitude threshold } & 13.1 & 6.2 &  13.7 
& 23.4 
 & 32.1
& 50.1 \\ 
\hline
\textrm{ First residual mode fulfilling \eqref{finitegrande} } & 2 & 1 & 2 & 1 & 8 
& 1 \\ 
\hline
\textrm{``Resistent'' mode(s)} & 10/12 
& 
8-12 & 
11/12 & 8 
& 10 & 12 \\ 
\hline
\end{tabular}
\vspace{0.2 cm}
\caption{Instability thresholds of the Galerkin approximation of \eqref{cdbeamcubic}.}\label{soglie}
\end{center}
\end{table}

\section{The case of a positive cubic nonlinearity}\label{poscub}

As an intermediate choice between $f(u)=u^+$ and $f(u)=u^3$, one can take $f(u)=(\alpha_1u+\alpha_2u^3)^+$ as in \cite{gazwang}. After dropping again the linear term ($\alpha_1=0$), we obtain the nonlinearity
$$
f(u)=(u^+)^3\,, 
$$
which will be our object of interest throughout the present section. Explicitly, we consider the problem
\neweq{poscubic}
\left\{\begin{array}{ll}
u_{tt}+u_{xxxx}+(u^+)^3=0\quad & \mbox{for }(x,t)\in(0,\pi)\times(0,\infty)\\
u(0,t)=u(\pi,t)=u_{xx}(0,t)=u_{xx}(\pi,t)=0\quad & \mbox{for }t\in(0,\infty)\\
u(x,0)=u_0(x)\, ,\quad u_t(x,0)=u_1(x)\quad & \mbox{for }x\in(0,\pi)\, .
\end{array}\right.\endeq
Notice that $f(u)=(u^+)^3$ is of class $C^2(\R)$ and satisfies \eq{flip}. By seeking solutions of \eq{poscubic} in the form \eq{general},
system \eqref{infinite} becomes
$$
\ddot{\phi}_n(t)+n^4\phi_n(t)+\frac{2}{\pi}\int_0^\pi\bigg[\Big(\sum_{m=1}^\infty\phi_m(t)\sin(mx)\Big)^+\bigg]^3\sin(nx)\, dx=0
\qquad(n=1,...,\infty)
$$
and the main difficulty is again to analyze the contributions of the integral term.
\par
We compare the properties of \eq{poscubic} with those of \eq{cdbeampositive}. Proposition \ref{homogeneity} no longer
holds, which means that \eq{poscubic} is ``more nonlinear'' than \eq{cdbeampositive}. Theorem \ref{evenodd} instead holds, but in a weaker form, since only the first statement remains true.
\begin{theorem}\label{evenodd2}
If $\{a_n\}_{_n}\in\ell^2_4$, $\{b_n\}_{_n}\in\ell^2_2$, and
$$
u_0(x)=\sum_{n=0}^\infty a_n\sin\big((2n+1)x\big)\ ,\quad u_1(x)=\sum_{n=0}^\infty b_n\sin\big((2n+1)x\big)\,,
$$
then there exist functions $\phi_{2n+1}\in C^2(\R_+)$ such that the solution of \eqref{poscubic} is given by \eqref{onlyodd}.
\end{theorem}

The proof of Theorem \ref{evenodd2} is similar to that of Theorem \ref{evenodd} (see Section \ref{casepos}) and therefore we omit it.
However, since a function of the form $u(x,t)=\phi_1(t)\sin(x)$ does not solve \eq{poscubic}, the second part of Theorem \ref{evenodd} is not true for \eqref{poscubic}. In Figure \ref{sempliceposcubo}, we notice that the energy is distributed also on the third mode, contrary to what happens for $f(u)=u^+$: this may be seen as a physiological effect due to the cubic power in the nonlinearity. On the contrary, the second mode is identically $0$ by Theorem \ref{evenodd2}. In Figure \ref{sempliceposcubo2}, we see that the energy is immediately spread from the even mode onto all the other ones as in the case of the positive part.

\begin{figure}[!h]
\center
\includegraphics[scale=0.42]{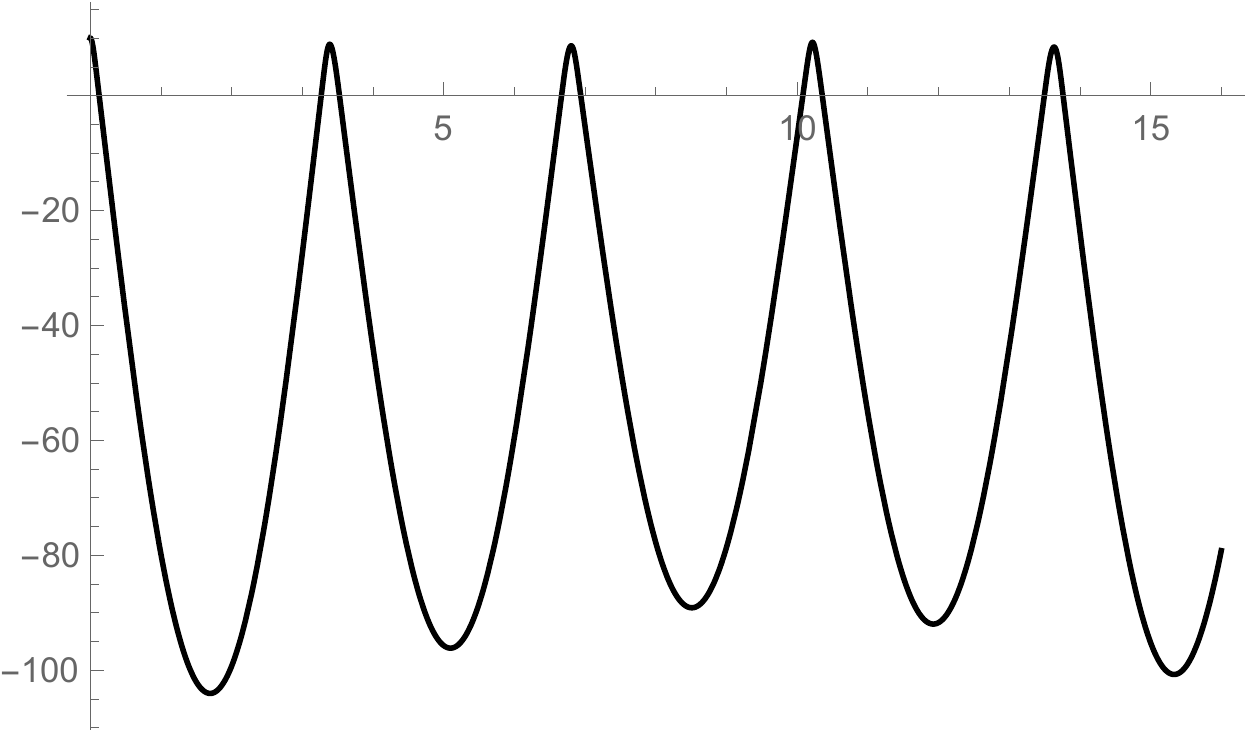}
\quad
\includegraphics[scale=0.42]{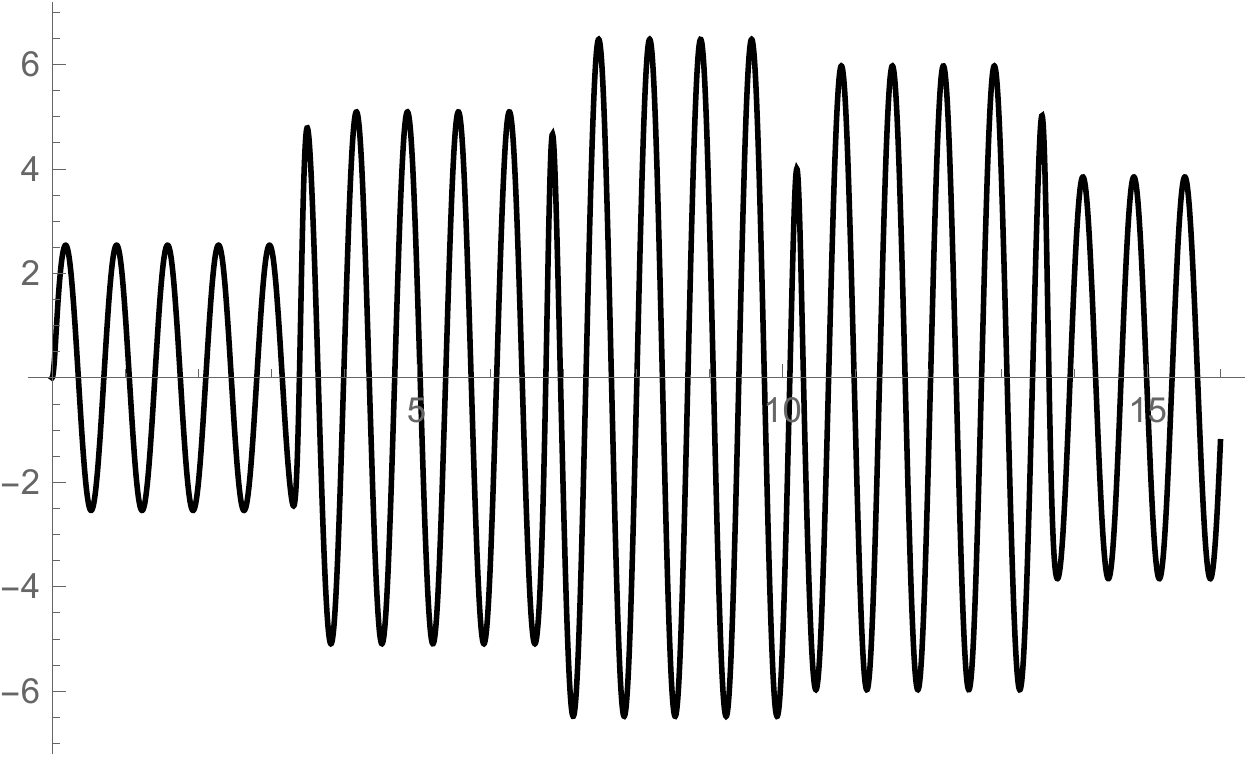}
\caption{Plot of the Fourier components $\varphi_1^3$ and $\varphi_3^3$ for problem \eqref{poscubic}, with $\varphi_1^3(0)=10$, $\varphi_2^3(0)=\varphi_3^3(0)=0$, on the time interval $[0, 16]$. }
\label{sempliceposcubo}
\vspace{0.8 cm}
\includegraphics[scale=0.42]{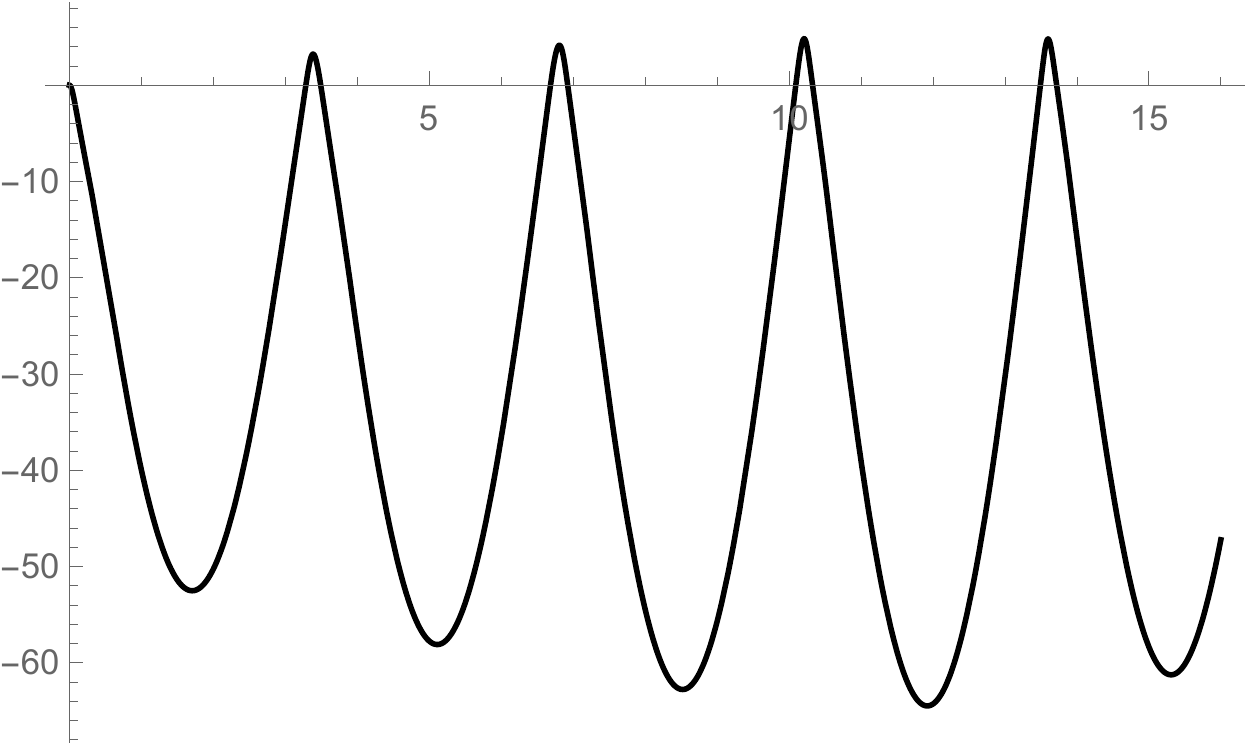}
\quad
\includegraphics[scale=0.42]{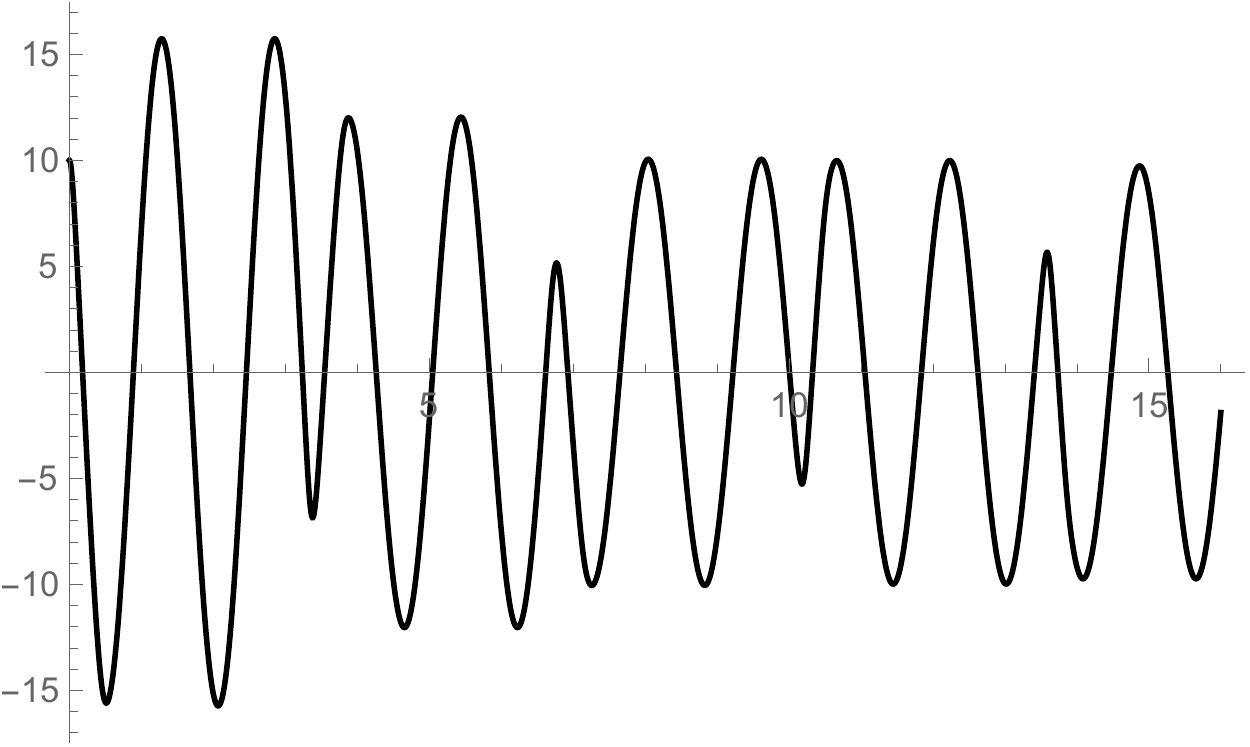}
\quad
\includegraphics[scale=0.42]{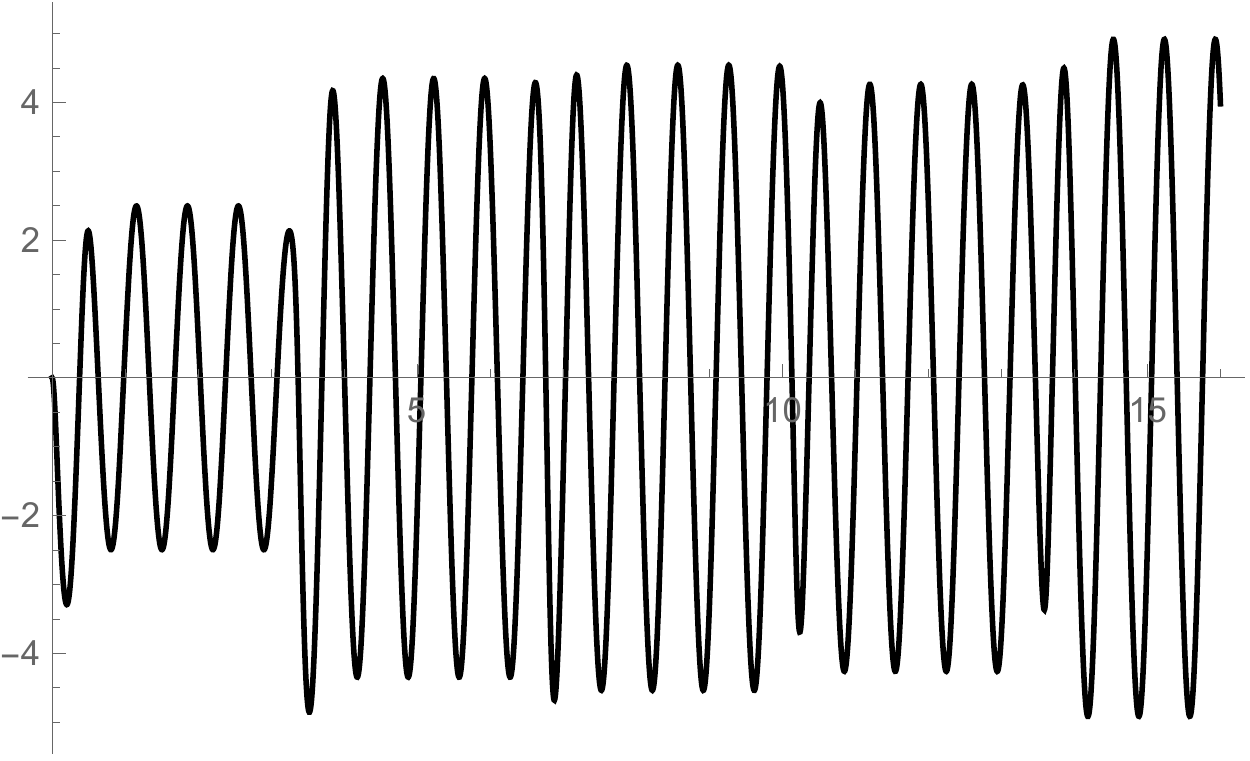}
\caption{Plot of the Fourier components $\varphi_1^3, \varphi_2^3$ and $\varphi_3^3$ for problem \eqref{poscubic}, with $\varphi_2^3(0)=10$, $\varphi_1^3(0)=\varphi_3^3(0)=0$, on the time interval $[0, 16]$.}
\label{sempliceposcubo2}
\end{figure}

\par
Theorems \ref{highermodes} and \ref{secondmode} hold instead in the same way as for \eq{cdbeampositive}, with similar proofs which we omit. Namely, we have the following.

\begin{theorem}\label{highermodesppc}
Assume that \eqref{nontrivial} is a solution of \eqref{poscubic};
then $\phi_1(t)\not\equiv0$.
Moreover, all the odd modes of the solution of \eqref{poscubic} with initial conditions \eqref{potential+} have a nontrivial coefficient for $t>0$.
\end{theorem}

In Figure \ref{plotsposcubo}, we highlight a ``partial'' instability for the autonomous Galerkin system ($N=5$) associated with problem \eqref{poscubic}, for initial data $u_0(x) \approx 50 \sin x$ and $u_1(x) \equiv 0$.
While the growth of the third and the fifth mode is due to a physiological absorption of energy, similarly to Proposition \ref{influence}, the growth of the fourth mode satisfies the second condition in \eqref{finitegrande} for $T = 42$.
Instead, the first condition in \eqref{finitegrande} is not satisfied because the first mode considerably increases its amplitude. However, only the negative part of the first mode increases, thereby not contributing at all to the potential energy. In fact, this is what happens in most of our experiments, independently of the prevailing mode $j$, and is the reason why the first condition in \eqref{finitegrande} is probably never satisfied. We thus conjecture that, as for the case $f(u)=\mu u^+$, a full instability does not appear in this case. There are no physical systems where such large displacements do not contribute to the energy, so that the interest of this simulation is mostly theoretical, aiming at describing the possible behaviours of the beam under different choices of the nonlinearity.
\begin{figure}[!h]
\center
\includegraphics[scale=0.42]{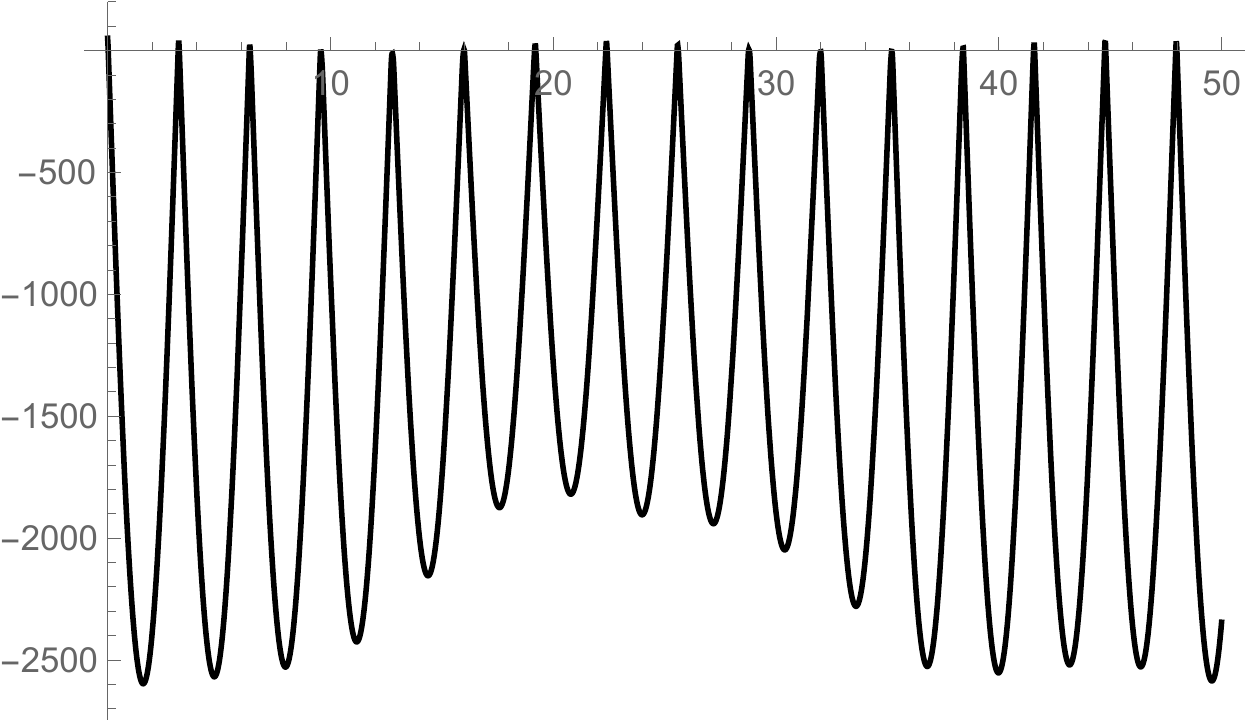}
\quad
\includegraphics[scale=0.42]{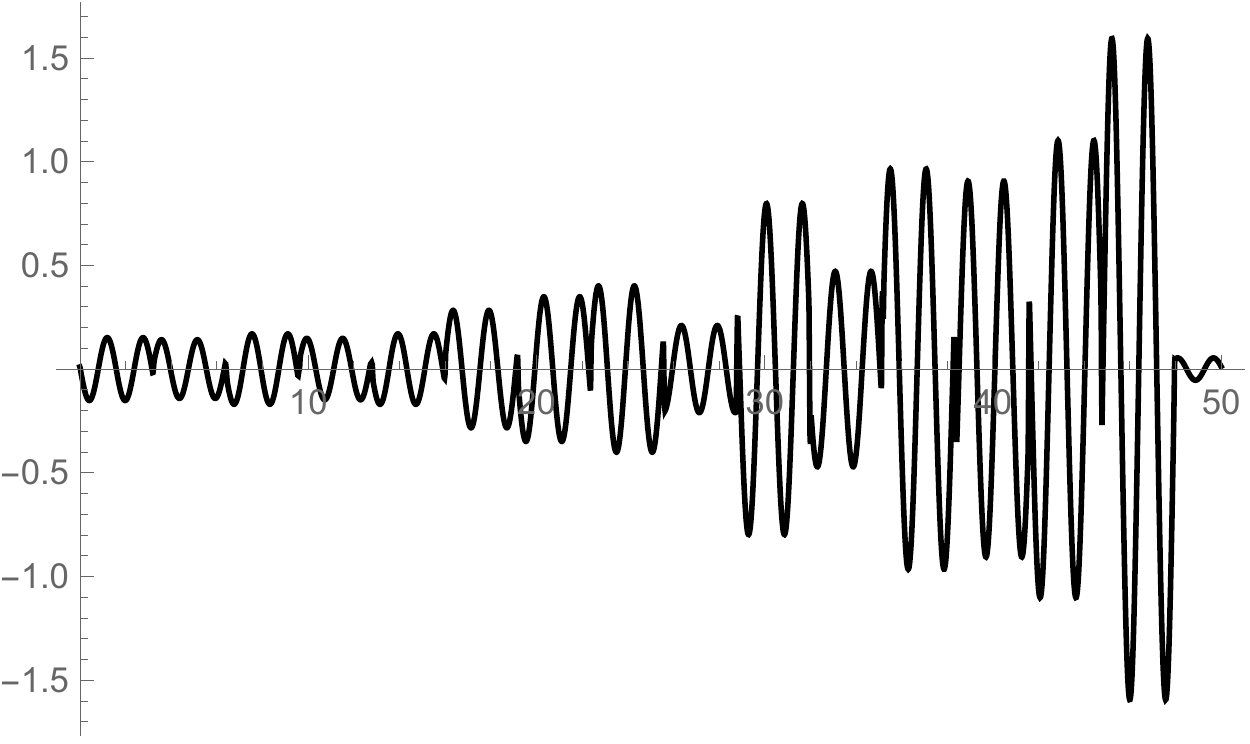}
\\
\vspace{0.3cm}
\includegraphics[scale=0.42]{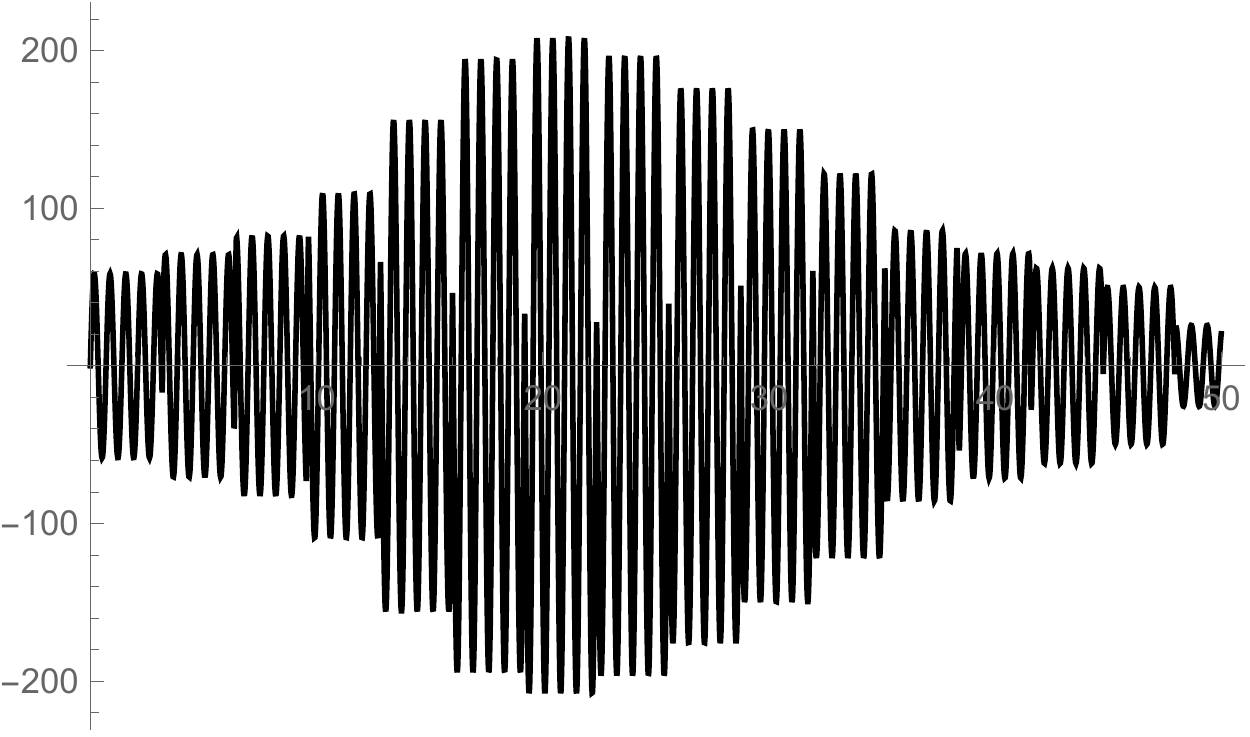}
\quad
\includegraphics[scale=0.42]{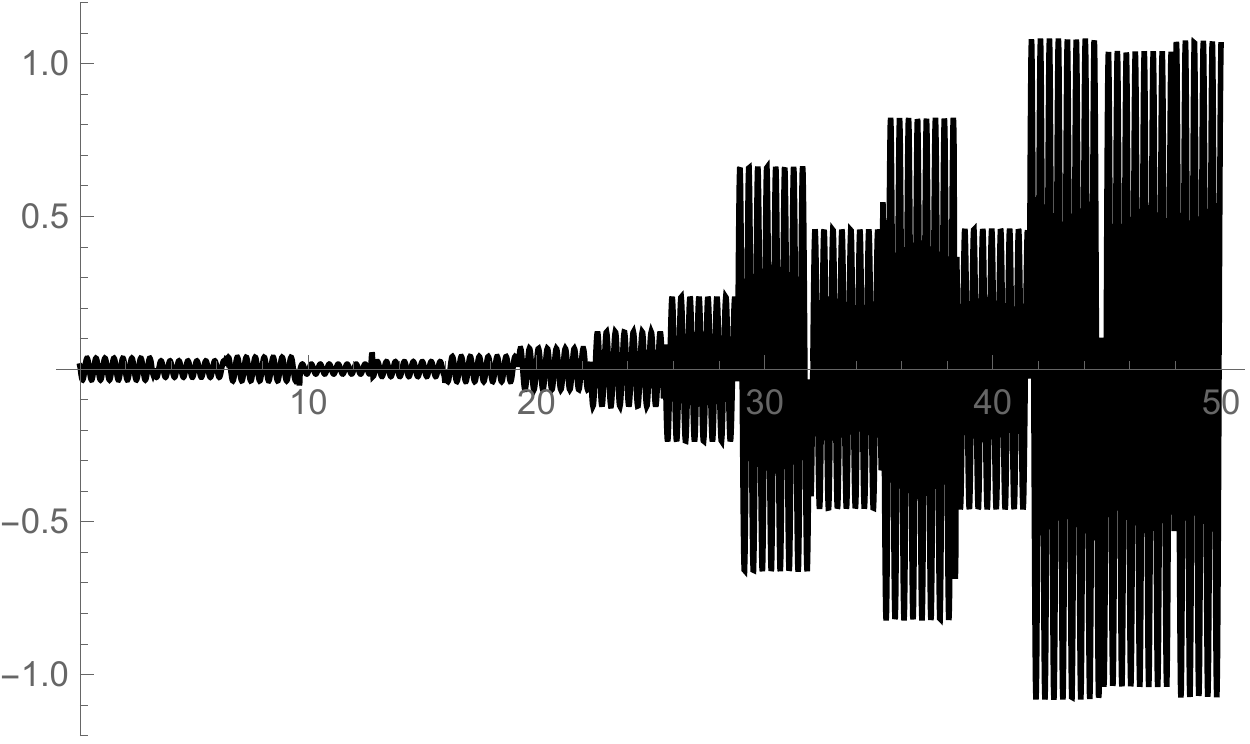}
\quad
\includegraphics[scale=0.42]{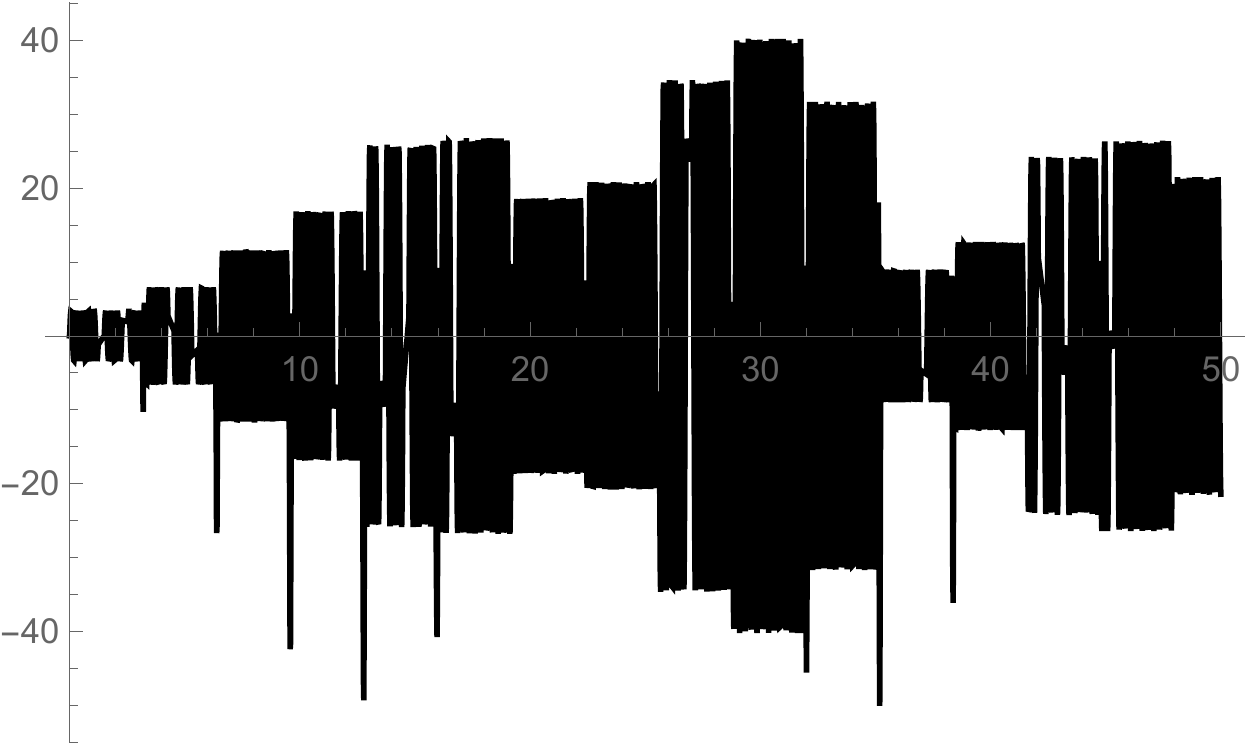}
\caption{The plots of $\varphi_1^5, \ldots, \varphi_5^5$ for problem \eqref{poscubic} on the time interval $[0, 50]$, with $\varphi_1^5(0)=50$, $\varphi_n^5(0)=0.01$ for $n > 1$ and $\dot{\varphi}_n(0) = 0$ for every $n=1, \ldots, 5$.}
\label{plotsposcubo}
\end{figure}

Summarizing, the numerical responses suggest that problem \eqref{poscubic} may be seen as an intermediate case between the positive part and the cubic nonlinearity.

\section{Quantitative estimates of the solutions}\label{dimostrazione}

In this section, we prove a more general version of Theorem \ref{qualitativo} which also provides quantitative estimates on each single Fourier component, see Theorem \ref{approximation} below. We have decided to postpone here such a statement in order not to overload the previous discussion.
\par
We first recall a generalization of the Poincar\'e inequality.
\begin{lemma}\label{Poinc}
Let $N$ be a nonnegative integer and $u \in H^2_*(0, \pi)$ be such that $P_N u = 0$, being $P_N$ the projector onto the space spanned by the first $N$ modes, as in \eqref{proiettori}.
Then,
\begin{equation}\label{stimau2}
\Vert u \Vert_{L^2(0, \pi)} \leq \frac{1}{(N+1)^2} \Vert u'' \Vert_{L^2(0, \pi)}.
\end{equation}
\end{lemma}
The proof is standard and is here omitted; in the limit case $N=0$, \eqref{stimau2} is just the Poincar\'e inequality. We also recall the inequality
\begin{equation}\label{stimauinfty}
\Vert u \Vert_{L^\infty(0, \pi)} \leq \Vert u'' \Vert_{L^2(0, \pi)}\qquad \forall u\in H^2_*(0, \pi)\, ,
\end{equation}
which will be used in the sequel.
\par
Let us now introduce some notation. For a nonlinearity $f$ as in \eqref{flip}, we denote by $L(K)$ its Lipschitz constant on the interval $[-K,K]$, namely
$$
L(K)= \sup_{\vert u \vert, \vert v \vert \leq K \atop u \neq v} \Big\vert \frac{f(u) - f(v)}{u-v} \Big\vert\, ;
$$
since $f$ is increasing, we have
\begin{equation}\label{derivata}
\Big\vert f'(\xi) - \frac{L(K)}{2} \Big\vert \leq \frac{L(K)}{2} \quad \mbox{ for all } \vert \xi \vert \leq K\, .
\end{equation}
Fixed $T > 0$, we also introduce the constant
$$
M=\sqrt{2E(0)} + \int_0^T \Vert g(s) \Vert_{L^2(0, \pi)} \, ds,
$$
where $E(t)$ is as in \eqref{energy}.
This constant is needed to obtain uniform bounds on the solutions; for forcing terms of the form \eqref{g}, we have
$$
M
\leq
\sqrt{2E(0)} + \sqrt{2\pi}\frac{\alpha}{\gamma} \mathcal{I}\Big(1+\frac{\gamma T}{\pi}\Big),
$$
where, as before, $\mathcal{I}(\cdot)$ denotes the integer part function.
\par
With these positions, the following statement holds.
\begin{theorem}\label{approximation}
Let $T > 0$ be fixed and assume that $Q_N u_0 = Q_N u_1 = 0$ and $Q_N g = 0$, for some integer $N$, where $Q_N$ is as in \eqref{proiettori}. Let $u(x,t)$
be the strong solution of \eqref{nonlinearforced}, written in the form \eqref{general}, and let $u^N$ be its approximation given by \eqref{tronca}.  Then,
for any $n \leq N$ the following estimate holds:
\begin{equation}\label{stimamodo}
\Vert \varphi_n - \varphi_n^N  \Vert_{L^\infty(0, T)} \leq  \frac{4M}{(N+1)^2}\, \sqrt{\frac{L(M)+2}{\pi}}\,
\frac{e^{CT/2}-1}{\sqrt{n^4+\frac{1}{T^2}+\frac{L(M)}{2}}} \qquad\forall0\le t\le T\, ,
\end{equation}
where
$$
C=\frac{L(M)}{\sqrt{2(L(M)+2)}}.
$$
\end{theorem}
We observe that the bound on the right-hand side of
\eqref{stimamodo} is decreasing with respect to $n$ and goes to $0$ for $N \to +\infty$, as expected.
As a byproduct of Theorem \ref{approximation}, it is possible to deduce a quantitative estimate of $N^\epsilon$ appearing in Theorem \ref{qualitativo} in terms of the initial data and the forcing term.

The proof makes use of the following lemma, which provides a general estimate of the solution $u$ of \eqref{nonlinearforced} and of its derivatives.

\begin{lemma}\label{stimaderivate}
Let $u(x, t)$ be a strong solution of \eqref{nonlinearforced}. Then, for every $t \in [0, T]$, it holds
$$
\Vert u_t(t) \Vert^2_{L^2(0, \pi)} + \Vert u_{xx}(t) \Vert^2_{L^2(0, \pi)} + 2 \int_0^\pi F(u(t)) \, dx \leq M^2. 
$$
\end{lemma}
\begin{proof}
Multiplying \eqref{nonlinearforced} by $u_t$ and integrating between $0$ and $\pi$, we find
\begin{eqnarray}\label{intermedia1}
\frac{dE(t)}{dt} & = & \frac{d}{dt} \left(\frac{\Vert u_t \Vert_{L^2(0, \pi)}^2}{2} + \frac{\Vert u_{xx} \Vert_{L^2(0, \pi)}^2}{2} + \int_0^\pi F(u(t)) \, dx \right) \nonumber
\\
& = & \int_0^\pi g(t) u_t \, dx
\leq \Vert g(t) \Vert_{L^2(0 ,\pi)}\Vert u_t \Vert_{L^2(0, \pi)} \leq \Vert g(t) \Vert_{L^2(0 ,\pi)} \sqrt{2E(t)},
\end{eqnarray}
where we recall that $E(t)$ is the energy associated with the autonomous problem given by \eqref{energy}.
Integrating \eqref{intermedia1} between $0$ and $t$ yields
$$
\sqrt{E(t)} \leq \sqrt{E(0)} + \frac{1}{\sqrt{2}} \int_0^t \Vert g(s) \Vert_{L^2(0 ,\pi)} \, ds \leq \frac{M}{\sqrt{2}},
$$
which proves the statement.
\end{proof}
Consequently, under the assumptions of Lemma \ref{stimaderivate}, by \eq{stimauinfty} we deduce that strong solutions $u$ of \eqref{nonlinearforced} satisfy the
estimate
$$
\Vert u(t) \Vert_{L^\infty(0,\pi)} \leq \Vert u_{xx}(t) \Vert_{L^2(0,\pi)} \leq M \qquad\mbox{for }0\le t\le T\, ;
$$
this will be needed to estimate the Lipschitz constant of $f$ on the solutions. 

We now observe that, by Lemma \ref{stimaderivate}, it holds that
$$
\Vert P_Nu_t(t) \Vert^2_{L^2(0, \pi)} +\Vert Q_Nu_t(t) \Vert^2_{L^2(0, \pi)} + \Vert P_Nu_{xx}(t) \Vert^2_{L^2(0, \pi)} + \Vert Q_Nu_{xx}(t) \Vert^2_{L^2(0, \pi)}
+ 2 \int_0^\pi F(u(t)) \, dx \leq M^2
$$
from which, thanks to \eqref{stimau2} and the fact that $F \geq 0$, we deduce the inequality
\begin{equation}\label{stimaproiez}
(N+1)^4\, \Vert Q_Nu(t) \Vert^2_{L^2(0, \pi)}\le M^2- \Vert P_Nu_{xx}(t) \Vert^2_{L^2(0, \pi)}-\Vert P_Nu_t(t) \Vert^2_{L^2(0, \pi)}\le M^2. 
\end{equation}
We are now ready to prove Theorem \ref{approximation}.
\begin{proof}[Proof of Theorem \ref{approximation}] Let us set $v=P_N u - u^N$: since
$$
v_{tt} + v_{xxxx} + P_N f(u) - P_N f(u^N) = P_N g - P_N g = 0,
$$
multiplying both sides by $v_t$ and integrating on $[0, \pi]$, we obtain 
\begin{equation}\label{laprima}
\frac{d}{dt} (\Vert v_t \Vert^2_{L^2(0, \pi)} + \Vert v_{xx}\Vert^2_{L^2(0, \pi)}) + 2\int_0^\pi \big(P_N f(u) - P_N f(u^N)\big)v_t = 0,
\end{equation}
where we omitted the integration variable for brevity.
Since $Q_N v = 0$, we can drop $P_N$ and, by the Lagrange Theorem, we get
$$
\int_0^\pi \big(P_N f(u) - P_N f(u^N)\big)v_t  = \int_0^\pi \big(f(u)-f(u^N)\big)v_t = \int_0^\pi f'(\xi) (u-u^N) v_t,
$$
where $\xi=\xi(x, t)$ is between $u(x, t)$ and $u^N(x, t)$. From \eq{derivata} we deduce that
$$
\int_0^\pi \Big(f'(\xi)-\frac{L(M)}{2}\Big) (u-u^N) v_t \leq \frac{L(M)}{2} \int_0^\pi \vert u-u^N \vert \vert v_t \vert,
$$
from which, in view of \eqref{laprima},
$$
\frac{d}{dt} (\Vert v_t \Vert^2_{L^2(0, \pi)} + \Vert v_{xx}\Vert^2_{L^2(0, \pi)}) + L(M) \int_0^\pi (u-P_N u + P_N u - u^N) v_t \leq  L(M) \int_0^\pi \vert u-u^N \vert \vert v_t \vert.
$$
We thus infer 
$$
\frac{d}{dt} \Big(\Vert v_t \Vert^2_{L^2(0, \pi)} + \Vert v_{xx}\Vert^2_{L^2(0, \pi)} + \frac{L(M)}{2} \Vert v \Vert^2_{L^2(0, \pi)}\Big) \leq 2L(M) \Vert Q_N u \Vert_{L^2(0, \pi)} \Vert v_t \Vert_{L^2(0, \pi)} +  L(M) \int_0^\pi \vert v \vert \vert v_t \vert \
$$
$$
 \leq 2 L(M) \Vert Q_N u \Vert_{L^2(0, \pi)} \Vert v_t \Vert_{L^2(0, \pi)} +  \frac{L(M)}{2}\Big[\sqrt{\tfrac{L(M)+2}{2}} \Vert v \Vert_{L^2(0, \pi)}^2 + \sqrt{\tfrac{2}{L(M)+2}} \Vert v_t \Vert_{L^2(0, \pi)}^2 \Big]
$$
$$
\leq \frac{2L(M)M}{(N+1)^2} \Vert v_t \Vert_{L^2(0, \pi)} + \frac{L(M)}{\sqrt{2(L(M)+2)}}\Big[\Vert v_t \Vert_{L^2(0, \pi)}^2 + \Vert v_{xx} \Vert_{L^2(0, \pi)}^2 + \frac{L(M)}{2} \Vert v \Vert_{L^2(0, \pi)}^2\Big],
$$
where we have used, respectively, the Young and the Poincar\'e inequality and \eqref{stimaproiez}. Set for simplicity
$$
Y(t)= \Vert v_t \Vert^2_{L^2(0, \pi)} + \Vert v_{xx}\Vert^2_{L^2(0, \pi)} + \frac{L(M)}{2} \Vert v \Vert^2_{L^2(0, \pi)},
$$
and recall that $C=\tfrac{L(M)}{\sqrt{2(L(M)+2)}}$.
Then, $Y$ satisfies the following Bernoulli differential inequality
$$
\dot{Y}(t) \leq \frac{2L(M)M}{(N+1)^2} \sqrt{Y(t)} + C Y(t)
$$
so that, after integration, we obtain
$$
Y(t) \leq \frac{8M^2(L(M)+2)}{(N+1)^4} (\textnormal{e}^{Ct/2}-1)^2.
$$
Writing $v(x, t)=\sum_{n=1}^{N} v_n(t) \sin (nx)$, we notice that
\begin{equation}\label{stimaY}
Y(t)=\frac{\pi}{2} \sum_{n=1}^{N} \Big[(\dot{v}_n)^2 + n^4 v_n^2 + \frac{L(M)}{2} v_n^2\Big].
\end{equation}
Since $v_n(0)=0$, we have $\Vert v_n \Vert_{L^\infty(0, T)} \leq T \Vert \dot{v}_n \Vert_{L^\infty(0, T)}$, implying then that
$$
Y(T) \geq \frac{\pi}{2} \Big(n^4 + \frac{L(M)}{2} + \frac{1}{T^2} \Big) \Vert v_n \Vert_{L^\infty(0, T)}^2.
$$
Since $v_n(t)=\varphi_n(t)- \varphi_n^N(t)$, this proves \eq{stimamodo} and completes the proof.
\end{proof}

We remark that, in view of \eqref{stimaY}, also the $C^1$-convergence of $\varphi_n^N$ to $\varphi_n$ can be quantitatively controlled with an estimate similar to \eqref{stimamodo}.

\section{Appendix}

We state here the result that we have used in Section \ref{casecub} for the computation of some integral terms when dealing with the nonlinearity $f(u)=u^3$. The proof follows from the
prostaphaeresis formulas.
\begin{lemma}\label{calculus}
For all $p\in\N$ we have
$$\frac{8}{\pi}\int_0^\pi\sin^4(px)\, dx=3\, .$$
For all $p,q\in\N$ ($q\neq p$) we have
$$
\frac{8}{\pi}\int_0^\pi\sin^2(px)\sin^2(qx)\, dx=2\, .
$$
For all $p,q\in\N$ ($q\neq p$) we have
$$
\frac{8}{\pi}\int_0^\pi\sin^3(px)\sin(qx)\, dx=\left\{\begin{array}{ll}-1\ &\mbox{if }q=3p\\
0\ &\mbox{if }q\neq3p\, .
\end{array}\right.
$$
For all $p,q,r\in\N$ (all different and $q<r$) we have
$$
\frac{8}{\pi}\int_0^\pi\sin^2(px)\sin(qx)\sin(rx)\, dx=\left\{\begin{array}{ll}1\ &\mbox{if }r+q=2p\\
-1\ &\mbox{if }r-q=2p\\
0\ &\mbox{if }r\pm q\neq2p\, .
\end{array}\right.
$$
For all $p,q,r,s\in\N$ (all different and $p<q$, $r<s$) we have
$$
\frac{8}{\pi}\int_0^\pi\sin(px)\sin(qx)\sin(rx)\sin(sx)\, dx=\left\{\begin{array}{ll}1\ &\mbox{if }q\pm p=s\pm r\\
-1\ &\mbox{if }q\pm p=s\mp r\\
0\ &\mbox{otherwise.}
\end{array}\right.
$$
\end{lemma}

\section{Conclusions}

In this paper, we have introduced a suitable notion of instability (Definition \ref{unstable}) for the nonlinear beam equation
$$
u_{tt}+u_{xxxx}+ f(u)= g(x, t) \quad \mbox{for }(x,t)\in(0,\pi)\times(0,\infty),
$$
aiming at detecting a nonlinear and  ``exponentially-growing'' behaviour of some Fourier component which may possibly lead to structural failures. We have focused on the situation where either $u_0$ and $u_1$ are concentrated on a single mode $j$, or $g$ containes only the $j$-th Fourier component. In this case, we call the $j$-th mode \emph{prevailing}. 
In line with the experimental observations, we detect instability each time that a residual (i.e., non-prevailing) Fourier component starts small and suddenly grows by an order of magnitude, both with respect to itself and with respect to the prevailing mode, see \eqref{grande}.
\par
We have then proved an approximation result which allows to reduce the study of this kind of instability for the beam equation to the analysis of its Galerkin finite-dimensional approximation. The finite-dimensional system has been investigated with three different nonlinearities which appear quite naturally in literature: $f(u)=u^+, f(u)=u^3$, and $f(u)=(u^+)^3$. It turns out that the nonlinearity creates an energy transfer between modes, but the observed behaviour is deeply influenced by the form of $f$: no instability seems to occur in the first case, while for the other nonlinearities we observe the possible appearance of instability above some critical amplitude thresholds of the prevailing mode.
We have also highlighted some structural differences between the three models, for instance concerning the mutual exchange of energy between even and odd modes. In fact, much more nonlinearities should be considered in order to have a qualitative feeling of what are their relevant features for instability: as an example, see \cite{Boe1, Boe2, Boe3}. 
The role of boundary conditions other than hinged (see, e.g., \cite{Swe}) should be investigated, as well.  
\par
Overall, this paper showed that the stability analysis for a particular class of solutions (those with a prevailing mode) of a nonlinear PDE may be reduced to a finite-dimensional ODE analysis. Our definition of stability is coherent with real structures, see \cite{ammann}, and we expect our results to be applicable to prevent structural instabilities possibly leading to collapses.
\par
One of the next steps should be to study the interaction between large initial data (below the instability threshold) on a certain mode $j$ and an external forcing term $g$
concentrated on the same mode, to see if instability may arise also thanks to their combined effect. \par
Moreover, in order to deepen the understanding of the instability of suspension bridges we have to take into account more complete beam models like string-beam systems \cite{bgv,bgv2,drabek3,HolMat,JacMcK}; see also \cite{bgvsurvey,drabek2} for instructive surveys. A further step will be represented by the study of plate models, possibly involving experimental parameters. At this level, the theoretical justifications of the results will probably require some computer assisted proofs.

\par\medskip\noindent
\textbf{Acknowledgements.} The second author is partially supported by the PRIN project {\em Equazioni alle derivate parziali di tipo ellittico
e parabolico: aspetti geometrici, disuguaglianze collegate, e applicazioni}. Both authors are members of the Gruppo Nazionale per l'Analisi Matematica, la Probabilit\`a
e le loro Applicazioni (GNAMPA) of the Istituto Nazionale di Alta Matematica (INdAM).
\\
The authors are grateful to the anonymous referees, whose valuable comments allowed to considerably improve the paper and its readability.

\bigskip
\bigskip
\noindent
M. Garrione - \emph{Dipartimento di Matematica e Applicazioni, Universit\`a di Milano-Bicocca, Via Cozzi 55, 20126 Milano, Italy} \\
F. Gazzola - \emph{Dipartimento di Matematica, Politecnico di Milano, Piazza Leonardo da Vinci 32, 20133 Milano, Italy}

\end{document}